\documentclass[times,sort]{elsarticle}
\journal{Journal of Computational Physics}

\usepackage{yfonts}

\usepackage[margin=2.54cm]{geometry}
\usepackage{graphicx}
\usepackage{amssymb}
\usepackage{amsmath}
\usepackage[ruled,vlined]{algorithm2e}
\DontPrintSemicolon
\SetAlFnt{\small}

\usepackage{mathrsfs}
\usepackage{epstopdf}
\usepackage{stmaryrd}
\usepackage{subfig}
\usepackage{amsthm}
\usepackage{lmodern}
\usepackage{enumitem}
\usepackage{amsfonts}
\usepackage{mathtools}
\usepackage{booktabs}
\usepackage{xcolor}
\usepackage{stackrel}
\usepackage{accents}

\usepackage{placeins} 

\theoremstyle{plain}
\newtheorem{thm}{Theorem}
\newtheorem{Lem}{Lemma}

\theoremstyle{remark}
\newtheorem{rem}{Remark}

\allowdisplaybreaks

\usepackage{placeins}

\usepackage{pgfplotstable,booktabs}
\usepackage{pgfplots}
\usetikzlibrary{patterns}
\usepackage[formats]{listings}
\lstdefineformat{C}
{
	\{=\newline\string\newline\indent,%
	\}=\newline\noindent\string\newline,%
	;=[\ ]\string\space,%
}
\usepackage{filecontents}

\numberwithin{equation}{section}

\newcommand{\pderivative}[2]{\frac{\partial #1}{\partial #2}}

\newcommand{\dmat}{\mat{D}}
\newcommand{\Dmod}{\tilde{\mat{D}}}

\newcommand{\dmathat}{\mat{\hat{D}}}

\newcommand{\mmat}{\mat{{M}}}

\newcommand{\smat}{\mat{{S}}}
\newcommand\statevec[1]{\vec{#1}}

\newcommand{\jump}[1]{\left\llbracket#1\right\rrbracket}
\newcommand{\average}[1]{\left\{\!\left\{ #1\right\}\!\right\}}

\newcommand\Ftilde{\tilde{F}}
\newcommand\Gtilde{\tilde{G}}

\newcommand\Ftildeavg{\Ftilde^{\#}}
\newcommand\Gtildeavg{\Gtilde^{\#}}

\newcommand{\mbgeq}{\overset{!}{\geq}}
\newcommand{\mbleq}{\overset{!}{\leq}}

\newcommand{\half}{\frac{1}{2}}
\newcommand{\fourth}{\frac{1}{4}}
\newcommand{\mat}[1]{\mathbf{#1}}

\newcommand\tensor[1]{\overset{\text{$\leftrightarrow$}}{#1}}

\newcommand{\Dhat}{\hat{{D}}}

\definecolor{redcol}{rgb}{1.,0.,0.0}
\newcommand{\dE}{\,\mathrm{dE}}
\newcommand{\dS}{\,\mathrm{dS}}

\newsavebox{\VZerokernel}
\newsavebox{\FluxBox}
\newsavebox{\VOptkernelOne}
\newsavebox{\VOptkernelTwo}
\newsavebox{\FluxOptBox}
\begin{document}

\begin{frontmatter}

\title{An entropy stable discontinuous Galerkin method for the shallow water equations on curvilinear meshes with wet/dry fronts accelerated by GPUs}
\author[mathematik]{Niklas Wintermeyer\corref{correspondingauthor}}
\cortext[correspondingauthor]{Corresponding author}
\ead{nwinterm@math.uni-koeln.de}
\author[mathematik]{Andrew R.~Winters}
\author[mathematik]{Gregor J.~Gassner}
\author[vtech]{Timothy Warburton}

\address[mathematik]{Mathematisches Institut, Universit\"at zu K\"oln, Weyertal 86-90, 50931 K\"oln}
\address[vtech]{Department of Mathematics, Virginia Tech, 225 Stanger Street, Blacksburg, VA 24061-0123}

\numberwithin{equation}{section}


\begin{abstract}
We extend the entropy stable high order nodal discontinuous Galerkin spectral element approximation for the non-linear two dimensional shallow water equations presented by Wintermeyer et al. [\textit{N. Wintermeyer, A. R. Winters, G. J. Gassner, and D. A. Kopriva. An entropy stable
nodal discontinuous Galerkin method for the two dimensional shallow water equations on unstructured
curvilinear meshes with discontinuous bathymetry. Journal of Computational Physics, 340:200-242,
2017}] with a shock capturing technique and a positivity preservation capability to handle dry areas. The scheme preserves the entropy inequality, is well-balanced and works on unstructured, possibly curved, quadrilateral meshes. For the shock capturing, we introduce an artificial viscosity to the equations and prove that the numerical scheme remains entropy stable. We add a positivity preserving limiter to guarantee non-negative water heights as long as the mean water height is non-negative. We prove that non-negative mean water heights are guaranteed under a certain additional time step restriction for the entropy stable numerical interface flux. We implement the method on GPU architectures using the abstract language OCCA, a unified approach to multi-threading languages. We show that the entropy stable scheme is well suited to GPUs as the necessary extra calculations do not negatively impact the runtime up to reasonably high polynomial degrees (around $N=7$). We provide numerical examples that challenge the shock capturing and positivity properties of our scheme to verify our theoretical findings.
\end{abstract}

\end{frontmatter}

\section{Introduction}

The shallow water equations including a non-constant bottom topography are a system of hyperbolic balance laws
\begin{equation}
\begin{aligned}\label{sw-invisc}
 h_t + (hu)_x + (hv)_y &= 0 ,\\
(hu)_t + \left(h u^2+ \half g\,h^2\right)_{\!\!x} + (huv)_y &= -g h b_x ,\\
(hv)_t + (huv)_x + \left(h v^2+\half g\,h^2\right)_{\!\!y} &= -g h b_y ,\\
\end{aligned}
\end{equation}
that are useful to model fluid flows in lakes, rivers, oceans or near coastlines, e.g \cite{bonev2017discontinuous,whitham1974,Marras_GalerkinViscSW}. We compactly write the system \eqref{sw-invisc} as
\begin{equation}
\begin{aligned}\label{sw-invisc-compact}
 \vec{w}_t + \nabla \cdot (\vec{f},\vec{g})^T &=   \vec{\mathcal{S}} ,\\
\end{aligned}
\end{equation}
with $\vec{w} = (h,hu,hv)^T$, $\vec{f} = (hu, hu^2 + \half gh^2, huv)^T$ and $\vec{g} = (hv, huv, hv^2+\half gh^2)^T$ and source term $ \vec{\mathcal{S}} = (0, -gh b_x, -gh b_y)^T$. The water height is denoted by $h=h(x,y,t)$ and is measured from the bottom topography $b=b(x,y)$. The total water height is therefore $H=h+b$. The fluid velocities are $u=u(x,y,t)$ and $v=v(x,y,t)$. An important steady state solution of \eqref{sw-invisc} is the preservation of a flat lake with no velocity, the so-called ``lake at rest'' condition
\begin{equation}
\label{eq:lakeAtRest}
\begin{aligned}
&h+b=\textrm{const},\\
&u=v=0.\\
\end{aligned}
\end{equation}
A numerical scheme that preserves non-trivial steady state solutions, such as the ``lake at rest'' problem, is \textit{well-balanced}, e.g. \cite{noelle2007,gassner2016well}. Methods that are not well-balanced can produce spurious waves in the magnitude of the mesh size truncation error which pollutes the solution quality. This is particularly problematic because many interesting shallow water phenomena can be interpreted as perturbations from the lake at rest condition \cite{leveque1998}.

Due to the non-linear nature of the shallow water equations \eqref{sw-invisc}, discontinuous solutions may develop regardless of the smoothness of the initial conditions. Therefore, solutions to the PDEs \eqref{sw-invisc} are sought in the weak sense. Unfortunately, weak solutions are non-unique and additional admissibility criteria are required to extract the physically relevant solution from the family of weak solutions. One important criteria for physically relevant solutions is the second law of thermodynamics, which guarantees that the entropy of a physical system increases as the fluid evolves. In mathematics, a suitable strongly convex entropy function can be used to ensure a numerical approximation obeys the laws of thermodynamics discretely \cite{tadmor2003}. A numerical scheme that satisfies the second law of thermodynamics is said to be \textit{entropy stable}. Physically, entropy must be dissipated in the presence of discontinuities. In order to discuss the mathematical entropy for the shallow water equations a suitable entropy function is the total energy $e=e(\vec w)$
\begin{equation}
    \label{eq:entropyFunction}
    e := \half h(u^2+v^2) + \half g h^2 + ghb.
\end{equation}
We take the derivative of the entropy function with respect to the conservative variables $\vec w$ to find the set of entropy variables $\vec q = \pderivative{e}{\vec w}$, which are
\begin{equation}
\label{eq:entropyVariables}
    q_1 = gH - \half (u^2 +v^2), \quad\quad q_2 = u, \quad\quad q_3 = v.
\end{equation}
If we contract the shallow water equations \eqref{sw-invisc} from the left with the entropy variables and apply consistency conditions on the fluxes developed by Tadmor \cite{tadmor1984} we obtain the entropy conservation law
\begin{equation}
\label{eq:entropyEquality}
    e_t + \mathcal F_x + \mathcal G_y = 0,
\end{equation}
with the entropy fluxes $\mathcal{F} = \half hu ( u^2+v^2) + g hu (h+b)$ and $\mathcal{G} = \half hv ( u^2+v^2) + g hv (h+b)$. In the presence of discontinuities \eqref{eq:entropyEquality} becomes the entropy inequality
\begin{equation}
\label{eq:entropyInequality}
    e_t + \mathcal F_x + \mathcal G_y \leq 0.
\end{equation}

 Unfortunately, for high-order numerical methods the discrete satisfaction of \eqref{eq:entropyInequality} does not guarantee that the approximation solution will be overshoot free, e.g. \cite{ESDGSEM2D_paper}. Therefore, the first contribution of this work is to add a shock capturing method that maintains the entropy stability of the nodal discontinuous Galerkin method (DGSEM) on curvilinear quadrilateral meshes developed by Wintermeyer et al. \cite{ESDGSEM2D_paper}. In particular, artificial viscosity is added into the two momentum equations. The amount of artificial viscosity is selected with the method developed by Persson and Peraire \cite{persson2006}. Even with shock capturing there can still be issues maintaining the positivity of the water height, $h$, particularly in flow regions where $h\rightarrow 0$. Thus, our second contribution is to incorporate the positivity preserving limiter of Xing et al. \cite{xing2010positivity} in an entropy stable way. To fulfill the requirements of the positivity limiter, we formally show that the entropy stable numerical fluxes of the entropy stable discontinuous Galerkin spectral element method (ESDGSEM) preserve positive mean water heights on two-dimensional curved meshes. We then generalize a result from Ranocha \cite{ranocha2017} to show that the posivitity preserving limiter itself is entropy stable on curvilinear meshes.
 
 Our third, and final, contribution is to implement the two-dimensional positive ESDGSEM on GPUs. The entropy stable approximation is built with specific split forms, which are linear combinations of the conservative and advective forms of the shallow water equations \cite{gassner2016well,ESDGSEM2D_paper}. However, the method remains fully conservative \cite{fisher2013} albeit with additional computational complexity in the form of an increased number of arithmetic operations, but without the need for more data storage and transfer. Therefore, the ESDGSEM seems a perfect candidate for implementation on GPUs. We demonstrate that this expectation is true through careful analysis and detailed discussion of how to implement the numerical method into a unified approach for multi-threading languages OCCA \cite{medina2014occa}.

This work is organized as follows: In Sec. \ref{Sec:DG} we briefly provide background details of the ESDGSEM described in \cite{ESDGSEM2D_paper} that serves as the baseline scheme. Next, shock capturing by way of artificial viscosity is discussed in Sec. \ref{Sec:artVisc}. A positivity preserving limiter is presented in Sec. \ref{sec:wetDry} that maintains the entropy stability of the DGSEM. Implementation details and analysis of the ESDGSEM compared to the standard DGSEM on GPUs is given in Sec. \ref{sec:GPU}. Numerical results that exercise and test the capabilities of the positive ESDGSEM are provided in Sec. \ref{sec:num}. Concluding remarks are given in the final section.

\section{Entropy Stable Discontinuous Galerkin Spectral Element Method}
\label{Sec:DG}
We briefly present the ESDGSEM developed in \cite{ESDGSEM2D_paper,gassner2016well,winters2015comparison} which we will extend in the following sections. The entropy stable scheme is based on a nodal DG approximation to the split-form shallow water equations on curved quadrilateral meshes. We start with a brief description of the curvilinear transformations and numerical approximations before summarizing the main properties of the ESDGSEM in Lemma \ref{Lem:ESDGSEM}. We refer to \cite{hesthaven_warburton,koprivabook} for a more thorough derivation of DG methods and to \cite{carpenter_esdg,gassner2016,ESDGSEM2D_paper} for a background on the entropy stable split form DG framework.
\subsection{Curvilinear mappings}
We separate the domain $\Omega$ into $K$ non-overlapping elements $E_k \subset \Omega$ and proceed by constructing a mapping on each element between the computational reference element $E=[-1,1]^2$ with coordinates $(\xi ,\eta)$ and the physical coordinates $(x,y)$. 
We use a transfinite interpolation with linear blending \cite{koprivabook}, where the Jacobian is computed by 
\begin{equation}
    \label{JacobianDef}
\mathcal J=x_{\xi}\,y_{\eta} - x_{\eta}\,y_{\xi} .
\end{equation}
The gradients in physical space $\nabla = \left(\pderivative{}{x}, \pderivative{}{y}\right)^T$ and computational space $\hat\nabla = \left(\pderivative{}{\xi}, \pderivative{}{\eta}\right)^T$ are related by the chain rule
\begin{equation}
\label{NablaRelation}
\begin{aligned}
\nabla	 = 
\frac{1}{\mathcal J}\left(
		\begin{array}{cc}
			y_\eta & -y_\xi      \\
			-x_\eta&  x_\xi    \\
		\end{array}
	\right) \hat \nabla.									
\end{aligned}
\end{equation}
With a sufficiently smooth mapping we can use \eqref{NablaRelation} to replace the physical $x$ and $y$ derivatives in \eqref{sw-invisc-compact} to get the PDE in reference space
\begin{equation}
\label{sw-invisc-refspace}
 \mathcal J \vec{w}_t + \hat\nabla \cdot (\vec{\tilde{f}},\vec{\tilde{g}})^T =   \vec{\tilde{\mathcal{S}}} ,
\end{equation}
where we introduce the contravariant fluxes defined by
\begin{equation}
\label{ContravariantFluxes}
\begin{aligned}
\vec{\tilde{f}}(\statevec{w})&=y_{\eta}\,\vec{f}(\statevec{w}) - x_{\eta}\,\vec{g}(\statevec{w}),\\
\vec{\tilde{g}}(\statevec{w})&=-y_{\xi}\,\vec{f}(\statevec{w}) + x_{\xi}\,\vec{g}(\statevec{w}),
\end{aligned}
\end{equation}
and the gradient of the bottom topography transforms by \eqref{NablaRelation} to create the contravariant source term $\vec{\tilde{\mathcal{S}}}$.

To find the weak form of the balance law in reference space \eqref{sw-invisc-refspace}, we multiply by a smooth test function $\phi $ and integrate over the reference element. Then we use integration by parts to move the differentiation of the fluxes, $\hat\nabla \cdot (\vec{\tilde{f}},\vec{\tilde{g}})^T$, onto the test function. We replace the fluxes across element interfaces by numerical surface fluxes $\tilde{F}^*$ and $\tilde{G}^*$ and obtain the weak form
\begin{equation}\label{Eq:ContinuousWeakForm}
\int\limits_E  \mathcal J\,w \,  \phi \dE - \oint\limits_{\partial E} \phi \, \left(\tilde{F}^* , \tilde{G}^*\right)\, \cdot \, \vec{n}\dS - \int\limits_E \,\left(\tilde{f}, \tilde{g}\right) \cdot \hat\nabla  \phi \dE =  \int\limits_E \,\vec{\tilde{\mathcal{S}}} \phi \dE.
\end{equation}
Integrating by parts once more yields the strong form 
\begin{equation}\label{Eq:ContinuousStrongForm}
\int\limits_E  \mathcal J\,w \,  \phi  \dE 
- \oint\limits_{\partial E} \phi \, \left(\tilde{F}^*  -\tilde{f}, \tilde{G}^* -\tilde{g}\right)\, \cdot \, \vec{n}\dS 
+ \int\limits_E \, \hat\nabla \cdot \left(\tilde{f}, \tilde{g}\right)^T \phi \dE =  \int\limits_E \, \vec{\tilde{\mathcal{S}}} \phi \dE.
\end{equation}

\subsection{Numerical Approximations}
We approximate quantities of interest such as variables $\vec{w}$ and fluxes $\vec{f}$, $\vec{g}$ by polynomials of degree $N$ and denote them by capital letters $\vec W$ or $\vec F$ and $\vec G$. We use a nodal form of the interpolation with nodes defined at the Legendre-Gauss-Lobatto (LGL) points $\{\xi_i\}_{i=0}^N$ and $\{\eta_j\}_{j=0}^N$ in the reference square.
We write the element-wise polynomial approximation (e.g for a component $W$ of $\statevec{W}$) as
\begin{equation}
\label{poly_approx}
w(x,y,t)\big|_{G}=w(x(\xi,\eta),y(\xi,\eta),t)\approx W(\xi,\eta,t) :=\sum\limits_{i=0}^N\sum\limits_{j=0}^N W_{i,j}(t)\,\ell_i(\xi)\,\ell_j(\eta),
\end{equation}
where $\{W_{i,j}(t)\}_{i,j=0}^{N,N}$ are the time dependent nodal degrees of freedom and the one-dimensional Lagrange basis functions for the interpolant are
\begin{equation}
\label{lagrange_basis}
\ell_j(\xi)=\prod\limits_{i=0,i\neq j}^N\frac{\xi - \xi_i}{\xi_j-\xi_i},\qquad j=0,\ldots,N.
\end{equation} 
Derivatives are approximated element-wise directly from the derivative of the polynomial approximation, e.g.,  
 \begin{equation}
 \frac{\partial }{\partial \xi}W(\xi,\eta,t) = \sum\limits_{i=0}^N\sum\limits_{j=0}^N W_{i,j}(t)\,\frac{\partial }{\partial \xi}\ell_i(\xi)\,\ell_j(\eta).
 \end{equation}
We introduce the polynomial derivative operator $\dmat$ with entries
 \begin{equation}
 \label{DmatDef}
 D_{ij}:=\frac{\partial\ell_j}{\partial\xi}\Bigg|_{\xi=\xi_i},\qquad i,j=0,\ldots,N,
 \end{equation}
 which is used to calculate the derivative with respect to $\xi$ at the interpolation nodes. It follows from the tensor product ansatz that the same derivative operator $\dmat$ can be used to evaluate derivatives in $\xi$ and $\eta$ direction.
 We also approximate integrals numerically using LGL quadrature and define the mass matrix $\mat{M} = \text{diag}(\omega_0,\omega_1,\ldots,\omega_N)$ with LGL quadrature weights on the diagonal. By the tensor product ansatz the two-dimensional quadrature is
  \begin{equation}
 \int\limits_{-1}^{1} \int\limits_{-1}^{1} w(\xi,\eta) \, \mathrm{d}\xi \mathrm{d}\eta \approx \sum_{i=0}^N\sum_{j=0}^N w(\xi_i,\eta_j) \omega_i \omega_j.
  \end{equation}
 We choose the test function $\phi$ in \eqref{Eq:ContinuousWeakForm}-\eqref{Eq:ContinuousStrongForm} to be a polynomial in the reference element $E$ 
 \begin{equation}
 \phi^E = \sum_{i=0}^N\sum_{j=0}^N \phi^E_{i,j} \ell_i (\xi) \ell_j (\eta).
 \end{equation}
 This choice of test functions and the collocation of the interpolation and quadrature nodes enables us to use the Kronecker delta property of the Lagrange interpolating polynomials to greatly simplify the integrals in \eqref{Eq:ContinuousWeakForm}. For example, the integral on the time derivative term becomes
 \begin{equation}\label{Eq:CollocationExample}
\int\limits_E \,\mathcal{J} W_t \ell_i (\xi) \ell_j (\eta)\dE 
\approx  \mathcal{J}_{ij} ({W_t})_{ij} \omega_i\omega_j.
\end{equation}
A significant advantage of the LGL quadrature nodes is that the corresponding derivative operator $\mat{D}$ \eqref{DmatDef} and the mass matrix $\mat{M}$ satisfy the summation-by-parts (SBP) property,
 \begin{equation}
 \label{SBP_def}
\mmat \dmat + \left(\mmat \dmat\right)^T = \text{diag}(-1,0,\ldots,0,1) ,
\end{equation}
for all polynomial orders \cite{gassner_skew_burgers}. The SBP property enables us to make a connection between LGL-based discontinuous Galerkin methods and sub-cell finite volume type differencing methods as proposed by Fisher and Carpenter \cite{fisher2013}. In \cite{ESDGSEM2D_paper} we used this relation to develop an entropy stable discontinuous Galerkin spectral element method for the shallow water equations on curvilinear geometries.
The SBP property \eqref{SBP_def} also implies $\dmat=-\smat+\hat{\dmat}$
with surface matrix $\mat{S} := \textrm{diag}\left(\frac{1}{\omega_0},0,\ldots,0,-\frac{1}{\omega_N}\right)$ and $\dmathat := - \mmat^{-1} \dmat \mmat$. It also follows that strong and weak form discretizations are equivalent \cite{KoprivaGassner_GaussLob}. We rewrite the derivative operator $\dmat$ in the computations to incorporate the surface parts and a factor of $2$ for the interior that stems from the split form approach \cite{gassner2016}
\begin{equation}
\label{eq:Dmat_modification}
\Dmod := 2 \dmat + \smat .
\end{equation}
The full discretization as well as a summary of the main properties of the ESDGSEM can be found in Lemma \ref{Lem:ESDGSEM}. We refer to \cite{ESDGSEM2D_paper} for the detailed proofs and derivations. As a notational convenience we introduce jump $\jump{\cdot}$ and average $\average{\cdot}$ operators which are defined by
\begin{equation}
\begin{aligned}
    \label{eq:jumpavgDef}
    \jump{W} &:= W^+ - W^-, \\
    \average{W} &:= \half \left( W^+ + W^-\right).
\end{aligned}
\end{equation}
We note that jumps only occur on element interfaces and have an orientation. The ``$+$'' and ``$-$'' states here are strictly related to the normal vector on the interface and the normal is always pointing outward from the ``$-$'' element and into the ``$+$'' element. The averages are also used for the evaluation of the interior two point fluxes and have no orientation, but require special compact notation. For example, in the $\xi-$direction, the interior arithmetic mean is
\begin{equation}
\average{\cdot}_{(i,m),j} = \frac{1}{2}\left(\left(\cdot\right)_{ij}+\left(\cdot\right)_{mj}\right).
\end{equation}
\begin{Lem}[ESDGSEM]\label{Lem:ESDGSEM}
The semi-discrete split DG approximation to the two dimensional shallow water equations on curvilinear grids
\begin{equation}\label{Eq:CurvilinearECDGSEM}
J \vec{W}_{t} + \vec{\mathcal{L}}_{\xi}+ \vec{\mathcal{L}}_{\eta} = \vec{\mathcal{S}} 
\end{equation}
with
\begin{equation}\label{eq:fluxDiffFormF}
\begin{aligned}
\left(\vec{\mathcal{L}}_{\xi}\right)_{ij} &=\frac{1}{\omega_i} \left(\delta_{iN} \left[\vec{\tilde{F}}^{*,es}\right]_{Nj} - \delta_{i0}\left[\vec{\tilde{F}}^{*,es}\right]_{0j}\right)  
+\sum\limits_{m=0}^{N} \tilde{D}_{i m} \vec{\tilde{F}}_{(i,m),j},
\\
\left(\vec{\mathcal{L}}_{\eta}\right)_{ij} &= \frac{1}{\omega_j} \left(\delta_{Nj} \left[\vec{\tilde{G}}^{*,es}\right]_{iN} - \delta_{0j}\left[\vec{\tilde{G}}^{*,es}\right]_{i0}\right) 
+\sum\limits_{m=0}^{N}\,\tilde{D}_{m j}\vec{\tilde{G}}_{i,(m,j)},
\end{aligned}
\end{equation}
with curvilinear volume fluxes
\begin{equation}
\begin{aligned}
\label{Fmij}
\vec{\tilde{F}}_{(i,m),j} &:= \vec{F}^{\#}(\statevec{W}_{i,j},\statevec{W}_{m,j})\average{y_\eta}_{(i,m),j}-\vec{G}^{\#}(\statevec{W}_{i,j},\statevec{W}_{m,j})\average{x_\eta}_{(i,m),j}, \\
\vec{\tilde{G}}_{i,(m,j)} &:= -\vec{F}^{\#}(\statevec{W}_{i,j},\statevec{W}_{i,m})\average{x_\xi}_{i,(j,m)}+\vec{G}^{\#}(\statevec{W}_{i,j},\statevec{W}_{i,m})\average{y_\xi}_{i,(j,m)},
\end{aligned}
\end{equation}
where the entropy conserving two-point volume fluxes are defined by
\begin{equation}\label{eq:volumeFluxes}
\begin{aligned}
\vec{F}^{\#}(\statevec{W}_{i,j},\statevec{W}_{m,j}) := \begin{pmatrix}
\average{hu}_{(i,m),j}\\[0.1cm]
\average{hu}_{(i,m),j}\average{u}_{(i,m),j} + \,g\, \average{h}_{(i,m),j}^2 - \half g \average{h^2}_{(i,m),j} \\[0.1cm]
\average{hu}_{(i,m),j}\average{v}_{(i,m),j}
\end{pmatrix},\\[0.2cm]
\vec{G}^{\#}(\statevec{W}_{i,j},\statevec{W}_{i,m}) := \begin{pmatrix}
\average{hv}_{i,(m,j)}\\[0.1cm]
\average{hv}_{i,(m,j)}\average{u}_{i,(m,j)}\\[0.1cm]
\average{hv}_{i,(m,j)}\average{v}_{i,(m,j)}  + \,g\, \average{h}_{i,(m,j)}^2 - \half g \average{h^2}_{i,(m,j)}
\end{pmatrix}.
\end{aligned}
\end{equation}
The numerical interface flux in normal direction is given by
\begin{equation}
\label{NumericalSurfaceFlux}
{\tilde{F}}^{*,es} =n_x \vec{F}^{*,es} +  n_y \vec{G}^{*,es} ,
\end{equation}
and is a combination of the fluxes in $x$ and $y$ direction
\begin{equation}
\begin{aligned}
\label{eq:fesandges}
\vec{F}^{*,es} &=\vec{F}^{*,ec} - \half \mat{R}_f  \, \big|\mat{\Lambda}_f \big| \, \mat{R}_f^T \jump{\,\statevec{q}\,},\\
\vec{G}^{*,es} &=\vec{G}^{*,ec} - \half \mat{R}_g  \, \big|\mat{\Lambda}_g \big| \, \mat{R}_g^T \jump{\,\statevec{q}\,},\\
\end{aligned}
\end{equation}
which include the entropy conserving fluxes $\statevec{F}^{*,ec}$ and $\statevec{G}^{*,ec}$ as well as an entropy stable dissipation term which depends on the scaled flux eigenvalues
\begin{equation}
\begin{aligned}
\big|\mat{\Lambda}_f \big|  &=\frac{1}{2g}\begin{pmatrix}
\big|\average{u}+\average{c} \big|& 0 & 0\\
 0 & 2g\big|\average{h}\average{u}\big| & 0 \\
0 & 0 & \big|\average{u}-\average{c}\big|
\end{pmatrix},
\\
\big|\mat{\Lambda}_g \big|  &=\frac{1}{2g}\begin{pmatrix}
\big|\average{v}+\average{c} \big|& 0 & 0\\
 0 & 2g\big|\average{h}\average{v}\big| & 0 \\
0 & 0 & \big|\average{v}-\average{c}\big|
\end{pmatrix},
\end{aligned}
\end{equation}
and eigenvectors
\begin{equation}
\begin{aligned}
\mat{R}_f &=\begin{pmatrix}
1 & 0 & 1\\
\average{u}+\average{c} & 0 & \average{u}-\average{c} \\
\average{v} & 1 & \average{v}\\
\end{pmatrix},\\
\mat{R}_g &=\begin{pmatrix}
1 & 0 & 1\\
\average{u} & 1 & \average{u} \\
\average{v}+\average{c} & 0 & \average{v}-\average{c}\\
\end{pmatrix},
\end{aligned}
\end{equation}
and the jump in entropy variables $\jump{q}$.
The entropy conserving interface fluxes are
\begin{equation}\label{eq:surfaceFluxes}
\begin{aligned}
\statevec{F}^{*,ec}(\statevec{W}^{+},\statevec{W}^{-}) := \begin{pmatrix}
\average{h}\average{u}\\[0.1cm]
\average{h}\average{u}^2 + \frac{1}{2}\,g\,\average{h^2}\\[0.1cm]
\average{h}\average{u}\average{v}
\end{pmatrix},\\[0.2cm]
\statevec{G}^{*,ec}(\statevec{W}^{+},\statevec{W}^{-}) := \begin{pmatrix}
\average{h}\average{v}\\[0.1cm]
\average{h}\average{u}\average{v}\\[0.1cm]
\average{h}\average{v}^2 + \frac{1}{2}\,g\,\average{h^2}
\end{pmatrix},
\end{aligned}
\end{equation}
and the components of the dissipation term can be found in \ref{app:SimpleNumFlux}. The source term discretization $ \vec{\mathcal{S}}$ is defined by
\begin{equation}
\begin{aligned}
\label{SourceDisc}
(\tilde{\mathcal{S}}_1)_{ij} &:= 0 ,\\
(\tilde{\mathcal{S}}_2)_{ij} &:= \frac{g}{2} \left(-h b_x
+ \mathcal{S}_{i,j}^{x,*}
\right),\\
(\tilde{\mathcal{S}}_3)_{ij} &:=\frac{g}{2} \left(-hb_y
+  \mathcal{S}_{i,j}^{y,*}
 \right),
\end{aligned}
\end{equation}
where the additional interface penalty terms $\mathcal{S}^{\cdot,*}$ that account for possibly discontinuous bottom topographies can be found in \cite{ESDGSEM2D_paper} and the derivatives of the bottom topography are discretized in split form
\begin{equation}
\begin{aligned}
(b_x)_{ij} &=(y_{\eta})_{ij}\sum\limits_{m=0}^{N}2\,D_{i m} b_{mj} + \sum\limits_{m=0}^{N}2\,D_{i m}(y_{\eta} b)_{mj}
		- (y_{\xi})_{ij}\sum\limits_{m=0}^{N}2\,D_{j m} b_{im} - \sum\limits_{m=0}^{N}2\,D_{j m} (y_{\xi}b)_{im}, \\
(b_y)_{ij} &=
(x_{\eta})_{ij}\sum\limits_{m=0}^{N}2\,D_{i m} b_{mj} + \sum\limits_{m=0}^{N}2\,D_{i m}(x_{\eta} b)_{mj}
-(x_{\xi})_{ij}\sum\limits_{m=0}^{N}2\,D_{j m} b_{im} - \sum\limits_{m=0}^{N}2\,D_{j m} (x_{\xi}b)_{im}.
\end{aligned}
\end{equation}
The scheme described above is called the ESDGSEM and has the following properties:
   \begin{enumerate}[label={\theLem.\arabic*}]
       \item Discrete conservation of the mass and discrete conservation of the momentum if the bottom topography is constant.\label{Thm2partone}
       \item Guaranteed dissipation of the total discrete energy, which is an entropy function for the shallow water equations. Hence it fulfills the discrete entropy inequality \eqref{eq:entropyInequality}. \label{Thm2parttwo}
       \item Discrete well-balanced property for arbitrary bottom topographies. \label{Thm2partthree}
   \end{enumerate}
\end{Lem}
\begin{proof}
See \cite{ESDGSEM2D_paper}.
\end{proof}
The semi-discrete ESDGSEM is only complete when equipped with a suitable time integrator. We previously used a low storage Runge-Kutta scheme in \cite{ESDGSEM2D_paper} but have now switched to a strong stability preserving Runge-Kutta scheme (SSPRK). This choice is necessary for the positivity preservation proof in Section \ref{sec:wetDry}. We use the SSPRK method that is frequently used in wet/dry shallow water schemes as in \cite{xing2010positivity} and shock capturing schemes as in \cite{shu1989efficient}. The three stage scheme is 
\begin{equation}
\begin{aligned}
    \label{RKSSP}
    W^{(1)} & = W^{n} + \Delta t  \mathcal{R}\left(W^{n}\right), \\
    W^{(2)} & = \frac{3}{4} W^{n} + \fourth \left(W^{(1)}+ \Delta t  \mathcal{R}\left(W^{1}\right) \right),\\
    W^{n+1} & = \frac{1}{3} W^{n} + \frac{2}{3} \left(W^{(2)}+ \Delta t  \mathcal{R}\left(W^{2}\right) \right),\\
\end{aligned}
\end{equation}
where $\mathcal{R}$ denotes the spatial ESDGSEM operator
\begin{equation}
 \mathcal{R} =- \frac{1}{J} \left(  \vec{\mathcal{L}}_{\xi}+ \vec{\mathcal{L}}_{\eta} - \vec{\mathcal{S}} \right). 
\end{equation}

\section{Artificial Viscosity}
\label{Sec:artVisc}
The entropy stable scheme is more robust compared to a standard DGSEM for the shallow water equations \cite{gassner2016well,ESDGSEM2D_paper}, but it is not oscillation free in the presence of shocks. These oscillations might lead to unphysical solutions and cause simulations to crash if the water height becomes negative. We discuss the addition of a positivity limiter in the next section. Usually, such positivity preserving limiters are coupled with some form of TVB limiter \cite{zhang2010positivity} or artificial viscosity \cite{Marras_GalerkinViscSW} to keep oscillations in check. Our main requirement for the additional smoothing or limiting is to preserve the entropy stability of the scheme. In \cite{Gassner2017} the authors prove that under certain conditions, adding artificial viscosity to the scheme maintains the entropy stability. We choose the gradient variables carefully and use the discretization by Bassi and Rebay \cite{BassiRebay} to fulfill these conditions analogous to the strategy in \cite{Gassner2017}. For our shock capturing, we only add artificial viscosity to the momentum equations. While the smoothing is weaker without direct artificial viscosity in the continuity equation, we trivially maintain conservation of mass and have a straight forward mapping to the entropy variables, which leads to an easy proof of the requirements for the entropy stability via a theorem from \cite{Gassner2017}. In fact, we choose the entropy variables $q_1 = u$ and $q_2=v$ as gradient variables, leading to a simple one-to-one mapping from gradient variables to entropy variables.  We scale the gradient by the water height $h$ and a viscosity parameter $\epsilon$ to obtain the viscous fluxes. The viscosity parameter is dynamically computed for each element and is based on a smoothness measure of the water height within the element. We give details on the computation of the viscosity parameter in \ref{app:ViscPara}. The modified shallow water equations with artificial viscosity are
\begin{equation}
\begin{aligned}\label{sw-visc}
 h_t + (hu)_x + (hv)_y &=  0,\\
(hu)_t + (h\,u^2+g\,h^2/2)_x + (huv)_y &= -g\,h\,b_x+  \nabla\cdot  \left(h\,\epsilon\,\vec{\mathcal{U}}\right),\\
(hv)_t + (huv)_x + (h\,v^2+g\,h^2/2)_y &= -g\,h\,b_y +  \nabla\cdot  \left(h\,\epsilon\,\vec{\mathcal{V}}\right),\\
\vec{\mathcal{U}}& =\nabla u, \\
\vec{\mathcal{V}}& =\nabla v, \\
\end{aligned}
\end{equation}
or in compact flux form
\begin{equation}
\begin{aligned}\label{sw-visc-compact}
 \vec{w}_t + \nabla \cdot (\vec{f},\vec{g})^T &=   \vec{\mathcal{S}} + \nabla \cdot (\vec{f}^v,\vec{g}^v)^T ,\\
\vec{\mathcal{U}}& =\nabla u, \\
\vec{\mathcal{V}}& =\nabla v, \\
\end{aligned}
\end{equation}
with viscous fluxes $\vec{f}^v(\vec{w},\vec{\mathcal{U}}, \vec{\mathcal{V}})=h \, \epsilon \,  (0,\, \mathcal{U}_1,\mathcal{V}_1 )^T $ and $\vec{g}^v(\vec{w},\vec{\mathcal{U}}, \vec{\mathcal{V}})=h \, \epsilon \,  (0,\, {\mathcal{U}}_2,{\mathcal{V}}_2 )^T$.
We include the new viscous fluxes into the flux divergence and get
\begin{equation}
\begin{aligned}\label{sw-visc-refspace}
 \mathcal J \vec{w}_t &= -\hat\nabla \cdot (\vec{\tilde{f}}+\vec{\tilde{f}}^v,\vec{\tilde{g}}+\vec{\tilde{g}}^v)^T +   \vec{\tilde{\mathcal{S}}}  ,\\
\mathcal J\vec{\mathcal{U}}& = \left(
		\begin{array}{cc}
			y_\eta &    -y_\xi   \\
			-x_\eta&  x_\xi    \\
		\end{array}
	\right) \hat\nabla u , \\
	\mathcal J\vec{\mathcal{V}}& = \left(
		\begin{array}{cc}
			y_\eta &    -y_\xi   \\
			-x_\eta&  x_\xi    \\
		\end{array}
	\right) \hat\nabla v ,
\end{aligned}
\end{equation}
where we transform the physical gradient operator with the metrics of the element mappings. 

Analogous steps as in Sec. \ref{Sec:DG} are then used to find the weak and strong form and their discretizations for the continuity and momentum equations.
We multiply the transformed gradient equations by the same test function $\phi$ and integrate over the domain to find the weak form
\begin{equation}\label{Eq:viscousWeakForm}
\begin{aligned}
\int\limits_E \mathcal J \,\vec{\mathcal{U}}\,\phi \dE - \oint\limits_{\partial E} \phi \,  {U}^* \, \vec{n}\dS + \int\limits_E  	u	\left(
		\begin{array}{cc}
			y_\eta &    -y_\xi   \\
			-x_\eta&  x_\xi    \\
		\end{array}
	\right) \, \hat \nabla \phi   \dE &= 0, \\
\int\limits_E \mathcal J \,\vec{\mathcal{V}}\,\phi \dE - \oint\limits_{\partial E} \phi \,  {V}^* \, \vec{n}\dS + \int\limits_E  v		\left(
		\begin{array}{cc}
			y_\eta &    -y_\xi   \\
			-x_\eta&  x_\xi    \\
		\end{array}
	\right) \, \hat \nabla \phi   \dE &= 0, \\
\end{aligned}
\end{equation}
where ${U}^*$ and ${V}^*$ are the numerical interface states for the gradient equations. We discretize the weak form gradient equations \eqref{Eq:viscousWeakForm} and use the Lagrange property to simplify. For example, the volume integral approximations for the $\vec{\mathcal{U}}$ equations are
\begin{equation}\label{Eq:viscousApproximationEq1}
\begin{aligned}
\int\limits_{-1}^{1} \int\limits_{-1}^{1} 
u\, \left( y_\eta \frac{\partial }{\partial \xi} (\ell_i (\xi) \ell_j (\eta) )  -y_\xi \frac{\partial }{\partial \eta} (\ell_i (\xi) \ell_j (\eta) ) \right)\, \mathrm{d}\xi \mathrm{d}\eta
&\approx
-\sum_{m=0}^N   \Dhat_{im} u_{mj} {y_\eta}_{mj}   \, \omega_i\omega_j 
+
\sum_{n=0}^N  \Dhat_{jn} u_{in}  {y_\xi}_{in}   \, \omega_i\omega_j, \\
\int\limits_{-1}^{1} \int\limits_{-1}^{1} 
u\, \left(- x_\eta \frac{\partial }{\partial \xi} (\ell_i (\xi) \ell_j (\eta) )  +x_\xi \frac{\partial }{\partial \eta} (\ell_i (\xi) \ell_j (\eta) ) \right)\, \mathrm{d}\xi \mathrm{d}\eta
&\approx
\sum_{m=0}^N   \Dhat_{im} u_{mj} {x_\eta}_{mj}   \, \omega_i\omega_j 
-
\sum_{n=0}^N  \Dhat_{jn} u_{in}  {x_\xi}_{in}   \, \omega_i\omega_j. \\
\end{aligned}
\end{equation}
Similarly, the surface integrals for $\vec{\mathcal{U}}$ are approximated by
\begin{equation}\label{Eq:viscousApproximationSurface}
\begin{aligned}
\oint\limits_{\partial E} \phi \,  {U}^*\, \vec{n}_1\dS  
\approx&-  \delta_{i0}\, U^*_{i0} {y_\eta}_{i0}\omega_j
+\delta_{iN} \, U^*_{iN} {y_\eta}_{i0}\omega_j
- \delta_{j0}\, U^*_{0j} {y_\xi}_{0j}\omega_i
+ \delta_{jN} \, U^*_{Nj} {y_\xi}_{Nj} \omega_i,
\\
\oint\limits_{\partial E} \phi \,  {U}^*\, \vec{n}_2\dS 
\approx&-  \delta_{i0}\, U^*_{i0} {x_\eta}_{i0}\omega_j
+\delta_{iN} \, U^*_{iN} {x_\eta}_{i0}\omega_j
- \delta_{j0}\, U^*_{0j} {x_\xi}_{0j}\omega_i
+ \delta_{jN} \, U^*_{Nj} {x_\xi}_{Nj} \omega_i.
\end{aligned}
\end{equation}
We summarize the ESDGSEM with artificial viscosity and state the full discretizations of all the new viscous terms and gradient equations in Theorem \ref{thm:artVisc} and proof that the resulting method is entropy stable.
\begin{thm}[Entropy stability of ESDGSEM with artificial viscosity]
\label{thm:artVisc}
The ESDGSEM \eqref{Eq:CurvilinearECDGSEM} with additional viscous terms as in \eqref{sw-visc} is
\begin{equation}\label{Eq:CurvilinearECDGSEMviscous}
J \vec{W}_t + \vec{\mathcal{L}}_\xi+ \vec{\mathcal{L}}_\eta = \vec{\tilde{\mathcal{S}}} + \vec{\mathcal{L}}_\xi^v+ \vec{\mathcal{L}}_\eta^v.
\end{equation}
The viscous terms $\vec{\mathcal{L}}_\xi^v$, $\vec{\mathcal{L}}_\eta^v$ are discretized in strong form by
\begin{equation}
\begin{aligned}
\label{eq:ViscTermsDisc}
 (\vec{\mathcal{L}}_\xi^v)_{ij} =  & 
 \sum\limits_{m=0}^{N} {D}_{i m} (\vec{\tilde{F}}^{v} )_{mj} 
  + \frac{ 1}{\omega_i} \left(\delta_{iN} \vec{\tilde{F}}^{v,*}_{Nj} - \delta_{i0}\vec{\tilde{F}}^{v,*}_{0j}\right)  - \frac{ 1}{\omega_i} \left(\delta_{iN} \vec{\tilde{F}}^{v}_{Nj} - \delta_{i0}\vec{\tilde{F}}^{v}_{0j}\right),   \\
 (\vec{\mathcal{L}}_\eta^v)_{ij} =  & 
 \sum\limits_{m=0}^{N} {D}_{j m} (\vec{\tilde{G}}^{v})_{im} 
+ \frac{1}{\omega_j} \left(\delta_{Nj} \vec{\tilde{G}}^{v,*}_{iN} - \delta_{0j}\vec{\tilde{G}}^{v,*}_{i0} \right)
- \frac{1}{\omega_j} \left(\delta_{Nj} \vec{\tilde{G}}^{v}_{iN} - \delta_{0j}\vec{\tilde{G}}^{v}_{i0} \right).  \\
\end{aligned}
\end{equation}
The curvilinear viscous fluxes are defined with the BR1 approach as
\begin{equation}
\begin{aligned}
\vec{\tilde{F}}^{v}  &=  & y_{\eta} \vec{F}^v - x_{\eta} \vec{G}^v, \\
\vec{\tilde{G}}^{v}  &=  & -y_{\xi} \vec{F}^v + x_{\xi} \vec{G}^v, 
\end{aligned}
\end{equation}
with viscous fluxes 
\begin{equation}
\begin{aligned}
\vec{F}^{v}  &=   h\,\epsilon \, 
\left(\mathcal{U}_1, \,\mathcal{V}_1 \right)^T, \\
\vec{G}^{v}  &=  h \, \epsilon \, 
\left( \mathcal{U}_2 ,\,\mathcal{V}_2 \right)^T, 
\end{aligned}
\end{equation}
and the viscous flux interface coupling is computed with
\begin{equation}
\begin{aligned}
\vec{\tilde{F}}^{v,*} &= \average{\vec{\tilde{F}}^{v}},\\
\vec{\tilde{G}}^{v,*} &= \average{\vec{\tilde{G}}^{v}}.
\end{aligned}
\end{equation}
The gradients $\vec{\mathcal{U}} = \nabla u $ and $\vec{\mathcal{V}} = \nabla v $ are computed by
\begin{equation}\label{Eq:BR1}
\begin{aligned}
 (\mathcal{U}_1)_{ij} &=   (y_{\eta})_{ij}\sum\limits_{m=0}^{N}\Dhat_{i m} u_{mj} - (y_{\xi})_{ij}\sum\limits_{m=0}^{N}\Dhat_{j m} u_{im} 
+ \frac{ (y_{\eta})_{ij}}{\omega_i} \left(\delta_{iN} {U}^{*}_{Nj} - \delta_{i0}{U}^{*}_{0j}\right) 
+ \frac{ (y_{\xi})_{ij}}{\omega_j} \left(\delta_{Nj} {U}^{*}_{iN} - \delta_{0j}{U}^{*}_{i0}\right)   \\
 (\mathcal{U}_2)_{ij} &=    - (x_{\eta})_{ij}\sum\limits_{m=0}^{N}\Dhat_{i m} u_{mj} + (x_{\xi})_{ij}\sum\limits_{m=0}^{N}\Dhat_{j m} u_{im}
+ \frac{ (x_{\eta})_{ij}}{\omega_i} \left(\delta_{iN} {U}^{*}_{Nj} - \delta_{i0}{U}^{*}_{0j}\right) 
+ \frac{(x_{\xi})_{ij}}{\omega_j} \left(\delta_{Nj} {U}^{*}_{iN} - \delta_{0j}{U}^{*}_{i0}\right)    \\
 (\mathcal{V}_1)_{ij} &=   (y_{\eta})_{ij}\sum\limits_{m=0}^{N}\Dhat_{i m} v_{mj} - (y_{\xi})_{ij}\sum\limits_{m=0}^{N}\Dhat_{j m} v_{im} 
+ \frac{ (y_{\eta})_{ij}}{\omega_i} \left(\delta_{iN} {V}^{*}_{Nj} - \delta_{i0}{V}^{*}_{0j}\right) 
+ \frac{ (y_{\xi})_{ij}}{\omega_j} \left(\delta_{Nj} {V}^{*}_{iN} - \delta_{0j}{V}^{*}_{i0}\right)   \\
 (\mathcal{V}_2)_{ij} &=    - (x_{\eta})_{ij}\sum\limits_{m=0}^{N}\Dhat_{i m} v_{mj} + (x_{\xi})_{ij}\sum\limits_{m=0}^{N}\Dhat_{j m} v_{im}
+ \frac{ (x_{\eta})_{ij}}{\omega_i} \left(\delta_{iN} {V}^{*}_{Nj} - \delta_{i0}{V}^{*}_{0j}\right) 
+ \frac{(x_{\xi})_{ij}}{\omega_j} \left(\delta_{Nj} {V}^{*}_{iN} - \delta_{0j}{V}^{*}_{i0}\right),    \\
\end{aligned}
\end{equation}
where the numerical states ${U}^{*}$ and ${V}^{*}$ are chosen according to BR1 as the average values
\begin{equation}
\begin{aligned}
{U}^{*} &= \average{u}, \\
{V}^{*} &= \average{v}.
\end{aligned}
\end{equation}
The ESDGSEM with artificial viscosity \eqref{Eq:CurvilinearECDGSEMviscous} and viscous terms discretized as in \eqref{eq:ViscTermsDisc} and \eqref{Eq:BR1}, is entropy stable.
\end{thm}
\begin{proof}
In \cite{Gassner2017} the authors show that viscous terms discretized by Bassi and Rebay \cite{BassiRebay} are entropy stable if the viscous fluxes can be rewritten as the product of a symmetric, positive definite block matrix $\mathcal{B}^\epsilon$ and the gradient of the entropy variables
\begin{equation}
    \tensor{f^v} \left( \vec{w}, \,  \nabla \vec{w} \right) = \mathcal{B}^\epsilon  \nabla \vec{q} ,
\end{equation}
with state vectors $\tensor{f^v} = \left(
\begin{array}{c}
\vec{f}^v\\
\vec{g}^v\\
\end{array}
\right)
 \in \mathbb{R}^{6} $  
 and  
 $\nabla \vec{q} = \left(
\begin{array}{c}
\vec{q}_x\\
\vec{q}_y\\
\end{array}
\right)
 \in \mathbb{R}^{6} $ .
In the case of the shallow water equations, the block matrix $\mathcal{B}^\epsilon$ can be expressed as
\begin{equation}
\mathcal{B}^\epsilon = \left(
\begin{array}{cc}
\mat{B}_{11} & \mat{B}_{12} \\
\mat{B}_{21} & \mat{B}_{22} \\
\end{array}
\right)\in \mathbb{R}^{6\times6}.
\end{equation}
The entropy stability requirements for the block matrix $\mathcal{B}^\epsilon$ are then that each block $\mat{ B}_{ij}$ is symmetric
\begin{equation}
    \mat{ B}_{ij}^\epsilon = \left(\mat{ B}_{ji}^\epsilon\right)^T,
\end{equation}
and positive (semi-)definite
\begin{equation}
\label{PSD_req}
  \sum_{i=1}^{d}\sum_{j=1}^{d} \pderivative{\vec{q}}{x_i}^T \mat{ B}_{ij}^\epsilon \pderivative{\vec{q}}{x_j}   \geq 0 ,\quad\quad  \forall \vec{q}.
\end{equation}
The choice of gradient variables and fluxes in \eqref{sw-visc} make the block matrix very simple in this case
\begin{equation}
  \mathcal{B}^\epsilon = \epsilon h \left(
\begin{array}{cccccc}
0 & 0 & 0 & 0 & 0 & 0 \\
0 & 1 & 0 & 0 & 0 & 0 \\
0 & 0 & 1& 0 & 0 & 0 \\
0 & 0 & 0 & 0 & 0 & 0 \\
0 & 0 & 0 & 0 & 1 & 0 \\
0 & 0 & 0 & 0 & 0 & 1 \\
\end{array}
\right)  .
\end{equation}
We check the requirement \eqref{PSD_req} to find
\begin{equation}
\begin{aligned}
  \sum_{i=1}^{2}\sum_{j=1}^{2} \pderivative{\vec{q}}{x_i}^T \mat{ B}_{ij}^\epsilon \pderivative{\vec{q}}{x_j}   
  &= (\vec{q}_x)^T \, \mat{B}_{11}^\epsilon \, \vec{q}_x 
    + (\vec{q}_y)^T \, \mat{B}_{22}^\epsilon \, \vec{q}_y \\
    &= \epsilon \, h  \left(u_x^2 + v_x^2 + u_y^2 + v_y^2\right) \geq 0.
\end{aligned}
\end{equation}
\end{proof}
The artificial viscosity reduces the amount and magnitude of oscillations as shown in the numerical results section, specifically in a dam break example in Section \ref{subsec:WetParaDam}. 

\section{Positivity Preserving Limiter}
\label{sec:wetDry}
The artificial viscosity from Sec. \ref{Sec:artVisc} greatly reduces the oscillations, but in cases with wet/dry regions, this is sometimes not enough as even small oscillations may render a numerical scheme unstable if they lead to negative water heights. A mechanism is needed, that strictly enforces the positivity of the water height $h$ without destroying accuracy, conservation, well-balancedness or entropy-stability of the ESDGSEM. In, e.g., \cite{perthame1996positivity,zhang2010positivity,xing2010positivity,xing2013positivity} the authors have developed a positivity preserving limiter and applied it to the shallow water equation. The limiter is based on a linear scaling around element averages and thus relies on non-negative average water heights in all elements. It is then proved that non-negative average water heights are guaranteed to be preserved for an Euler time step. The proof relies on the specific choice of the numerical flux. Numerical fluxes that preserve non-negative water heights for finite volume schemes are called positivity preserving \cite{zhang2010positivity}. For Cartesian meshes, this property directly translates to DG methods by way of selecting any positivity preserving numerical interface fluxes, e.g. \cite{xing2010positivity}. In \cite{zhang2010positivity} it is proven that the Lax-Friedrichs numerical flux \cite{perthame1996positivity} is positivity preserving and the same is noted for the Godunov flux \cite{einfeldt1991godunov}, Boltzmann type flux \cite{perthame1992second} and the Harten-Lax-van-Leer flux \cite{harten1983upstream}. The positivity preserving property is shown for Euler time integration but naturally extends to SSPRK methods, e.g. \eqref{RKSSP}, as these are convex combinations of Euler time steps, see \cite{shu1988efficient} for details.

In this section we prove that the entropy stable numerical flux of the ESDGSEM \eqref{NumericalSurfaceFlux} is positivity preserving in the sense that non-negative mean water heights are preserved for one Euler time step. On Cartesian meshes it is possible to prove this similarly to \cite{zhang2010positivity}, where the update of the mean water height is written in finite volume form. Due to generally different Jacobians of the curvilinear mappings and possible changing normals on opposing sides this is not as straightforward on curved meshes. We directly prove the positivity preservation for curved quadrilateral meshes in Lemma \ref{Lem:PosPres}. Also, we verify that the positivity preserving limiter is entropy stable in Lemma \ref{Lem:ESPP}. Both results are then summarized in Theorem \ref{thm:PosPresESDGSEM}.

We will now show that the ESDGSEM with Euler time integration preserves a non-negative water height for a sufficiently small time step. For notational convenience we use the notation $W_{j,s}$ (opposed to $W_{ij}$ for internal nodal values) to denote the value of $W$ at node $j$ on interface $s$, where $s\in\{1,\ldots,4\}$ denotes the element local interface number. We also introduce the surface Jacobian $\mathcal J^{\text{surf}}$ on $\xi=\pm 1$ ($s=2,4)$ and $\eta=\pm 1$ ($s=1,3$) interfaces by
\begin{equation}
\label{eq:SurfJac}
    \begin{aligned}
     \mathcal J^{\text{surf}} &:= \sqrt{y_\xi y_\xi + x_\xi x_\xi} , \quad\quad \text{for $\eta=\pm 1$ },\\
     \mathcal J^{\text{surf}} &:= \sqrt{y_\eta y_\eta + x_\eta x_\eta}, \quad\quad \text{for $\xi=\pm 1$ }.\\
    \end{aligned}
\end{equation}
\begin{Lem}[Preservation of non-negative mean water heights in ESDGSEM]\label{Lem:PosPres}
If the water height, $h$, is non-negative for all LGL nodes then the average water height in the next time step is non-negative for all elements under the additional time step restrictions
\begin{equation}
\begin{aligned}
\label{eq:CFLConstraints}
    \Delta t &\leq  \frac{\omega_0\, a_{j,s}}{\left(A_{j,s} + 2 \average{\tilde{u}}_{j,s}\right)}, \\
    \Delta t &\leq \Bigg|\frac{ \omega_0\, a_{j,s}\,g \, h_{j,s}}{\average{c}_{j,s}B_{j,s} \jump{\tilde{u}}_{j,s}}\Bigg|,  \quad\quad \text{only if $h_{j,s}>0$},
\end{aligned}
\end{equation}
where we introduce the rotated normal velocity
\begin{equation}
    \tilde{u} := n_x u + n_y v
\end{equation}
for all edge nodes $j=0,\ldots,N$ on all element sides $s=1,\ldots, 4$, and 
\begin{equation}
\begin{aligned}
A &:= \big|\average{\tilde u}+\average{c}\big| + \big|\average{\tilde u}-\average{c}\big|,\\
B &:= \big|\average{\tilde u}+\average{c}\big| - \big|\average{\tilde u}-\average{c}\big|.
\end{aligned}
\end{equation}
For the Legendre Gauss Lobatto nodes, the quadrature weight $\omega_0$ is given by $\omega_0=\half N(N-1)$. We also have geometric scaling factors on the interfaces given by $a_{j,s} := \frac{\mathcal J_{j,s}}{\mathcal J^{\text{surf}}_{j,s}}$ with volume and surface Jacobians defined in \eqref{JacobianDef} and \eqref{eq:SurfJac}.
\end{Lem}
\begin{proof}
We first examine the numerical flux. Since the shallow water equations are rotationally invariant, we can calculate the numerical flux in normal direction and then rotate back. We want to guarantee positive water height and, thus, proceed to examine the numerical flux contributions for the water height equation. The first entry of the entropy stable numerical flux in the normal direction can be simplified to
\begin{equation}
\begin{aligned}
\label{RotatedfES_FirstEntry}
F_1^{*,es} 
&=\average{h}\average{\tilde{u}}- \frac{1}{4g} \left( A \jump{gh +gb } +   \average{c}B\jump{\tilde{u}}\right),
\end{aligned}
\end{equation}
where $\tilde{u} = n_x u + n_y v$ and full details on the derivation of \eqref{RotatedfES_FirstEntry} are given in \ref{app:SimpleNumFlux}.
As the ESDGSEM is a conservative numerical scheme, we can write the update of the element average water height in one Euler time step as
\begin{equation}\label{eq:CellAvgScheme}
\begin{aligned}
\overline{h}^{t_{n+1}} 
&= \overline{h}^{t_{n}}
-\frac{\Delta t}{\big|E\big|}\sum_{s=1}^{4}\sum_{j=0}^N \omega_j \mathcal J^{\text{surf}}_{j,s} \tilde{F}_1^{*,es}\left(W^{\text{int}}_{j,s},W^{\text{ext}}_{j,s}, n_{j,s}\right)   \\
&= \overline{h}^{t_{n}}
-\frac{\Delta t}{\big|E\big|}\sum_{s=1}^{4} \sum_{j=0}^N  \omega_j   \mathcal J^{\text{surf}}_{j,s} F_1^{*,es}\left(\tilde W^{\text{int}}_{j,s},\tilde W^{\text{ext}}_{j,s}\right)  , \\
\end{aligned}
\end{equation}
where $\mathcal J^{\text{surf}}_{j,s}$ is the surface Jacobian at node $j$ on interface $s$ defined in \eqref{eq:SurfJac}. 
We can also write the average water height $\overline{h}^{t_{n}}$ as 
\begin{equation}\label{eq:AverageTerm}
\begin{aligned}
\overline{h}^{t_{n}} = \frac{1}{\big|E\big|}\sum_{j=0}^N \sum_{i=0}^N h_{ij} \mathcal J_{ij} \omega_i\omega_j & \\
=\frac{1}{\big|E\big|}\sum_{i=1}^{N-1} \sum_{j=1}^{N-1}  h_{ij} \mathcal J_{ij} \omega_i\omega_j 
& +\half \frac{1}{\big|E\big|}\sum_{s=1}^4  \sum_{j=1}^{N-1}  h_{j,s} \mathcal J_{j,s} \omega_0\omega_j
+\half \frac{1}{\big|E\big|}\sum_{s=1}^4  \sum_{j=0}^{N}  h_{j,s} \mathcal J_{j,s} \omega_0\omega_j
\\
=\frac{1}{\big|E\big|}\sum_{i=1}^{N-1} \sum_{j=1}^{N-1}  h_{ij} \mathcal J_{ij} \omega_i\omega_j 
& +\half \frac{1}{\big|E\big|}\sum_{s=1}^4  \sum_{j=1}^{N-1}  h_{j,s} \mathcal J_{j,s} \omega_0\omega_j
\\
+\frac{1}{2\big|E\big|} \sum_{j=0}^{N}   h_{j,3} \mathcal J_{j,3} \omega_0\omega_j
&+ \frac{1}{2\big|E\big|}\sum_{j=0}^{N}  h_{j,1} \mathcal J_{j,1} \omega_0\omega_j \\
+ \frac{1}{2\big|E\big|}\sum_{i=0}^{N}  h_{j,4} \mathcal J_{j,4} \omega_0\omega_j
&+ \frac{1}{2\big|E\big|}\sum_{i=0}^{N}  h_{j,2} \mathcal J_{j,2} \omega_0\omega_j .
\end{aligned}
\end{equation}
Inserting the new expression for the average water height \eqref{eq:AverageTerm} into the update scheme \eqref{eq:CellAvgScheme} we find
\begin{equation}\label{eq:CellAvgScheme2}
\begin{aligned}
\overline{h}^{t_{n+1}} 
&= \frac{1}{\big|E\big|}\sum_{i=1}^{N-1} \sum_{j=1}^{N-1}  h_{ij} \mathcal J_{ij} \omega_i\omega_j 
 +\half \frac{1}{\big|E\big|}\sum_{s=1}^4  \sum_{j=1}^{N-1}  h_{j,s} \mathcal J_{j,s} \omega_0\omega_j
\\
& + \frac{1}{\big|E\big|}\sum_{j=0}^{N} \mathcal J_{j,1} \omega_0\omega_j \left[ \half h_{j,1}  
- \frac{\Delta t}{\omega_0 \, a_{j,1}} F_1^{*,es}\left(\tilde W^{\text{int}}_{j,1},\tilde W^{\text{ext}}_{j,1}\right)\right]  
\\
& +\frac{1}{\big|E\big|} \sum_{j=0}^{N} \mathcal J_{j,3} \omega_0\omega_j \left[ \half h_{j,3}  
- \frac{\Delta t}{\omega_0  \, a_{j,3}} F_1^{*,es}\left(\tilde W^{\text{int}}_{j,3},\tilde W^{\text{ext}}_{j,3}\right) 
\right]
\\
&+ \frac{1}{\big|E\big|}\sum_{j=0}^{N}  \mathcal J_{j,2} \omega_0\omega_j \left[\half h_{j,2} 
-\frac{\Delta t}{\omega_0 \, a_{j,2}} F_1^{*,es}\left(\tilde W^{\text{int}}_{j,2},\tilde W^{\text{ext}}_{j,2}\right) \right]
\\
&+ \frac{1}{\big|E\big|}\sum_{j=0}^{N}  \mathcal J_{j,4} \omega_0\omega_j \left[\half h_{j,4} 
- \frac{\Delta t}{\omega_0 \, a_{j,4}} 
F_1^{*,es}\left( \tilde W^{\text{int}}_{j,4},\tilde W^{\text{ext}}_{j,4}\right) 
\right],
\end{aligned}
\end{equation}
with $a_{j,s} := \frac{\mathcal J_{j,s}}{\mathcal J^{\text{surf}}_{j,s}}$. We note that for the special case of uniform Cartesian meshes this factor is simply $a_{j,s} = \Delta y$ for $s=1,3$ and $a_{j,s} = \Delta x$ for $s=2,4$.
The first two sums are clearly non-negative for meshes with positive Jacobians. We proceed to examine the interface terms
\begin{equation}
 \half h_{j,s}  
- \frac{\Delta t}{\omega_0 \, a_{j,s}}  F_1^{*,es}\left(\tilde W^{\text{int}}_{j,s},\tilde W^{\text{ext}}_{j,s}\right), \quad\quad s=1,\ldots,4 \quad .
\end{equation}
We can do this for an arbitrary node $j$, side $s$, and only need to distinguish between internal values $W^-$ and external values $W^+$ on the interface. We thus omit the indices $j$ and $s$ in the following steps.
By inserting the compact expression for $F_1^{*,es}$ from \eqref{RotatedfES_FirstEntry} and assuming continuous bottom topographies across element interfaces we find
\begin{equation}
\label{eq:UpdateConditions}
\begin{aligned}
 &\half h^- - \frac{1}{4g} \frac{\Delta t}{\omega_0\, a} \left( 4g\average{h}\average{\tilde{u}}-  g A \jump{h } -  \average{c}B \jump{\tilde{u}}\right) \\
&=\half  h^- - \frac{1}{4g} \frac{\Delta t}{\omega_0\, a} \left( 
g h^+ \left( 2\average{\tilde{u}} -  A\right) +gh^-\left(2\average{\tilde{u}}+  A\right) -  \average{c}B \jump{\tilde{u}}\right) \\
&=\frac{1}{4} \frac{\Delta t}{\omega_0\, a} h^+\left(A - 2 \average{\tilde{u}}\right) + \fourth h^- \left(1 - \frac{\Delta t}{\omega_0\, a} \left(A + 2 \average{\tilde{u}}\right)\right) + \fourth \left(  h^- + \frac{\Delta t}{ g \, \omega_0\, a}   \average{c}B \jump{\tilde{u}}\right)
\end{aligned}
\end{equation}
We examine the three terms in \eqref{eq:UpdateConditions} individually. The first term is always non-negative, since
\begin{equation}
  A =   \big|\average{\tilde{u}}+\average{c}\big| + \big|\average{\tilde{u}}-\average{c}\big| \geq  2\big|\average{\tilde{u}}\big|.
\end{equation}
For the second term we require an additional time step condition. From
\begin{equation}
    \left(1 - \frac{\Delta t}{\omega_0\, a}  \left(A + 2 \average{\tilde{u}}\right)\right) \mbgeq 0,
\end{equation}
we find
\begin{equation}
    \Delta t \mbleq  \frac{\omega_0\, a}{\left(A + 2 \average{\tilde{u}}\right)}.
\end{equation}
For the third term, we cannot generally factor out $h^-$, so we need to treat this carefully. We need to show that the last term is not negative in the cases $h^-=0$ and $h^+=0$ as well as in the wet case where both water heights are positive. In the case $h^-=0$, we also have $\tilde{u}^-=0$ and thus we see $ \jump{\tilde{u}} = \tilde{u}^+$, and the whole term becomes
\begin{equation}
  \frac{\Delta t}{g \, \omega_0\, a}  \average{c}B \tilde{u}^+ .
\end{equation}
We note that $\average{c}$ is always non-negative, so we must examine the signs of B and $\tilde{u}^+$. The velocity can have an arbitrary sign but from the definition of $B$
\begin{equation}
\begin{aligned}
B &= \big|\average{\tilde{u}}+\average{c}\big| - \big|\average{\tilde{u}}-\average{c}\big| = \bigg|\half \tilde{u}^+ +\average{c}\bigg| - \bigg|\half \tilde{u}^+-\average{c}\bigg|, \\ 
\end{aligned}
\end{equation}
we see that the sign of $B$ matches the sign of $\tilde{u}^+$ and thus the whole term is guaranteed non-negative for $h^-=0$ and $h^+\geq 0$. If $h^->0$ (and thus in general $\tilde{u}^-\neq0$), we are allowed to factor out $ h^-$. Then we require
\begin{equation}
\frac{\Delta t}{g\,\omega_0\, a} \frac{ \average{c}B \jump{\tilde{u}}}{h^-} \mbgeq -1 .
\end{equation}
This condition guarantees non negativity for the third term in \eqref{eq:UpdateConditions} and can only be violated if $B \jump{\tilde{u}}<0$. There are two sets of conditions where this is the case. Either we have $\tilde{u}^- > \tilde{u}^+$ and $\average{\tilde{u}}>0$, which implies $\tilde{u}^->0$. Or, alternatively, we have $B<0$ and $\jump{\tilde{u}}>0$, which implies $0>\tilde{u}^+ > \tilde{u}^-$.
In either way, for $B \jump{\tilde{u}}<0$, we can guarantee non negativity by enforcing the additional time step restriction
\begin{equation}
 \Delta t \mbleq \Bigg|\frac{g\,\omega_0\, a\, h^-}{\average{c}B \jump{\tilde{u}}}\Bigg|  .
\end{equation}
This proof for one Euler time step extends to SSPRK methods as used in the ESDGSEM as seen in, e.g., \cite{xing2010positivity}.
\end{proof}
While Lemma \ref{Lem:PosPres} guarantees a non-negative average water height in the next time step, we still need to ensure point-wise non-negativity. We use the limiter applied to the shallow water equations by Xing et al \cite{xing2010positivity} and developed in \cite{perthame1996positivity,zhang2010positivity,xing2013positivity} which is a linear scaling around the element average
\begin{equation}
\label{eq:PosPres}
    \widehat{\vec{W}}_{ij}  = \theta \left(\vec{W}_{ij} - \overline{\vec{W}}_{E} \right) + \overline{\vec{W}}_{E},
\end{equation}
where $\theta$ is computed by
\begin{equation}
  \theta =  \min \left(1, \frac{\overline{h}_{E}}{\overline{h}_{E} - m_{E}} \right),
\end{equation}
and $m_E$ is the minimum and $\overline{h}_{E}$ is the average water height in element $E$. The scaling is applied to the water height $h$ and the discharges $hu$ and $hv$ with the same parameter $\theta$ based on the water height. Since we have shown that the average water height is guaranteed positive in Lemma \ref{Lem:PosPres}, this limiter ensures positive water height for all computation nodes. It can be shown that this limiter maintains high order accuracy and is conservative \cite{zhang2010maximum}. We show that the positivity preserving limiter is also entropy stable in Lemma \ref{Lem:ESPP}.
\begin{Lem}[Entropy Stability of Positivity Preservation]
\label{Lem:ESPP}
An entropy stable method  coupled with the positivity preserving limiter \eqref{eq:PosPres} is still entropy stable.
\end{Lem}
\begin{proof}
We prove this result in a similar fashion to Ranocha \cite{ranocha2017}. Let $\mathcal{E}(\vec{W})$ denote the discrete total energy (entropy) within an element with solution polynomial $\vec W$. Also, we introduced the limited value of the solution polynomial around the element average $\widehat{\vec{W}}$ \eqref{eq:PosPres}. Then
\begin{equation}
\begin{aligned}
\overline{\mathcal{E}\!\left(\widehat{\vec{W}}\right)} = \frac{1}{\big|E\big|}\sum_{i=0}^N\sum_{j=0}^N \mathcal{E}\!\left(\widehat{\vec{W}}_{ij}\right) J_{ij}\, \omega_i \omega_j \stackrel{\mathcal{E}\text{ convex}}{\leq}& \frac{1}{\big|E\big|}\theta\sum_{i=0}^N\sum_{j=0}^N \mathcal{E}( \vec{W}_{ij}) J_{ij}\, \omega_i \omega_j + \frac{1}{\big|E\big|}(1-\theta)\sum_{i=0}^N\sum_{j=0}^N \mathcal{E}\!\left(\,\overline{\vec{W}}\,\right) J_{ij}\, \omega_i \omega_j \\
=&\frac{1}{\big|E\big|}\theta\sum_{i=0}^N\sum_{j=0}^N \mathcal{E}(\vec{W}_{ij}) J_{ij}\, \omega_i \omega_j + (1-\theta) \mathcal{E}\!\left(\, \overline{\vec{W}}\,\right) \\
\stackrel{\text{Jensen's inequality}}{\leq}&\frac{1}{\big|E\big|}\theta\sum_{i=0}^N\sum_{j=0}^N \mathcal{E}(\vec{W}_{ij}) J_{ij}\, \omega_i \omega_j + (1-\theta) \overline{\mathcal{E}(\vec W)} \\
=& \theta \overline{\mathcal{E}(\vec W)} + (1-\theta) \overline{\mathcal{E}(\vec W)} \\
=& \overline{\mathcal{E}(\vec W)}.
\end{aligned}
\end{equation}
So, the entropy of the modified solution polynomial $\hat W$ is less or equal to the entropy of the unmodified polynomial and it follows that the positivity limiter does not increase the entropy of the system.
\end{proof}
We summarize the results of this section in Theorem \ref{thm:PosPresESDGSEM}.
\begin{thm}[ESDGSEM with positivity limiter]\label{thm:PosPresESDGSEM}
The ESDGSEM \eqref{Eq:CurvilinearECDGSEM} combined with the positivity preserving limiter \eqref{eq:PosPres} and the additional time step restrictions \eqref{eq:CFLConstraints} fulfills all the properties from Lemma \ref{Lem:ESDGSEM} and also guarantees non-negative water heights for all LGL-nodes.
\end{thm}
\begin{proof}
We proved in Lemma \ref{Lem:PosPres} that preservation of non-negative water mean height is guaranteed if the water height in the previous time step is non-negative for all LGL nodes. The positivity preserving limiter \eqref{eq:PosPres} then guarantees non-negative water height at all LGL nodes. Finally, Lemma \ref{Lem:ESPP} proves that the positivity preserving limiter is entropy stable. We note that the properties from Lemma \ref{Lem:ESDGSEM} of mass conservation and the well-balancedness are unaffected by the positivity preservation procedure.
\end{proof}
\begin{rem}
Theorem \ref{thm:PosPresESDGSEM} also holds when combining the positivity preserving limiter with the artificial viscosity shock capturing from Sec. \ref{Sec:artVisc} since no artificial viscosity is added to the continuity equation.
\end{rem}
\begin{rem}
The positivity limiter does not affect the well-balanced property of the ESDGSEM since the scaling will not be applied for the ``lake at rest'' test case. However, the capability of handling dry areas leads to a generalization called the ``dry lake,'' defined by
\begin{equation}
\begin{aligned}
     h &= \max \left\{H_\text{const}-b, 0\right\} \\
     u&=v=0,
\end{aligned}
\end{equation}
where the bottom topography surpasses the constant water level, creating dry areas. If this leads to partially dry elements, the well-balanced property of the scheme is lost, as the proof relies strongly on the property $H=h+b=\text{const}$ and the consistency of the derivative operator $\mat{D}$. Retaining the well-balanced property for partly dry elements is a difficult challenge and subject to ongoing research. Strategies include adaptive mesh refinement or the development of different local derivative operators that account for the dry nodes within the element, e.g. \cite{bonev2017discontinuous}.
\end{rem}

\section{GPU Implementation}\label{sec:GPU}

Discontinuous Galerkin implementations on graphics processing units (GPUs) have been previously studied, for example in \cite{klockner2009nodal,klockner2013high,chan2016gpu,klockner2012solving,gandham2015gpu,modave2016gpu}.
DG algorithms with explicit time integration are particularly well suited to the massively parallel GPU architectures as most of the computational work is element local and elements are only coupled through interface exchanges for the computation of the surface integrals. None of the modifications in the ESDGSEM change the strong parallelizability. The ESDGSEM is, however, more computationally expensive than the standard DGSEM when counting the number of operations. Specifically the split form volume integral requires additional floating point operations compared to a standard volume integral. Our results show that the immense processing power of modern GPU hardware alleviates most of this increased computational complexity. We even observe that for polynomial degrees $N\leq 7$ the increased computational complexity of the split form is completely mitigated by the unleashed GPU processing power.

The ESDGSEM GPU implementation is based on the abstract language OCCA, a unified approach to multi-threading languages, developed by Medina et al. \cite{medina2014occa}. OCCA compiles OCCA kernel language (OKL) code at runtime for either CPU (Serial, OpenMP) or GPU (OpenCL, CUDA) architectures.
Our main test system features a NVIDIA GTX 1080 which is a higher end consumer grade card at the time this paper is written. We compile the kernels in CUDA and run the code in single precision as the GTX 1080 lacks double precision processing power. Only about $\frac{1}{32}$ of its CUDA cores are suited for double precision computations such that the theoretical peak performance is only 257 GLFOPS/s compared to the 8228 GFLOPS/s for single precision calculation. Thus, we use a different GPU, the NVIDIA Tesla V100, to produce double precision results. The Tesla V100 is a card designed for scientific computations and very well suited for double precision calculations as the double precision processing power is half of the single precision processing power. We show the double precision results in Subsection \ref{subsec:dpResults}.

In the implementation, we set up the mesh and problem on the CPU host and copy all the necessary data onto the GPU at the beginning of the computation. Then, data is only transferred back for MPI communication when using multiple GPUs, or for visualization purposes. Otherwise the whole computation is done on the GPU. As host to device data transfer is rather slow, this is an important aspect for the efficiency of the code. The different parts of the ESDGSEM are separated into individual kernels. There are, for instance, kernels for the computation of the surface integral and for the volume integral. However, the most computationally expensive kernel is by far the volume kernel. This is true for the implementation of the standard discontinuous Galerkin method as well as the ESDGSEM described in this paper. Thus, optimizing the performance of the volume kernel is a key priority and topic of the following discussion. 

The computation of the ESDGSEM volume integral requires significantly more operations than the standard DGSEM volume kernel. To see this, we restate the algorithm for the volume integrals at a node $i,j$:
\begin{equation}
\begin{aligned}
\label{VolInt}
(\text{Split Form Volume})_{ij} &=\sum_{l=0}^{N} \Dmod_{il} \Ftildeavg_{(l,i),j} +  \sum_{l=0}^{N} \Dmod_{jl} \Gtildeavg_{i,(j,l)} , \\
(\text{Standard Volume})_{ij} &=\sum_{l=0}^{N} D_{il} \Ftilde_{l,j} +  \sum_{l=0}^{N} D_{jl} \Gtilde_{i,l} . \\
\end{aligned}
\end{equation}
While the formulas look similar, the computation of the fluxes $\Ftildeavg$ and $\Gtildeavg$ each consists of computing $N+1$ flux evaluations for each node $i,j$, totaling $2(N+1)^3$ flux evaluations. In the standard DG formulation the fluxes need to be evaluated once for each node $i,j$, resulting in $2(N+1)^2$ flux evaluations. Additionally, each individual flux evaluation is more expensive for the ESDGSEM as it essentially consists of averaging two flux evaluations at $i,j$ and $l,j$ for $\Ftildeavg$ or $i,j$ and $i,l$ for $\Gtildeavg$. 

One strategy to increase the performance of the split form volume kernel is to use the symmetry $\Ftildeavg_{(l,i),j}=\Ftildeavg_{(i,l),j}$ to drastically reduce the number of floating point operations. This effectively halves the number of operations at the cost of storing more data, yielding a cost factor of about $2.5$ between standard and ESDGSEM. On GPUs, this trick is not possible due to the limited shared memory space. Precomputing and storing the fluxes $\Ftildeavg_{l,i,j}$ and $\Gtildeavg_{l,i,j}$ for $i,j,l = 0 \ldots,N$ would exceed the shared memory space even for medium sized problems. These architectural limitations raise the question of how to optimize GPU kernels and which end performance is actually satisfactory.
While there are performance numbers provided by the manufacturer, in practice these theoretical numbers are not achievable. In an effort to find a more realistic upper limit for the kernel performance, we estimate an empirical bound by investigating the effective bandwidth computed by
\begin{equation}
\text{effective bandwidth} = \frac{\text{Bytes Read + Bytes Written}}{t}.
\end{equation}
We compare the effective bandwidth of the volume kernels with the effective bandwidth of a GPU memory copy with the same number of bytes read and written. This memory copy is executed with a cudaMemcpyDeviceToDevice command and the resulting bandwidth is called the MemCopy-bandwidth.
Since a cudaMemcpy reads and writes each entry, a buffer of the size $\frac{\text{Bytes Read + Bytes Written}}{2}$ is used to calculate the MemCopy-bandwidth. We compute an empirical MemCopy roofline by scaling the GFLOPS/s achieved by the kernel with the factor of effective kernel bandwidth over MemCopy-bandwidth
\begin{equation}
\text{empirical MemCopy roofline} = \frac{\text{GFLOPS/s} \times \text{effective bandwidth}}{\text{MemCopy-bandwidth}}.
\end{equation}
The MemCopy roofline is a good upper bound on kernel performance whenever a kernel is limited by the memory bandwidth or \textbf{memory-bound}. When a kernel's performance is limited by the computations, or \textbf{compute-bound}, stricter bounds are needed. Another limiting factor can be the shared memory bandwidth. The shared memory bandwidth is estimated by
\begin{equation}
\text{shared memory bandwidth} = \text{\#cores} \times \text{\#SIMD Lanes} \times \|\text{word in bytes}\| \times \text{clock frequency},
\end{equation}
which for the GTX 1080 is 
\begin{equation}
\text{Shared-Mem-Bandwidth} = 20\times32\times4 \text{bytes} \times 1.607 \text{GHz} =4113.92\text{GB/s} .
\end{equation}
With the shared memory bandwidth we can estimate a bounding roofline by
\begin{equation}
\text{shared memory roofline}=\text{Shared-Mem-Bandwidth} \times \frac{\text{flops per block}}{\text{shared-mem bytes loaded+stored per block}} .
\end{equation}
We combine the MemCopy roofline and the shared memory roofline to find the combined roofline we will use in the following kernel analysis:
\begin{equation}
\text{combined roofline}=\min \left(\text{shared memory roofline},\text{MemCopy roofline}\right).
\end{equation}
These performance bounds give us a sharper hardware limit on the performance than numbers provided by the manufacturer and whenever the actual kernel performance is close to the roofline we can be satisfied with the kernel efficiency. In compute bound cases it may, however, still be impossible to actually reach the performance bound provided by the roofline.

Starting with a straightforward, naive implementation of the split form volume integral from \eqref{VolInt}  we sequentially introduce optimization steps and document the impact on the kernel performance. We measure the average kernel runtime for a sample test case over 1000 kernel executions. We count the number of floating point operations and obtain the performance measure GFLOPS/s as the number of floating point operations performed during each second of runtime.
We increase the polynomial order $N$ and decrease the number of elements $K$ to increase the computational complexity of the volume integral, while staying as close as possible to the the GPU memory of $8$ GB for the GTX 1080. The exact combinations of polynomial orders and number of elements can be found in Figure \ref{runtimeSameLOAD}.
The optimization steps applied in each different kernel version are described in \ref{app:Versioning}. The optimization techniques are similar to previous works, e.g. in \cite{gandham2015gpu} and subject to current research, e.g. \cite{chan2017gpu,swirydowicz2017acceleration,karakus2017gpu}.

We illustrate the impact of optimization techniques and plot the achieved performance of the different kernel versions and the empirical roofline for split form and standard kernels in Figure \ref{SharedMemRoof1} and Figure \ref{SharedMemRoof2}. 
Finally, we compare the most optimized versions for each kernel and show operation counts and runtimes for $N=1,\ldots,15$ with similar degrees of freedom in Figure \ref{runtimeSameDOF} and for similar memory loads in Figure \ref{runtimeSameLOAD}. The achieved performances of the ESDGSEM volume kernel are memory bound and thus close to optimal for $N\le 7$ and then stagnate as the performance becomes compute bound. The standard volume kernel behaves differently as even for $N=15$ the achieved performance lies close to the roofline.  As split form and standard kernels require the same amount of data, the runtime is almost identical in the memory bound region. For higher polynomial orders, the split form volume kernel becomes compute bound and the runtimes deviate. Here we can clearly observe the diverging number of floating point operations. However, while there is a factor of $6$ difference in the number of floating point operations for the ESDGSEM compared to standard DG, we only observe a runtime difference of a factor of $1.5$ in the most computationally expensive $N=15$ case.
Another observation to point out is that the achieved GFLOPS/s of the split form kernel are higher by a factor of $3$ to $5.5$. This suggests that the ESDGSEM volume kernel makes good use of the computational capabilities of the GPU architecture and is indeed a very good fit. Especially for the lower polynomial orders the additional computational complexity is completely mitigated by GPU processing power.

\begin{figure}
\centering
\begin{tikzpicture}[scale=1.60]
\begin{axis}[ylabel=GFLOPS/s,xlabel=$N$,legend pos=north west,legend style={nodes={scale=0.5, transform shape}}, ]
\addplot[color=red,dashed]  table[x=N, y=GFLOPSV0, col sep=semicolon] {Analysis/GFLOPS_FD_Allversions.csv};
\addplot[color=green,mark=*] table [x=N, y=GFLOPSV1, col sep=semicolon] {Analysis/GFLOPS_FD_Allversions.csv};
\addplot[color=pink,mark=star] table [x=N, y=GFLOPSV2, col sep=semicolon] {Analysis/GFLOPS_FD_Allversions.csv};
\addplot[color=gray,mark=diamond*] table [x=N, y=GFLOPSV3, col sep=semicolon] {Analysis/GFLOPS_FD_Allversions.csv};
\addplot[color=purple, mark=triangle*] table [x=N, y=GFLOPSV4, col sep=semicolon] {Analysis/GFLOPS_FD_Allversions.csv};
\addplot[color=teal, mark=otimes*] table [x=N, y=GFLOPSV5, col sep=semicolon] {Analysis/GFLOPS_FD_Allversions.csv};
\addplot[color=blue,mark=square*] table [x=N, y=GFLOPSV6, col sep=semicolon] {Analysis/GFLOPS_FD_Allversions.csv};
\addplot[color=black] table [x=N, y=MinBound, col sep=semicolon] {Analysis/GFLOPS_FD_Allversions.csv};
\addlegendentry{V0}
\addlegendentry{V1}
\addlegendentry{V2}
\addlegendentry{V3}
\addlegendentry{V4}
\addlegendentry{V5}
\addlegendentry{V6}
\addlegendentry{Roof}
\end{axis}
\end{tikzpicture}
\caption{Comparison of all ESDGSEM volume kernel versions and the memory roofline on a NVIDIA GTX 1080 in single precision.}
\label{SharedMemRoof1}
\end{figure}
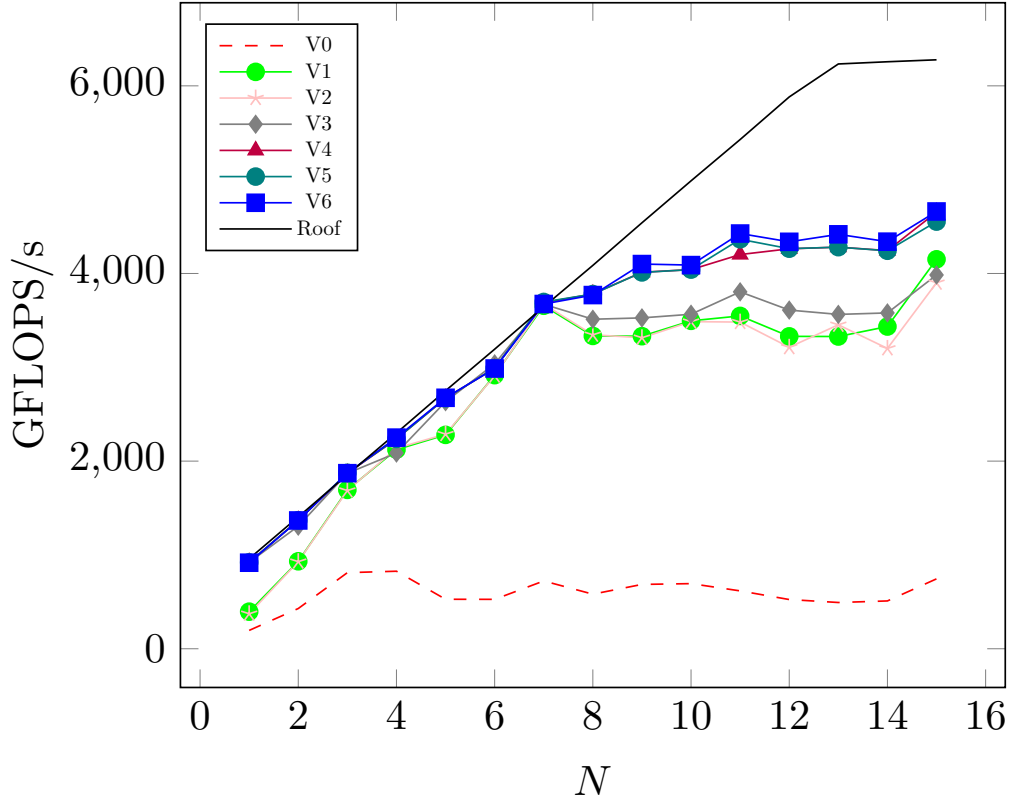
\begin{figure}
\centering
\begin{tikzpicture}[scale=1.60]
\begin{axis}[ylabel=GFLOPS/s,xlabel=$N$,legend pos=north west,legend style={nodes={scale=0.5, transform shape}}, ]
\addplot[color=red,dashed] table [x=N, y=GFLOPSV0, col sep=semicolon] {Analysis/GFLOPS_SD_Allversions.csv};
\addplot[color=green,mark=*] table [x=N, y=GFLOPSV1, col sep=semicolon] {Analysis/GFLOPS_SD_Allversions.csv};
\addplot[color=pink,mark=star] table [x=N, y=GFLOPSV2, col sep=semicolon] {Analysis/GFLOPS_SD_Allversions.csv};
\addplot[color=gray,mark=diamond*] table [x=N, y=GFLOPSV3, col sep=semicolon] {Analysis/GFLOPS_SD_Allversions.csv};
\addplot[color=purple, mark=triangle*] table [x=N, y=GFLOPSV4, col sep=semicolon] {Analysis/GFLOPS_SD_Allversions.csv};
\addplot[color=blue,mark=square*] table [x=N, y=GFLOPSV5, col sep=semicolon] {Analysis/GFLOPS_SD_Allversions.csv};
\addplot[color=black] table [x=N, y=MinBound, col sep=semicolon] {Analysis/GFLOPS_SD_Allversions.csv};
\addlegendentry{V0}
\addlegendentry{V1}
\addlegendentry{V2}
\addlegendentry{V3}
\addlegendentry{V4}
\addlegendentry{V5}
\addlegendentry{Roof}
\end{axis}
\end{tikzpicture}
\caption{Comparison of all standard DGSEM volume kernel versions and the memory roofline on a NVIDIA GTX 1080 in single precision.}
\label{SharedMemRoof2}
\end{figure}
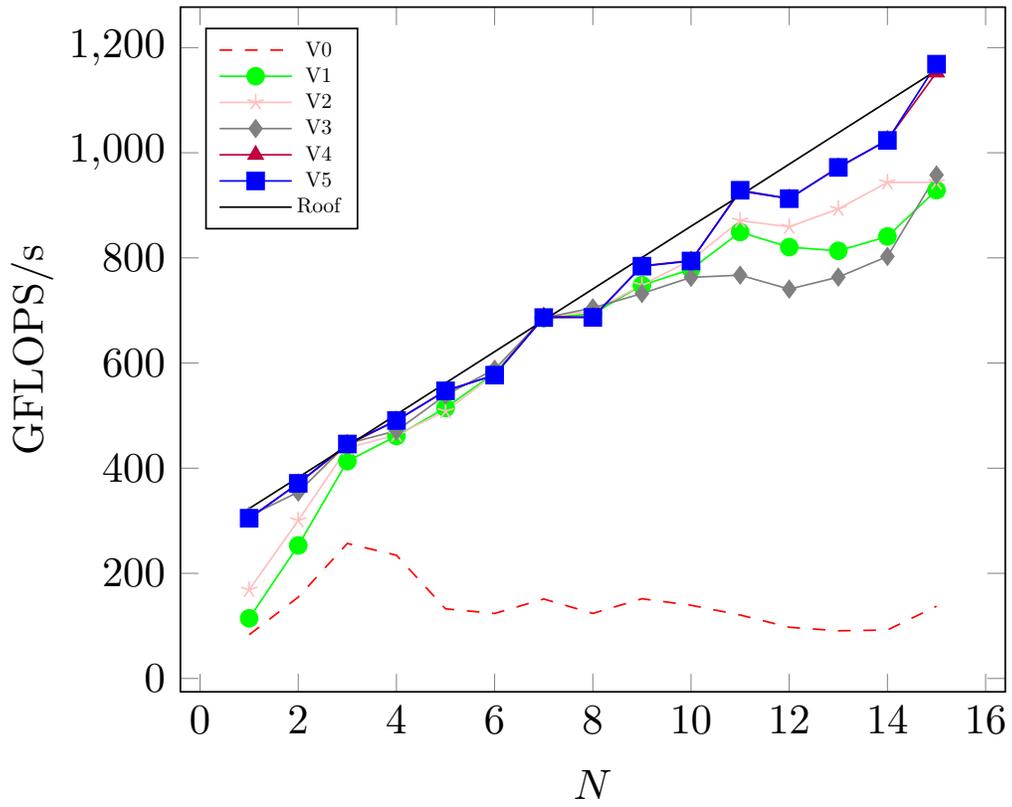

\begin{figure}
\centering
\begin{filecontents*}{Analysis/Disc_sameDof.csv}
N;K
1;3000
2;2000
3;1500
4;1200
5;1000
6;858
7;750
8;667
9;600
10;546
11;500
12;462
13;429
14;400
15;375
\end{filecontents*}
\pgfplotstabletypeset[
col sep = semicolon,
every head row/.style={before row=\toprule,after row=\midrule},
every last row/.style={after row=\bottomrule},
]
{Analysis/Disc_sameDof.csv}
\hspace{2cm}
\begin{tikzpicture}[baseline]
\begin{axis}[ylabel=operation count (billions),
            xlabel=N,
            legend pos=north west,
            anchor=center]
\addplot table [x=N, y=GFLOPSFD, col sep=semicolon] {Analysis/TimingsTable_sameDOF.csv};
\addplot table [x=N, y=GFLOPSSD, col sep=semicolon] {Analysis/TimingsTable_sameDOF.csv};
\addlegendentry{ESDGSEM}
\addlegendentry{Standard}
\end{axis}
\end{tikzpicture}
\begin{tikzpicture}
\begin{axis}[ylabel=runtime ($\mu s$ per million DOFs),
            xlabel=N,
            legend pos=north west,
            ybar,
            ymin=0,
            ymax=110,
            xmin = 0.5, xmax = 15.5,
            bar width=.2cm,
            width=\textwidth,
            height=.5\textwidth]
\addplot[fill=blue, postaction={pattern=horizontal lines}] table [x=N, y=FD_ms_mio, col sep=semicolon] {Analysis/TimingsTable_sameDOF.csv};
\addplot[fill=red] table [x=N, y=SD_ms_mio, col sep=semicolon] {Analysis/TimingsTable_sameDOF.csv};
\addplot[fill=yellow, postaction={pattern=dots}] table [x=N, y=MemCopy_ms_mio, col sep=semicolon] {Analysis/TimingsTable_sameDOF.csv};
\addlegendentry{ESDGSEM}
\addlegendentry{Standard}
\addlegendentry{MemCpy}
\end{axis}
\end{tikzpicture}
\caption{Number of operations and runtime comparison for one kernel executio between the split form volume integral computation of the ESDGSEM and the volume integral of the standard DG method for similar number of degrees of freedom (DOFs) on a NVIDIA GTX 1080 in single precisions. Polynomial order $N$ and number of elements $K$ per spatial direction are listed in the top left table.}
\label{runtimeSameDOF}
\end{figure}

\begin{figure}
\centering
\begin{filecontents*}{Analysis/Disc_incLoad.csv}
N;K
1;4000
2;2800
3;2000
4;1600
5;1400
6;1200
7;1000
8;900
9;800
10;750
11;700
12;650
13;600
14;550
15;500
\end{filecontents*}
\pgfplotstabletypeset[
col sep = semicolon,
every head row/.style={before row=\toprule,after row=\midrule},
every last row/.style={after row=\bottomrule},
]
{Analysis/Disc_incLoad.csv}
\hspace{2cm}
\begin{tikzpicture}[baseline]
\begin{axis}[ylabel=operation count (billions),
            xlabel=N,
            legend pos=north west,
            anchor=center]
\addplot table [x=N, y=GFLOPSFD, col sep=semicolon] {Analysis/TimingsTable_incLoad.csv};
\addplot table [x=N, y=GFLOPSSD, col sep=semicolon] {Analysis/TimingsTable_incLoad.csv};
\addlegendentry{ESDGSEM}
\addlegendentry{Standard}
\end{axis}
\end{tikzpicture}
\begin{tikzpicture}
\begin{axis}[ylabel=runtime ($\mu s$ per million DOFs),
            xlabel=N,
            legend pos=north west,
            ybar,ymin=0,ymax=110,
            xmin = 0.5, xmax = 15.5,
            bar width=.2cm,
            width=\textwidth,
            height=.5\textwidth]
\addplot[fill=blue, postaction={pattern=horizontal lines}] table [x=N, y=FD_ms_mio, col sep=semicolon] {Analysis/TimingsTable_incLoad.csv};
\addplot[fill=red]  table [x=N, y=SD_ms_mio, col sep=semicolon] {Analysis/TimingsTable_incLoad.csv};
\addplot[fill=yellow, postaction={pattern=dots}] table [x=N, y=MemCopy_ms_mio, col sep=semicolon] {Analysis/TimingsTable_incLoad.csv};
\addlegendentry{ESDGSEM}
\addlegendentry{Standard}
\addlegendentry{MemCpy}
\end{axis}
\end{tikzpicture}
\caption{Number of operations and runtime comparison for one kernel executio between the split form volume integral computation of the ESDGSEM and the volume integral of the standard DG method for increasing computational complexity on a NVIDIA GTX 1080 in single precision. Polynomial order $N$ and number of elements $K$ per spatial direction are listed in the top left table.}
\label{runtimeSameLOAD}
\end{figure}

\subsection{Double Precision Results}
\label{subsec:dpResults}
We use a different test system with a Tesla V100 to run the ESDGSEM volume kernel in double precision. 
We use the same kernels as for the single precision results. We show the performance of the different split form kernel versions in Figure \ref{ESDGSEM_SharedMemRoof_DP} as well as for the standard kernel versions in Figure \ref{STD_SharedMemRoof_DP}. We, again, compare runtimes for the most optimized kernels in Figure
\ref{DP_runtimeSameLOAD}. The elements per work block were found by optimization for the GTX 1080 so it may be possible to find better values for the Tesla V100. This could explain why the most optimized kernel is already compute bound for $N=6$ to $N=7$ and why there is a significant drop in performance for $N=9$. We do however see very satisfactory performance for lower order $N$ where, once again, the split form kernel performance is close to optimal.
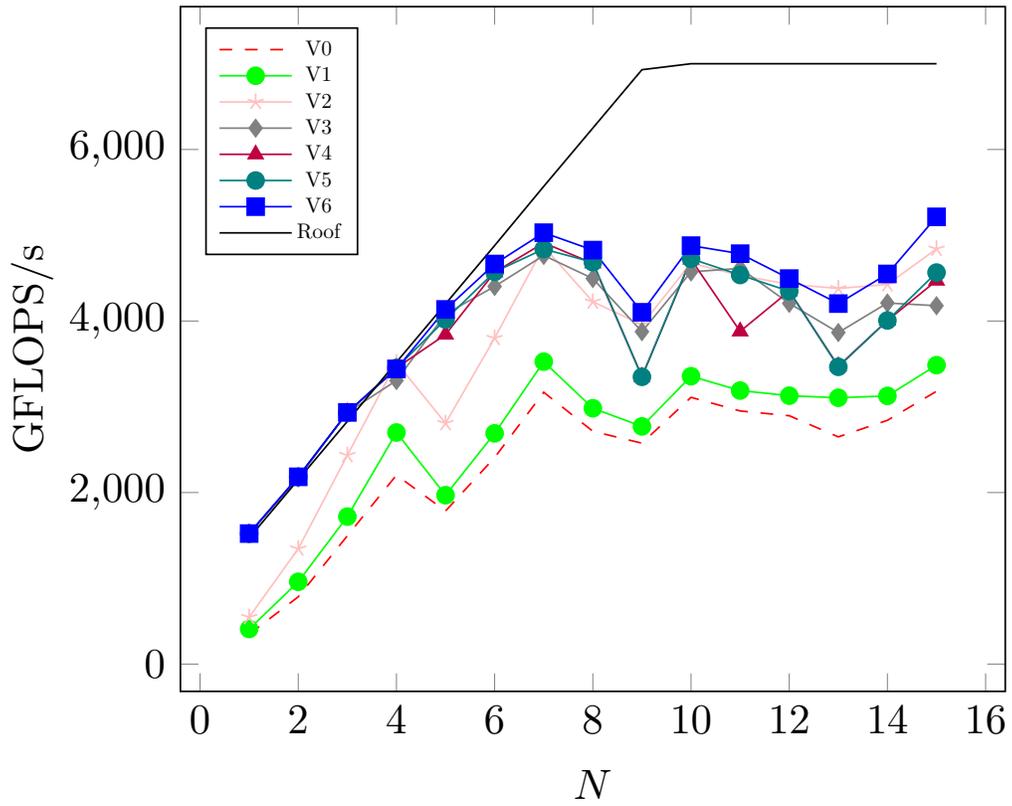
\begin{figure}
\centering
\begin{tikzpicture}[scale=1.60]
\begin{axis}[ylabel=GFLOPS/s,xlabel=$N$,legend pos=north west,legend style={nodes={scale=0.5, transform shape}}, ]
\addplot[color=red,dashed]  table[x=N, y=GFLOPSV0, col sep=semicolon] {Analysis/GFLOPS_FD_Allversions_DP.csv};
\addplot[color=green,mark=*] table [x=N, y=GFLOPSV1, col sep=semicolon] {Analysis/GFLOPS_FD_Allversions_DP.csv};
\addplot[color=pink,mark=star] table [x=N, y=GFLOPSV2, col sep=semicolon] {Analysis/GFLOPS_FD_Allversions_DP.csv};
\addplot[color=gray,mark=diamond*] table [x=N, y=GFLOPSV3, col sep=semicolon] {Analysis/GFLOPS_FD_Allversions_DP.csv};
\addplot[color=purple, mark=triangle*] table [x=N, y=GFLOPSV4, col sep=semicolon] {Analysis/GFLOPS_FD_Allversions_DP.csv};
\addplot[color=teal, mark=otimes*] table [x=N, y=GFLOPSV5, col sep=semicolon] {Analysis/GFLOPS_FD_Allversions_DP.csv};
\addplot[color=blue,mark=square*] table [x=N, y=GFLOPSV6, col sep=semicolon] {Analysis/GFLOPS_FD_Allversions_DP.csv};
\addplot[color=black] table [x=N, y=MinBound, col sep=semicolon] {Analysis/GFLOPS_FD_Allversions_DP.csv};
\addlegendentry{V0}
\addlegendentry{V1}
\addlegendentry{V2}
\addlegendentry{V3}
\addlegendentry{V4}
\addlegendentry{V5}
\addlegendentry{V6}
\addlegendentry{Roof}
\end{axis}
\end{tikzpicture}
\caption{Comparison of all ESDGSEM volume kernel versions and the memory roofline on a Tesla V100 in double precision.}
\label{ESDGSEM_SharedMemRoof_DP}
\end{figure}

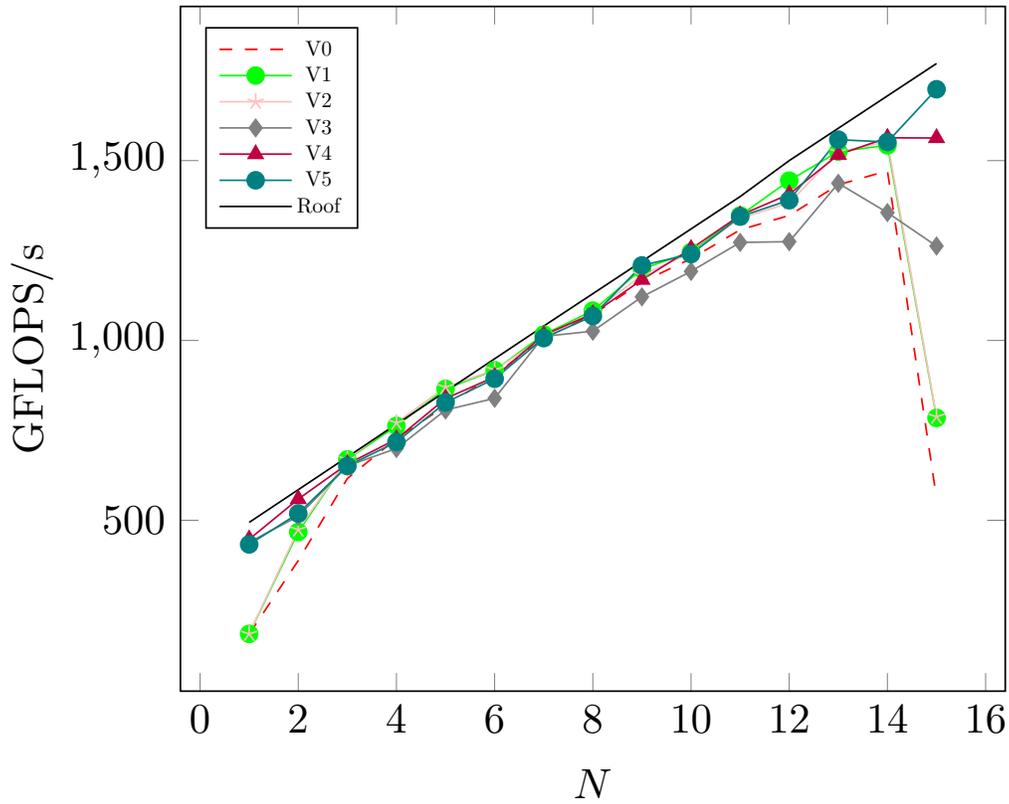
\begin{figure}
\centering
\begin{tikzpicture}[scale=1.60]
\begin{axis}[ylabel=GFLOPS/s,xlabel=$N$,legend pos=north west,legend style={nodes={scale=0.5, transform shape}}, ]
\addplot[color=red,dashed]  table[x=N, y=GFLOPSV0, col sep=semicolon] {Analysis/GFLOPS_SD_Allversions_DP.csv};
\addplot[color=green,mark=*] table [x=N, y=GFLOPSV1, col sep=semicolon] {Analysis/GFLOPS_SD_Allversions_DP.csv};
\addplot[color=pink,mark=star] table [x=N, y=GFLOPSV2, col sep=semicolon] {Analysis/GFLOPS_SD_Allversions_DP.csv};
\addplot[color=gray,mark=diamond*] table [x=N, y=GFLOPSV3, col sep=semicolon] {Analysis/GFLOPS_SD_Allversions_DP.csv};
\addplot[color=purple, mark=triangle*] table [x=N, y=GFLOPSV4, col sep=semicolon] {Analysis/GFLOPS_SD_Allversions_DP.csv};
\addplot[color=teal, mark=otimes*] table [x=N, y=GFLOPSV5, col sep=semicolon] {Analysis/GFLOPS_SD_Allversions_DP.csv};
\addplot[color=black] table [x=N, y=MinBound, col sep=semicolon] {Analysis/GFLOPS_SD_Allversions_DP.csv};
\addlegendentry{V0}
\addlegendentry{V1}
\addlegendentry{V2}
\addlegendentry{V3}
\addlegendentry{V4}
\addlegendentry{V5}
\addlegendentry{Roof}
\end{axis}
\end{tikzpicture}
\caption{Comparison of all standard DG volume kernel versions and the memory roofline on a Tesla V100 in double precision.}
\label{STD_SharedMemRoof_DP}
\end{figure}

\begin{figure}
\centering
\begin{tikzpicture}
\begin{axis}[ylabel=runtime ($\mu s$ per million DOFs),
            xlabel=N,
            legend pos=north west,
            ybar,ymin=0,ymax=110,
            xmin = 0.5, xmax = 15.5,
            bar width=.2cm,
            width=\textwidth,
            height=.5\textwidth]
\addplot[fill=blue, postaction={pattern=horizontal lines}] table [x=N, y=FD_ms_mio, col sep=semicolon] {Analysis/Timings_DP.csv};
\addplot[fill=red] table [x=N, y=SD_ms_mio, col sep=semicolon] {Analysis/Timings_DP.csv};
\addplot[fill=yellow, postaction={pattern=dots}] table [x=N, y=MemCopy_ms_mio, col sep=semicolon] {Analysis/Timings_DP.csv};
\addlegendentry{ESDGSEM}
\addlegendentry{Standard}
\addlegendentry{MemCpy}
\end{axis}
\end{tikzpicture}
\caption{Number of operations and runtime comparison for one kernel execution between the split form volume integral computation of the ESDGSEM and the volume integral of the standard DG method for increasing computational complexity in double precision on the NVIDIA Tesla V100. The number of elements per polynomial order are identical to the single precision case in Figure \ref{runtimeSameLOAD}.}
\label{DP_runtimeSameLOAD}
\end{figure}
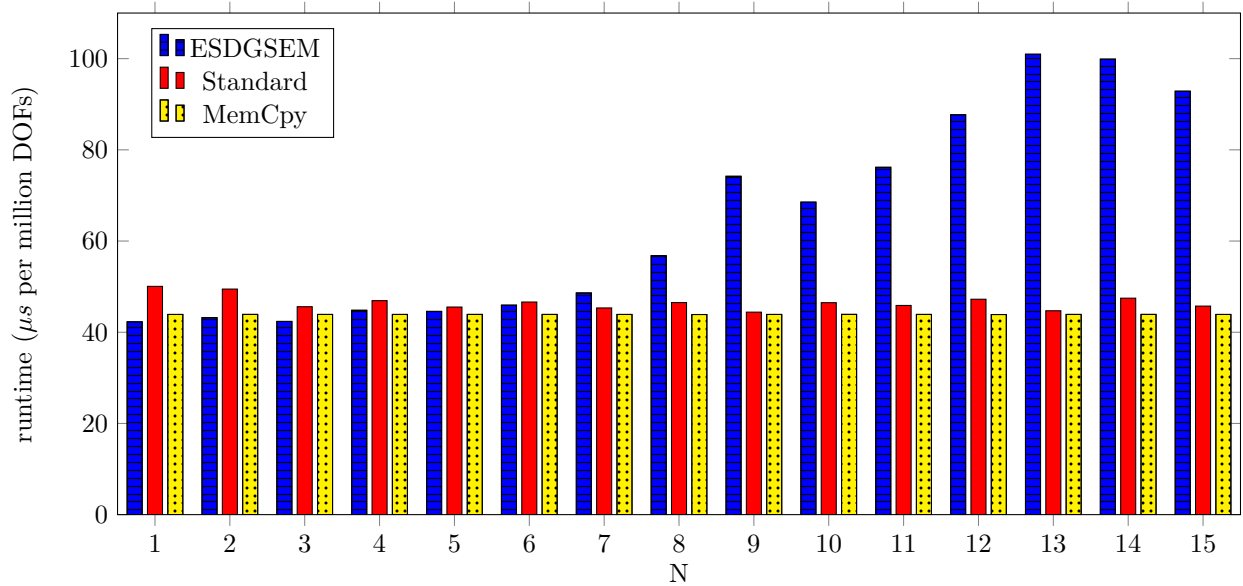
\FloatBarrier

\section{Numerical Results}\label{sec:num}
In this section we first numerically verify and demonstrate the theoretical properties of the scheme in Subsection \ref{sec:consES}. 
In particular, we present the difference in the numerical solution of the entropy stable scheme with artificial viscosity presented in this paper, the ESDGSEM without artifical viscosity, and a standard DGSEM of the same polynomial order. To fully explore the solution quality of the new scheme, we apply the ESDGSEM with artificial viscosity to several well-known test cases which require shock capturing and positivity preservation. Namely, the oscillating lake \cite{xing2011high,gallardo2007well,swefv2,Marras_GalerkinViscSW,VATER20151}, the three mound dam break on a closed channel \cite{gallardo2007well,swefv2,Marras_GalerkinViscSW,brufau2002numerical, xing2013positivity} and a solitary wave run-up \cite{bollermann2011finite,Marras_GalerkinViscSW,phung2008numerical}. Lastly, to test the properties on a curvilinear mesh, we modify the partial dam break from \cite{ESDGSEM2D_paper} to feature a dry area on the shallow side of the dam. 

We choose the time step based on a typical CFL condition. The additional time step restrictions of the positivity limiter \eqref{eq:CFLConstraints} are sufficient but not necessary. If we detect that a smaller time step is needed, we adjust it accordingly. A more detailed discussion on choosing an appropriate time step for positivity preserving schemes is given by Xing and Zhang \cite{xing2013positivity}.

All the examples have been computed on two NVIDIA GTX 1080 GPUs, where we use MPI parallelization such that each MPI rank hosts one GPU via OCCA \cite{medina2014occa,medina2015okl}.

\subsection{Theory Validation}\label{sec:consES}

We first verify the theoretical entropy stable properties of the approximation described in this work. The ESDGSEM was previously shown to be high-order accurate but unphysical overshoots remained near discontinuities in Wintermeyer et al. \cite{ESDGSEM2D_paper}. Therefore, we demonstrate in Sec. \ref{subsec:EntropyGlitch} that applying the artificial viscosity from Thm. \ref{thm:artVisc} can remove spurious oscillations and remain entropy stable. We then consider a shocktube test in Sec. \ref{subsubsec:PPnec} that requires the positivity preserving limiter to model a flow with dry regions. Further, we numerically verify the result of Lemma \ref{Lem:PosPres} that the limited DG solution is entropy stable. We choose a positivity tolerance level of $h^{\text{TOL}}=10^{-4}$ below which velocities are set to zero in the positivity preserving stage.

\subsubsection{Entropy Glitch}
\label{subsec:EntropyGlitch}
We use a specific dam break problem to demonstrate the necessity of entropy stability to guarantee correct solutions. We set up a two-dimensional version of the one-dimensional test case previously done in \cite{lukacova2009entropy}. We use a uniform Cartesian mesh on the domain $\Omega=[-1,1]^2$ with $100\times 100$ elements. There is a shock at $x=0$ with water heights $h_L =1$ and $h_R=0.1$ and the velocities are zero, i.e.,
\begin{equation}
\begin{aligned}
&h(x,y,0)= \left\{
\begin{aligned}
&1.0, \quad&\textrm{if } x < 0  \\
&0.1, \quad&\textrm{otherwise}
\end{aligned}
\right.,\\
&u=v=0.
\end{aligned}
\end{equation}
For simplicity, the gravity constant is set to $g=10$ for this example and the test is run up to $T=0.2$.  We show results for $N=1$ as the entropy glitch even occurs for this most robust case. Also the standard DGSEM quickly becomes unstable for higher polynomial orders due to the oscillations introduced by the entropy glitch. We still use the positivity limiter to catch some minor overshoots due to the oscillations. We compare a slice of the solution at $y=0$ with the 1D solution provided in \cite{lukacova2009entropy}.
Solutions obtained by the standard DG method with a local Lax-Friedrichs numerical interface flux show an unphysical discontinuity, an ``entropy glitch,'' at $x=0$, see Figure \ref{fig:EntropyGlitchN1}. The unphysical shock does not appear in the solutions obtained by the ESDGSEM. There are however oscillations at the shock front for both schemes. We also plot the evolution of the total entropy for the both cases and observe that the total entropy builds up at around $t=0.02$ for the standard DGSEM whereas the total entropy is strictly decreasing for the ESDGSEM. We conclude that even limited entropy build up may cause unphysical phenomena to appear and destroy the correctness of the solution.

\begin{figure}[!ht]
   \centering
    \subfloat[Water height $H$, $N=1$]
    {
        \includegraphics[width=0.5\textwidth]{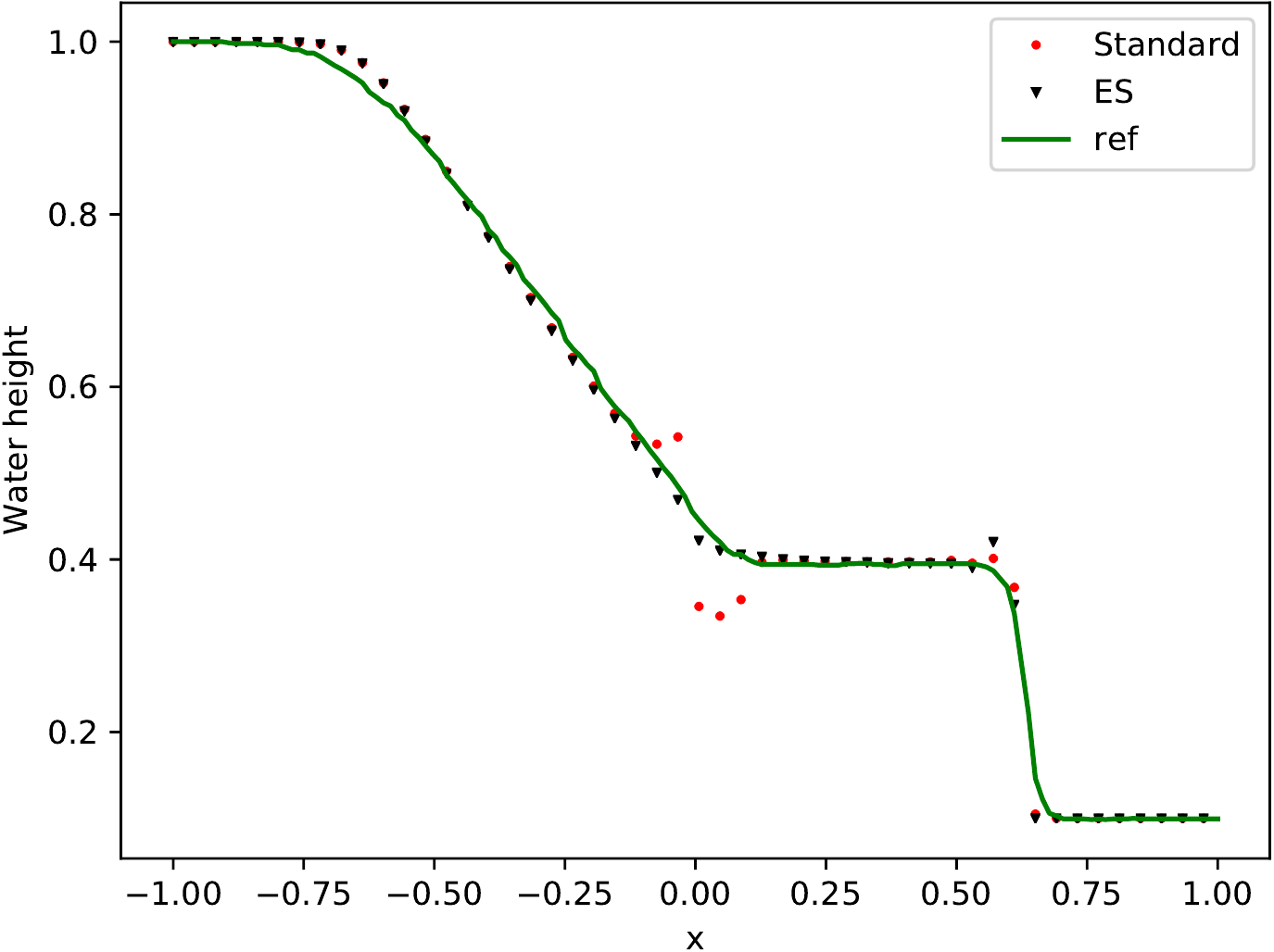}
    }
    \subfloat[Total Entropy, $N=1$]
    {
        \includegraphics[width=0.5\textwidth]{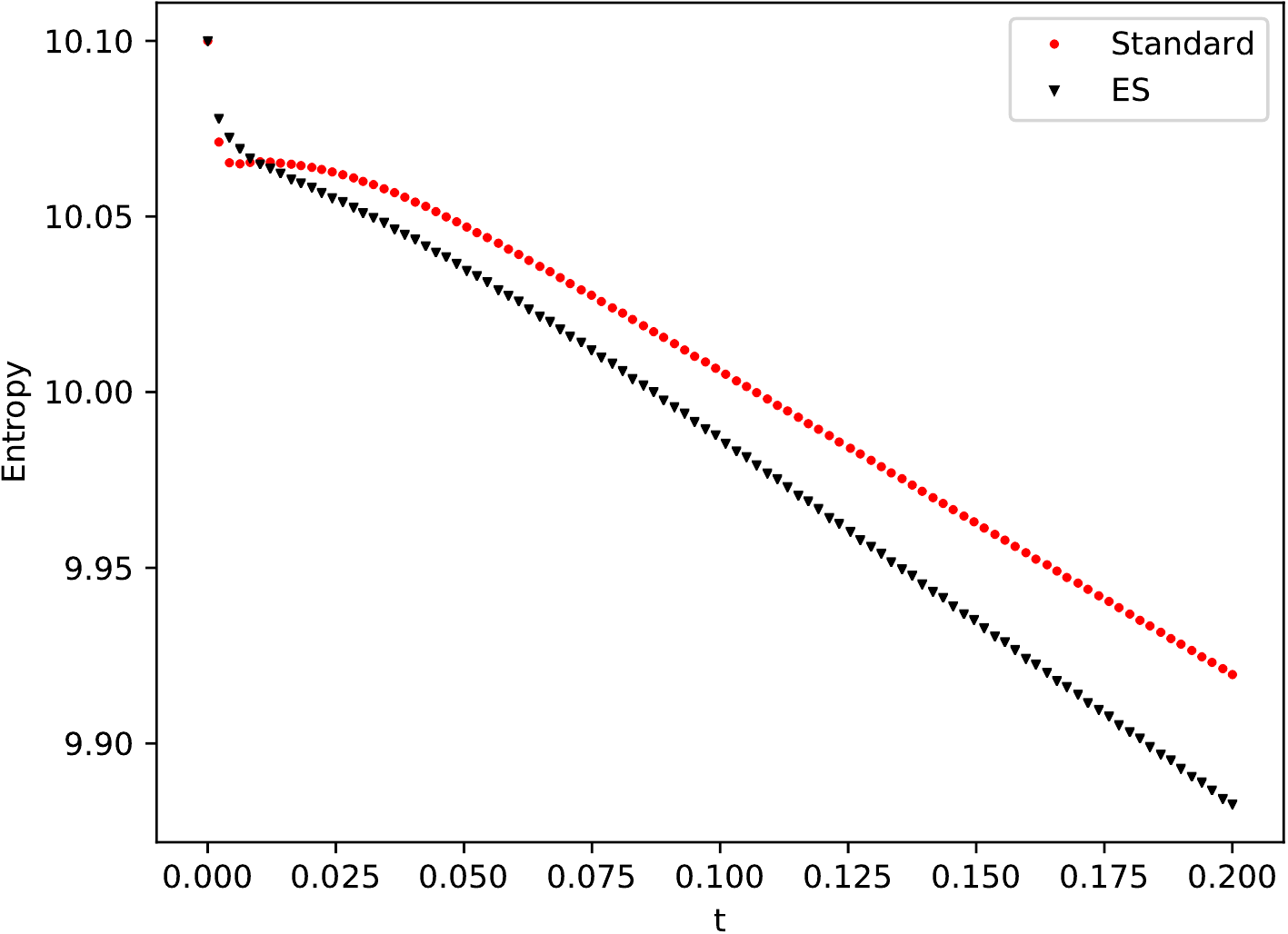}
    }
    \caption{Entropy glitch test case for ESDGSEM and standard DGSEM for $N=1$ at $T=0.2$ sliced at $y=0$ compared to a 1D reference solution from \cite{lukacova2009entropy}. The standard DGSEM produces an incorrect shock due to the unphysical entropy production.}
    \label{fig:EntropyGlitchN1}
\end{figure}

\subsubsection{Necessity of Positivity Limiter}
\label{subsubsec:PPnec}

First, we numerically verify the mass conservation and entropy stability of the ESDGSEM with artificial viscosity and positivity preservation in the presence of dry areas. The test case is a dam break problem with periodic boundary conditions at $y=\pm20$ and solid walls at $x=\pm20$ on the domain $\Omega=[-20,20]^2$ with initial conditions
\begin{equation}
\begin{aligned}
&h(x,y,0)= \left\{
\begin{aligned}
&10.0, \quad&\textrm{if } x < 0  \\
&0.0, \quad&\textrm{otherwise}
\end{aligned}
\right.,\\
&u=v=0.
\end{aligned}
\end{equation}
We use a polynomial order of $N=3$ and a uniform Cartesian mesh with $50\times 50$ elements. We set the viscosity coefficient to be $\epsilon_0=0.1$ and use a gravitational constant of $g=9.81$.
From the results in Figure \ref{fig:EntropyStabilityPP} we see that the entropy is monotonically decreasing as expected. The mass is conserved up to machine precision. This test case crashes immediately without the use of a positivity limiter.
\begin{figure}[!ht]
   \centering
    \subfloat[Water height $H$ at $y=0$, $T=1.0$]
    {
        \includegraphics[width=0.5\textwidth]{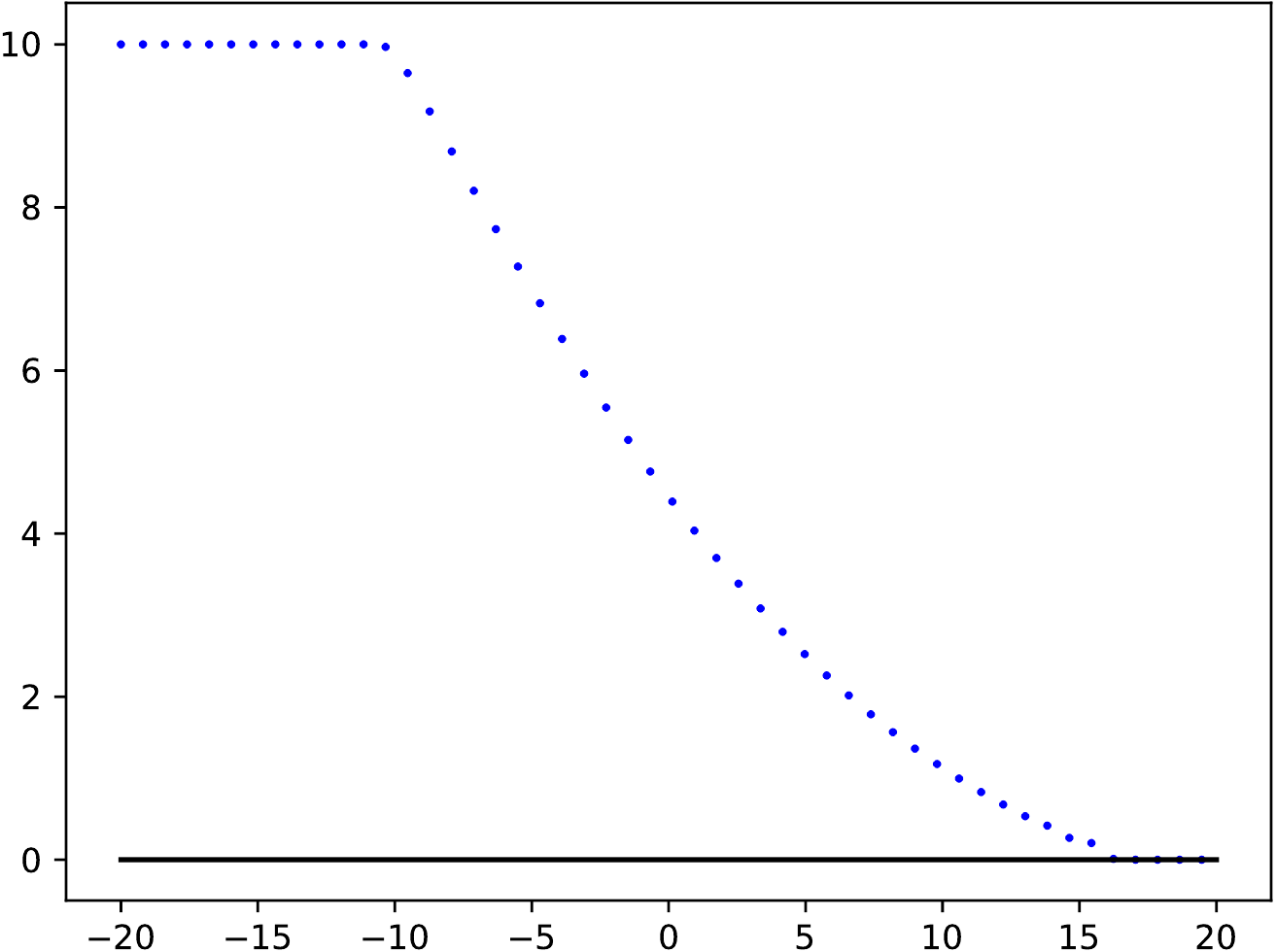}
    }
    \subfloat[Total Entropy over time]
    {
        \includegraphics[width=0.5\textwidth]{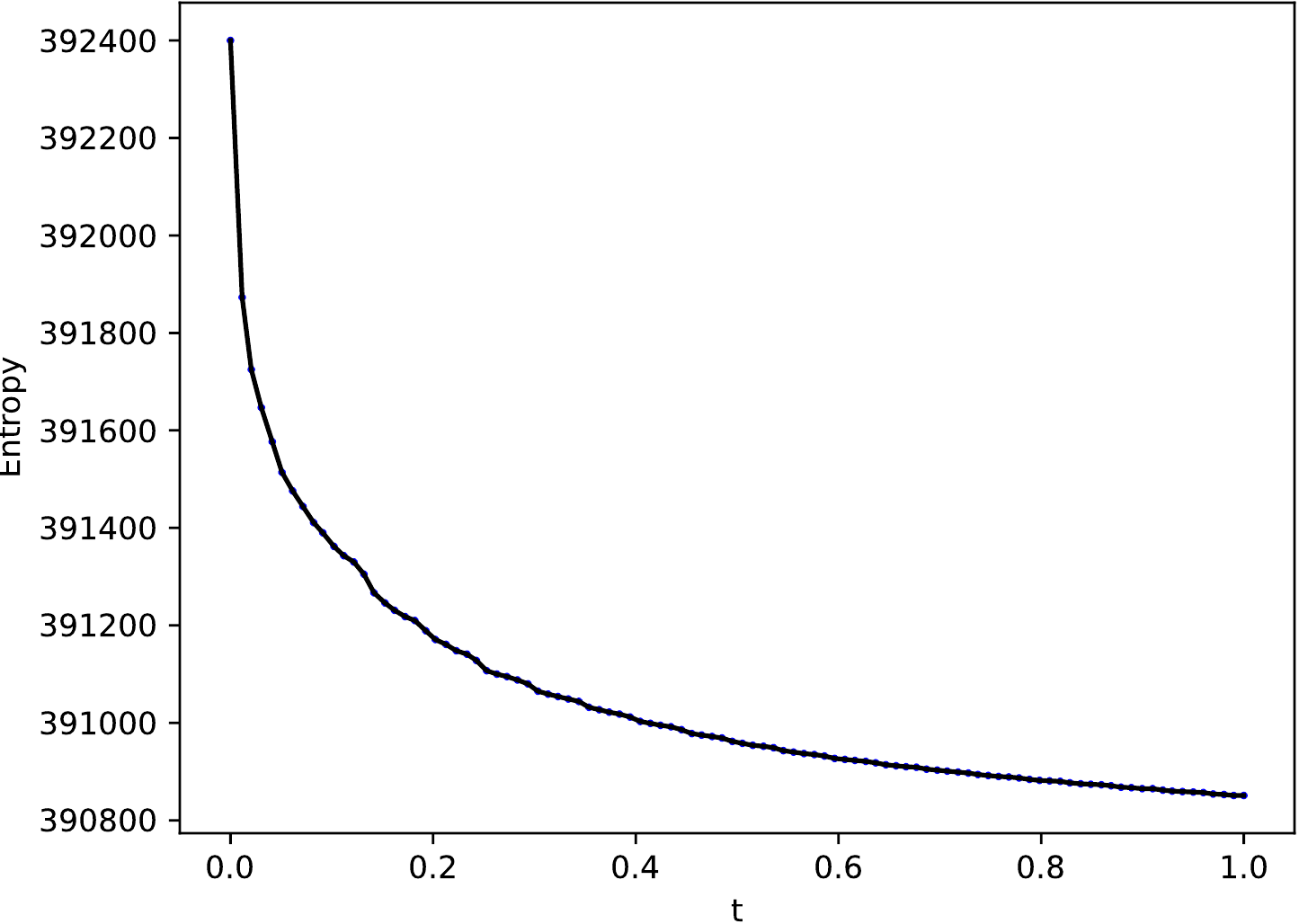}
    }
    \\
    \caption{Slice of total water height at final time and entropy evolution over time for a dam break problem with a dry zone approximated with the ESDGSEM with artificial viscosity and positivity limiter at polynomial order $N=3$.}
    \label{fig:EntropyStabilityPP}
\end{figure}

\FloatBarrier

\FloatBarrier

\subsection{Oscillating Lake}
The oscillating lake is a parabolic bowl partly covered with water that is moving around the center of the domain $\Omega = [-2,2]\times[-2,2]$ and defined by
\begin{equation}
\begin{aligned}
h(x,y,0)&= \max\left(0, \sigma \frac{h_0}{a^2} \left (2 x \cos{\omega t} + 2y \sin{\omega t} - \sigma ) + h_0 - b\right) \right), \\
u(x,y,0)&=-\sigma \omega \sin{\omega t} ,\\
v(x,y,0)&=\sigma \omega \cos{\omega t} .
\end{aligned}
\end{equation}
with parabolic bottom topography
\begin{equation}
\begin{aligned}
\label{ParabolicBowlBottom}
b(x,y) = h_0 \frac{x^2+y^2}{a^2},
\end{aligned}
\end{equation}
with parameters $ h_0  = 0.1$, $a = 1$ $\sigma  =0.5$ and $\omega = \frac{\sqrt{2g h_0}}{a}$. The boundary conditions can be set to solid walls as the water flow never reaches the domain boundaries. The base viscosity parameter is set to $\epsilon_0 = 0.01$. The gravitational constant is set to $g=9.81$. We use a uniform Cartesian mesh with $200\times 200$ elements for the computation.

The oscillating lake test case tests how well a numerical method is able to handle wetting and drying as the fluid evolves. There are no strong shocks so the water pond should travel smoothly around the center of the domain. We plot a slice through $y=0$ as well as the dynamic viscosity coefficient for the domain in Figure \ref{fig:OscillatingLake}. We see that viscosity is only applied on the edges of the wet circle with small magnitudes, as not much viscosity is necessary for this case.
\begin{figure}[!ht]
   \centering
    \subfloat[$H$-slice at $y=0$, $t=T/6.0$]
    {
        \includegraphics[scale=0.4]{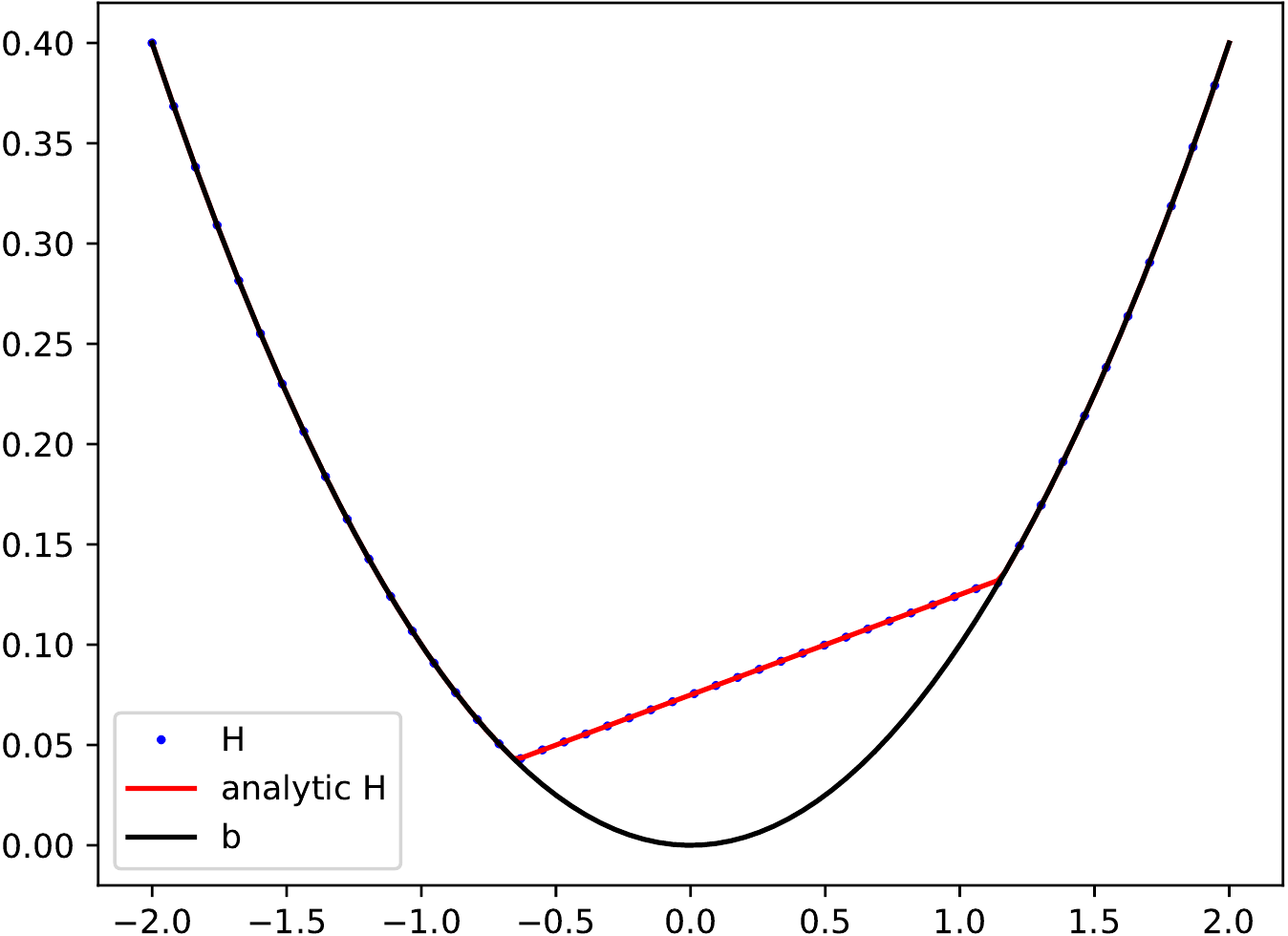}
    }
    \subfloat[Viscous coefficient $\epsilon$, $t=T/6.0$]
    {
        \includegraphics[scale=0.4]{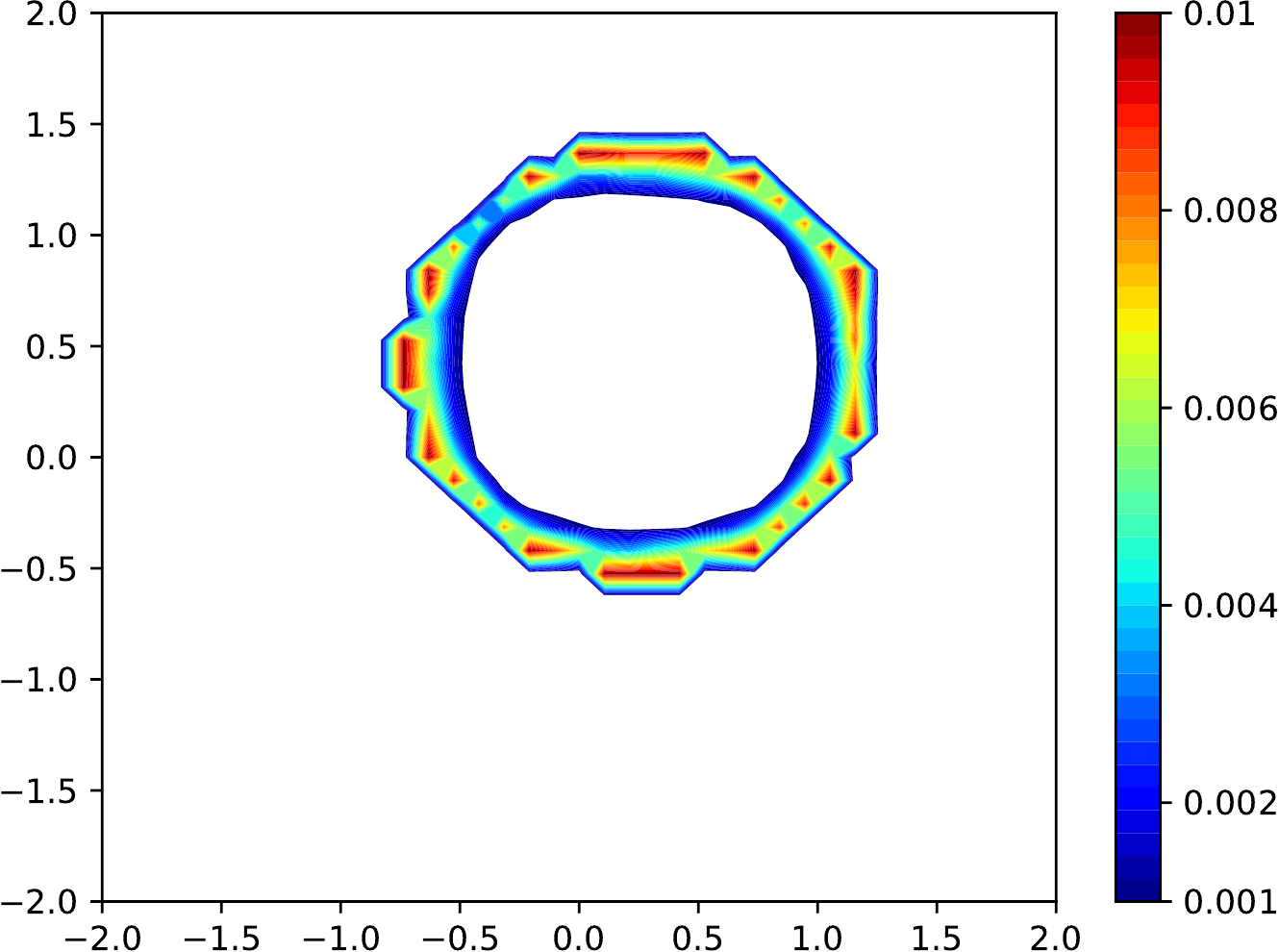}
    }
    \\
\subfloat[$H$-slice at $y=0$, $t=T/3.0$]
    {
        \includegraphics[scale=0.4]{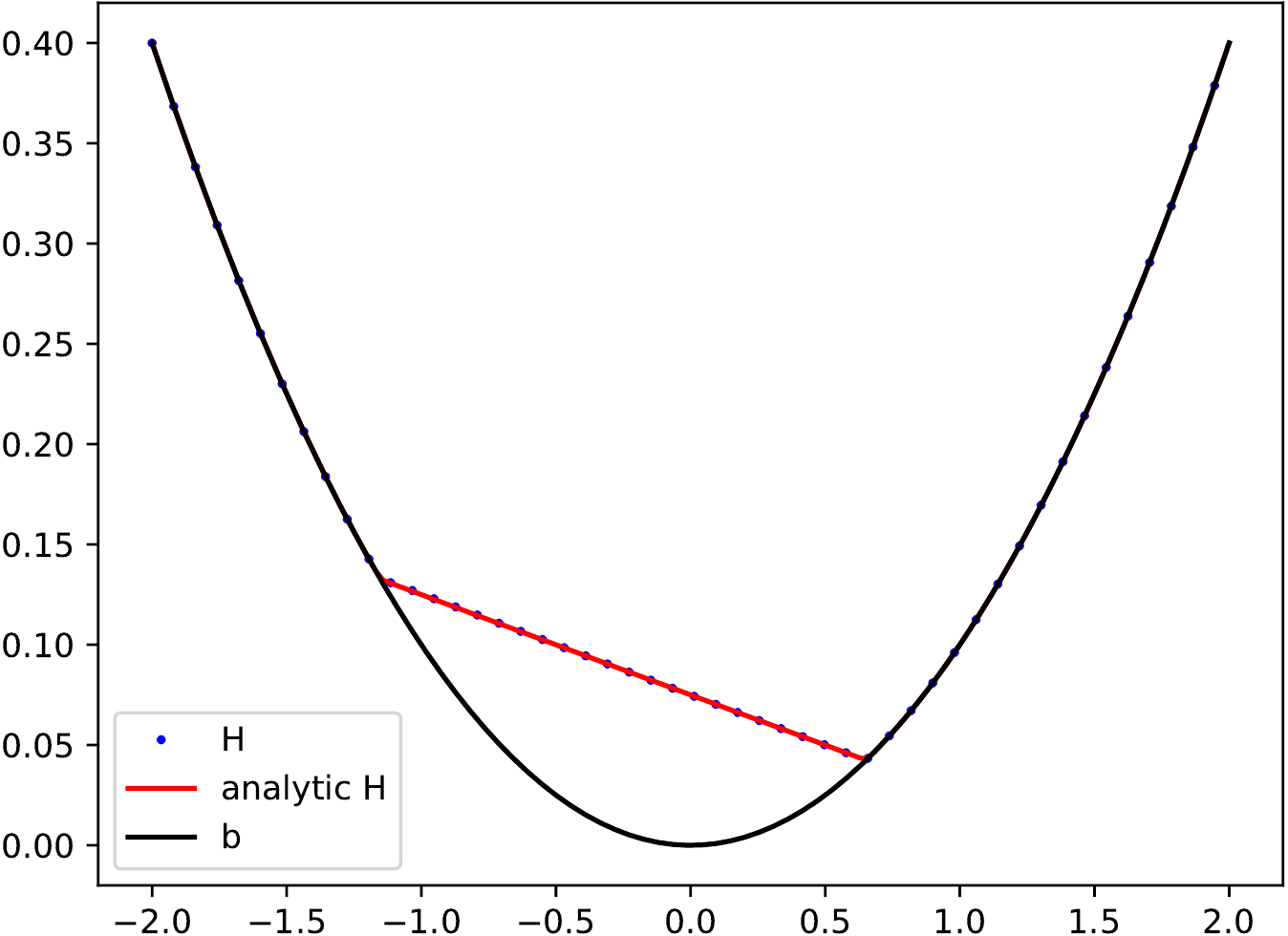}
    }
\subfloat[Viscous coefficient $\epsilon$, $t=T/3.0$]
    {
        \includegraphics[scale=0.4]{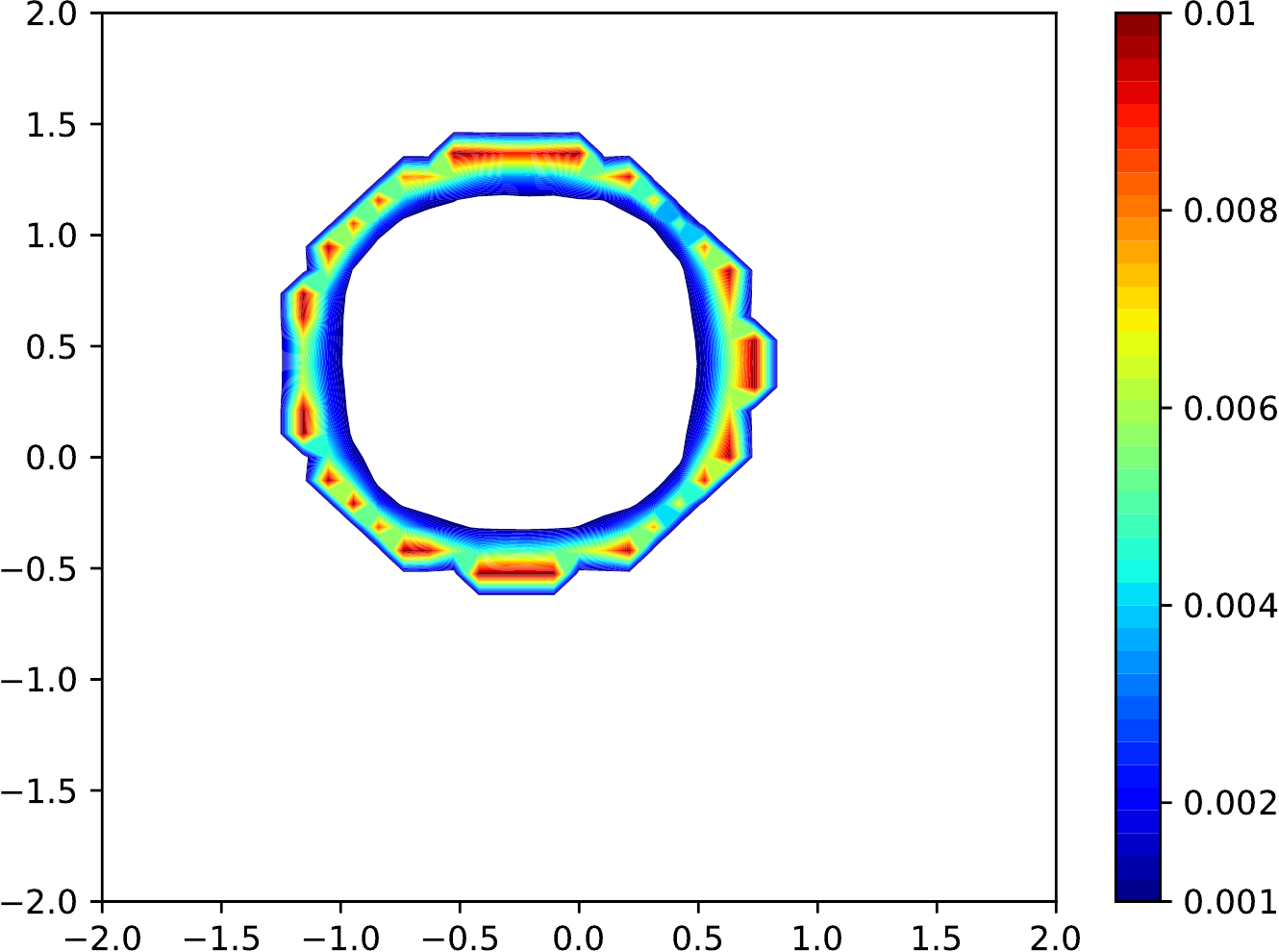}
    }
    \\
\subfloat[$H$-slice at $y=0$, $t=T/2.0$]
    {
        \includegraphics[scale=0.4]{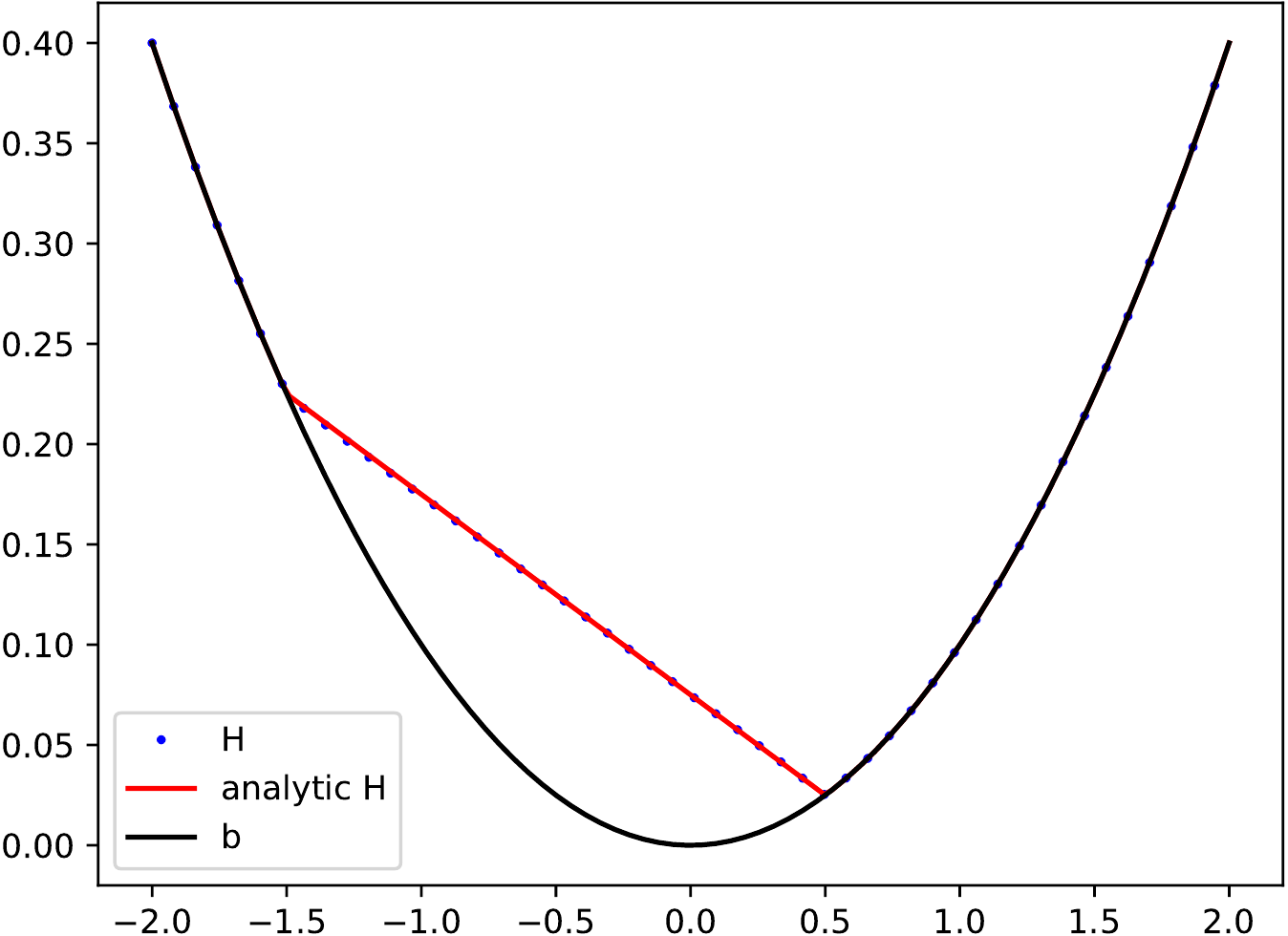}
    }
    \subfloat[Viscous coefficient $\epsilon$, $t=T/2.0$]
    {
        \includegraphics[scale=0.4]{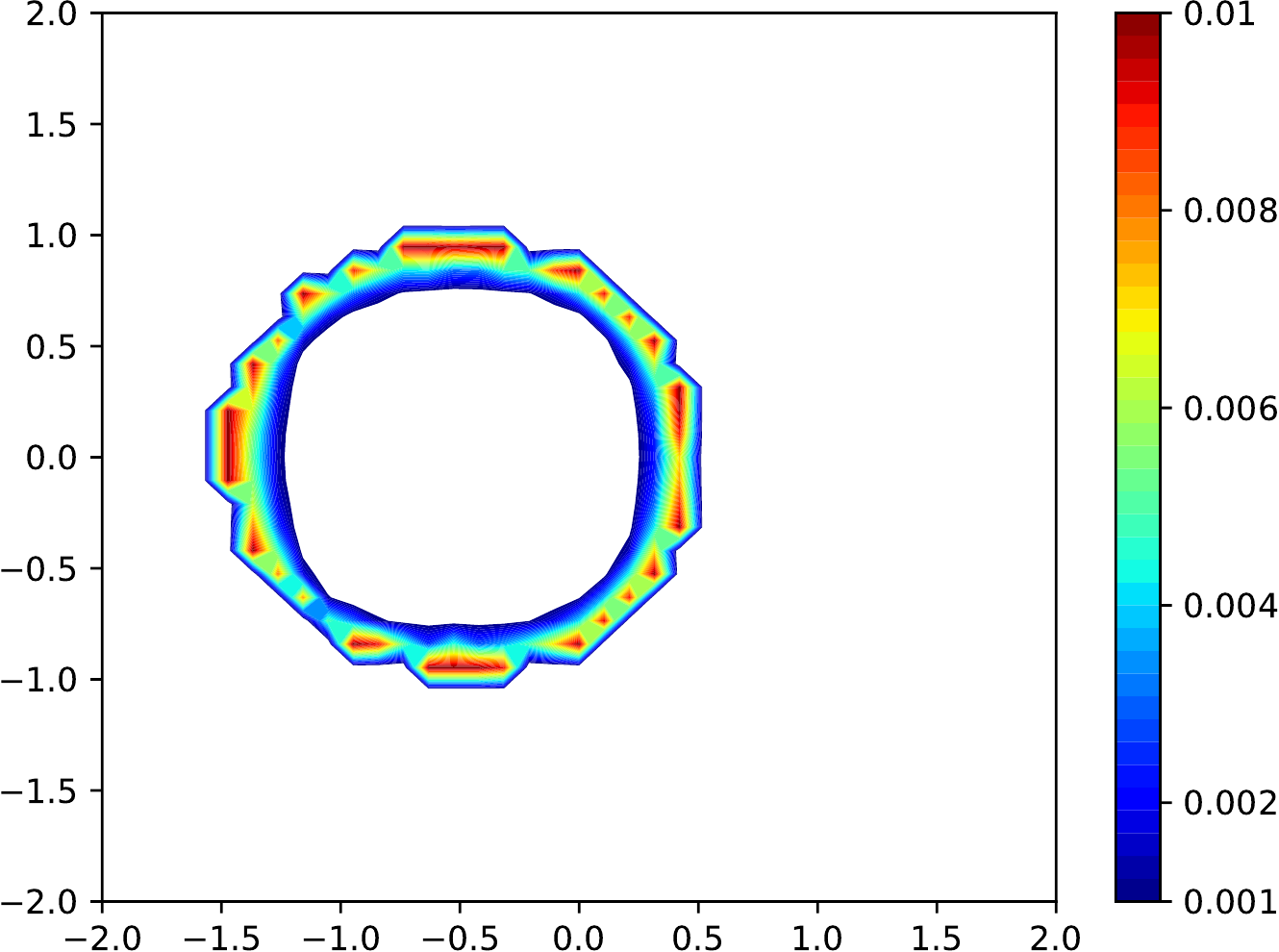}
    }
\\
\subfloat[$H$-slice at $y=0$, $t=2T$]
    {
        \includegraphics[scale=0.4]{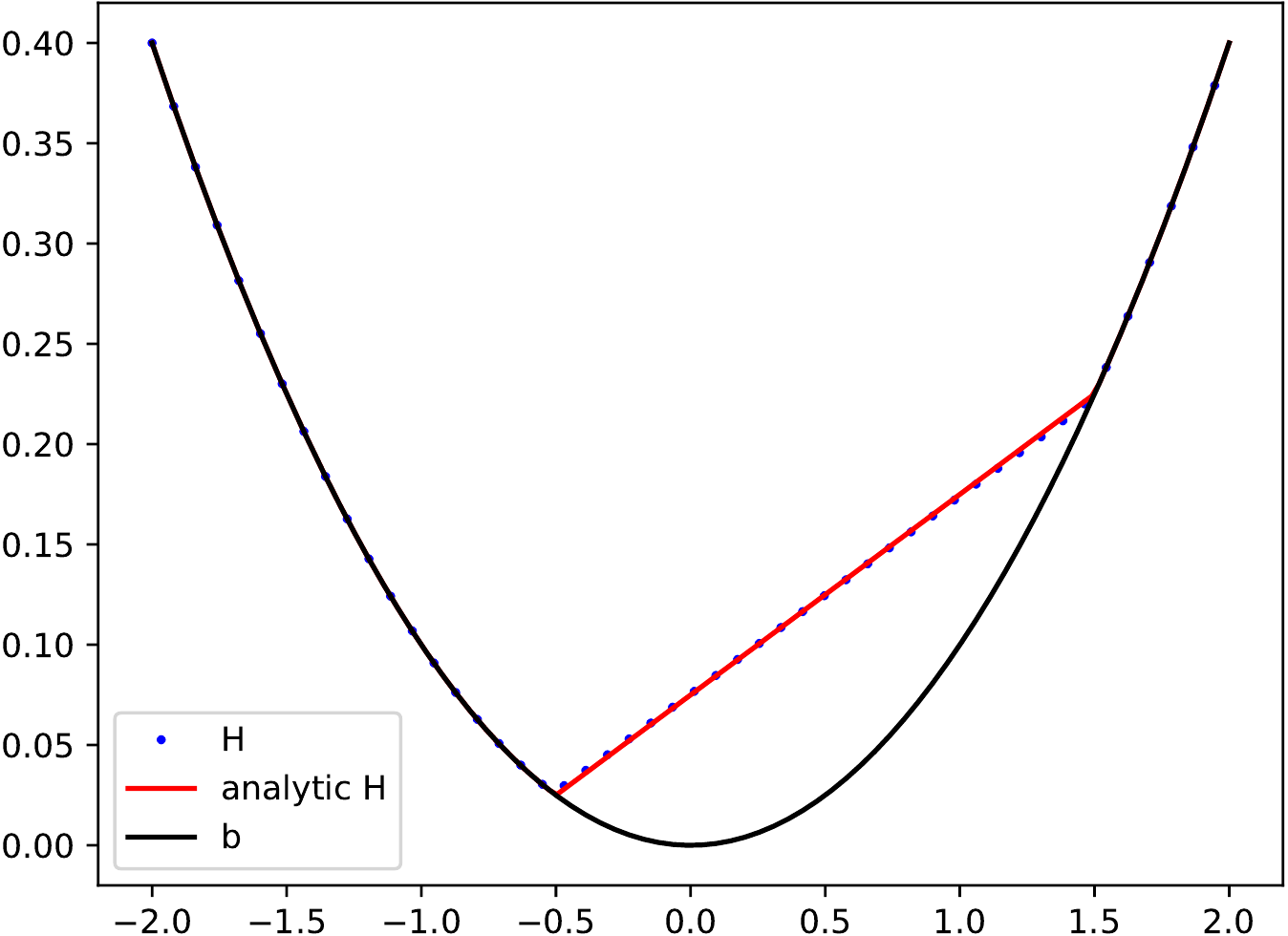}
    }
        \subfloat[Viscous coefficient $\epsilon$, $t=2T$]
    {
        \includegraphics[scale=0.4]{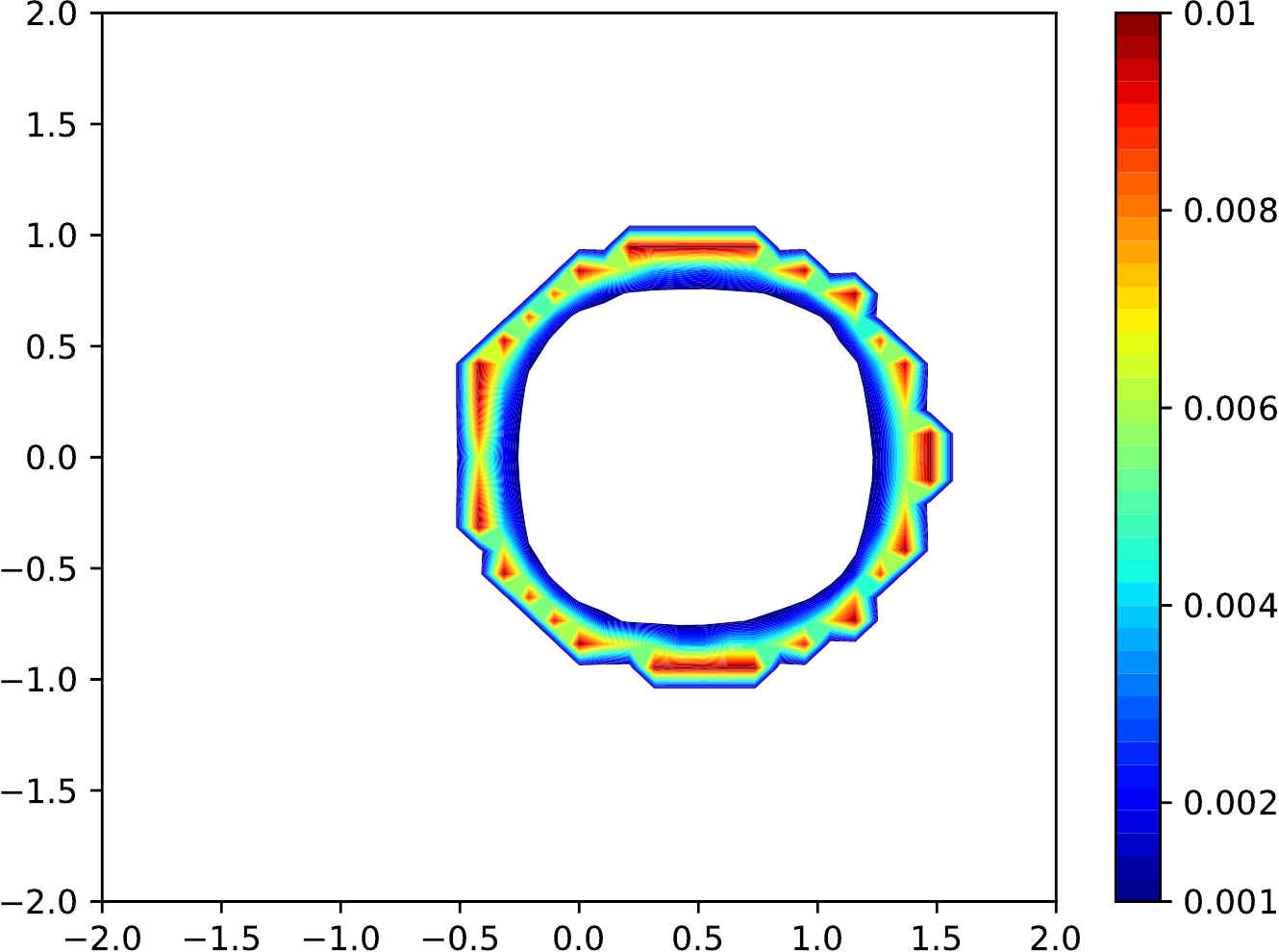}
    }
    \caption{ESDGSEM approximation with artificial viscosity and positivity limiter for the 2D oscillating lake at $N=3$.}
    \label{fig:OscillatingLake}
\end{figure}

\FloatBarrier

\subsection{Three Mound Dam Break}
We procceed with a more challenging test case to thoroughly test the shock capturing capabilities as well as the positivity preservation. A dam break is set up on the domain $\Omega = [0,75]\times[0,45]$ at $x=16$ with a water height of $h=1.875$ on the top and zero at the bottom. On the dry side of the dam are three mounds which will be partially flooded during the computation.
The water is initialized at rest 
\begin{equation}
\begin{aligned}
&h(x,y,0)= \left\{
\begin{aligned}
&1.875, \quad&\textrm{if } x < 16  \\
&0, \quad&\textrm{otherwise}
\end{aligned}
\right.,\\
&u=v=0.
\end{aligned}
\end{equation}
The three mounds on the down hill side are defined by
\begin{equation}
\begin{aligned}
M_1(x,y) &= 1 - 0.1 \sqrt{(x-30)^2 + (y-22.5)^2}, \\
M_2(x,y) &= 1 - 0.1 \sqrt{(x-30)^2 + (y-7.5)^2}, \\
M_3(x,y) &= 2.8 - 0.28 \sqrt{(x-47.5)^2 + (y-15)^2}, \\
\end{aligned}
\end{equation}
and the bottom topography is taken as the maximum elevation level
\begin{equation}
\begin{aligned}
\label{ThreeMoundBottom}
b(x,y) = \max (0, M_1(x,y),M_2(x,y),M_3(x,y))
\end{aligned}
\end{equation}
\begin{figure}[!ht]
   \centering
    \subfloat[$H$, $T=0.0$]
    {
        \includegraphics[width=0.5\textwidth]{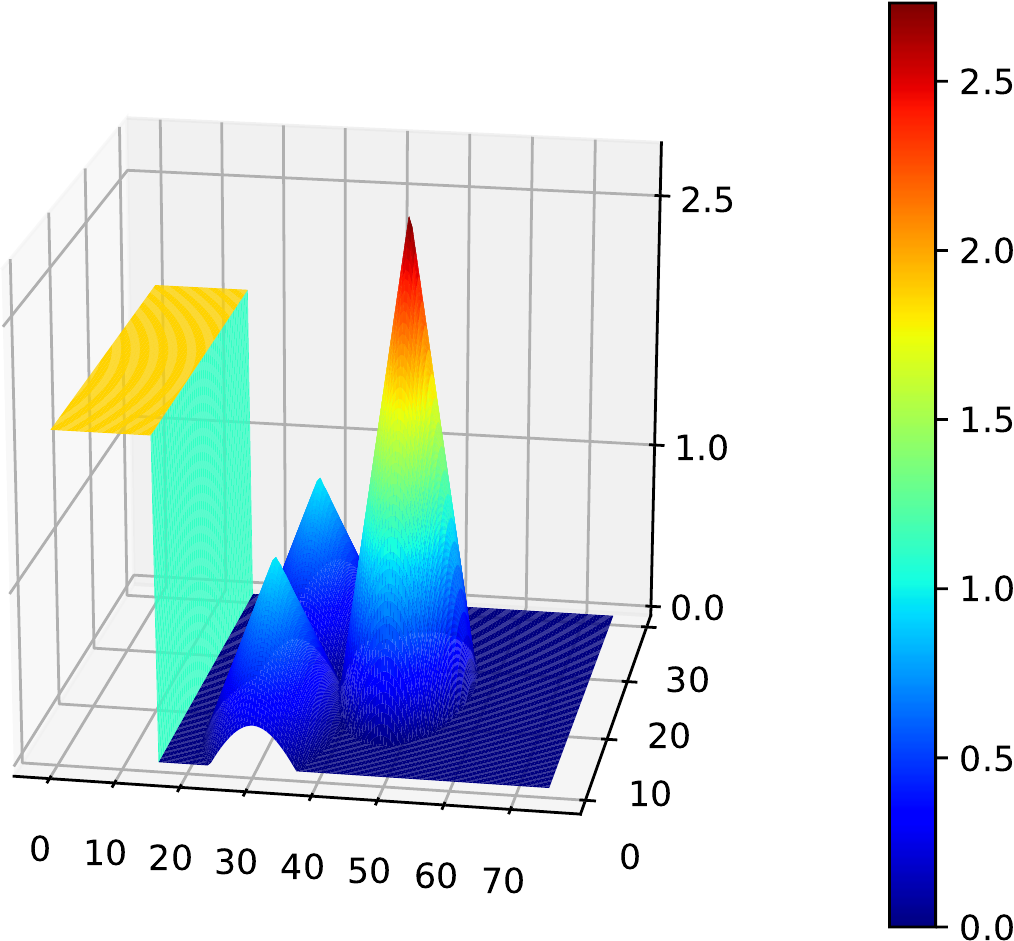}
    }
     \caption{Initial condition for the dam break over three mounds $N=3$.}
    \label{fig:ThreeMoundDamBreak1}
\end{figure}       
The set up uses solid wall boundary conditions on all four sides. We use a Cartesian mesh with $150\times 100$ elements. Due to the combination of a strong shock and a dry area including varying bottom elevations, the base viscosity parameter is set to be $\epsilon_0 = 0.2$. The gravitational constant is set to $g=9.81$. We show the total water height at various times in Figure \ref{fig:ThreeMoundDamBreak1} and Figure \ref{fig:ThreeMoundDamBreak2}. Also included in these figures is the dynamic viscous coefficient $\epsilon$. These plots show that the shock capturing mechanism accurately tracks the shock fronts across the domain and around the three mounds.
\begin{figure}[!ht]
   \centering
    \subfloat[$H$, $T=5.0$]
    {
        \includegraphics[width=0.5\textwidth]{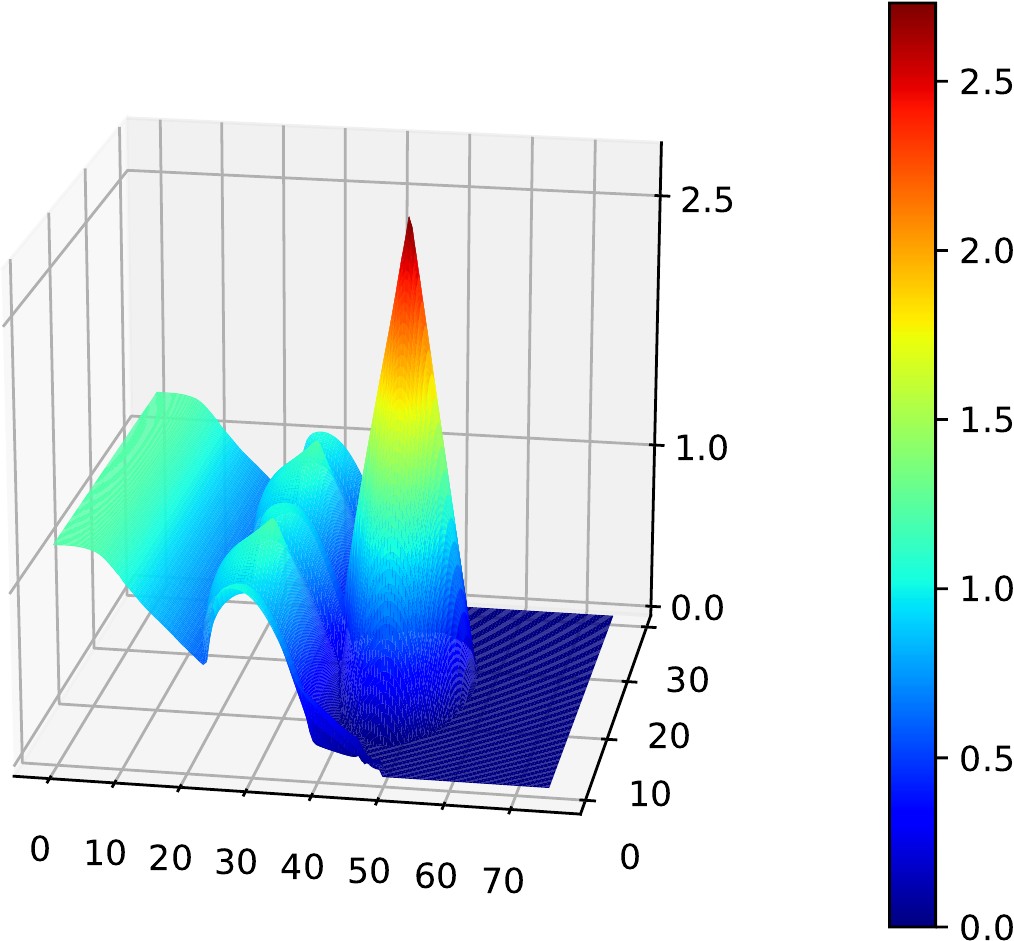}
    }
    \subfloat[Viscous coefficient $\epsilon$, $T=5.0$]
    {
        \includegraphics[width=0.5\textwidth]{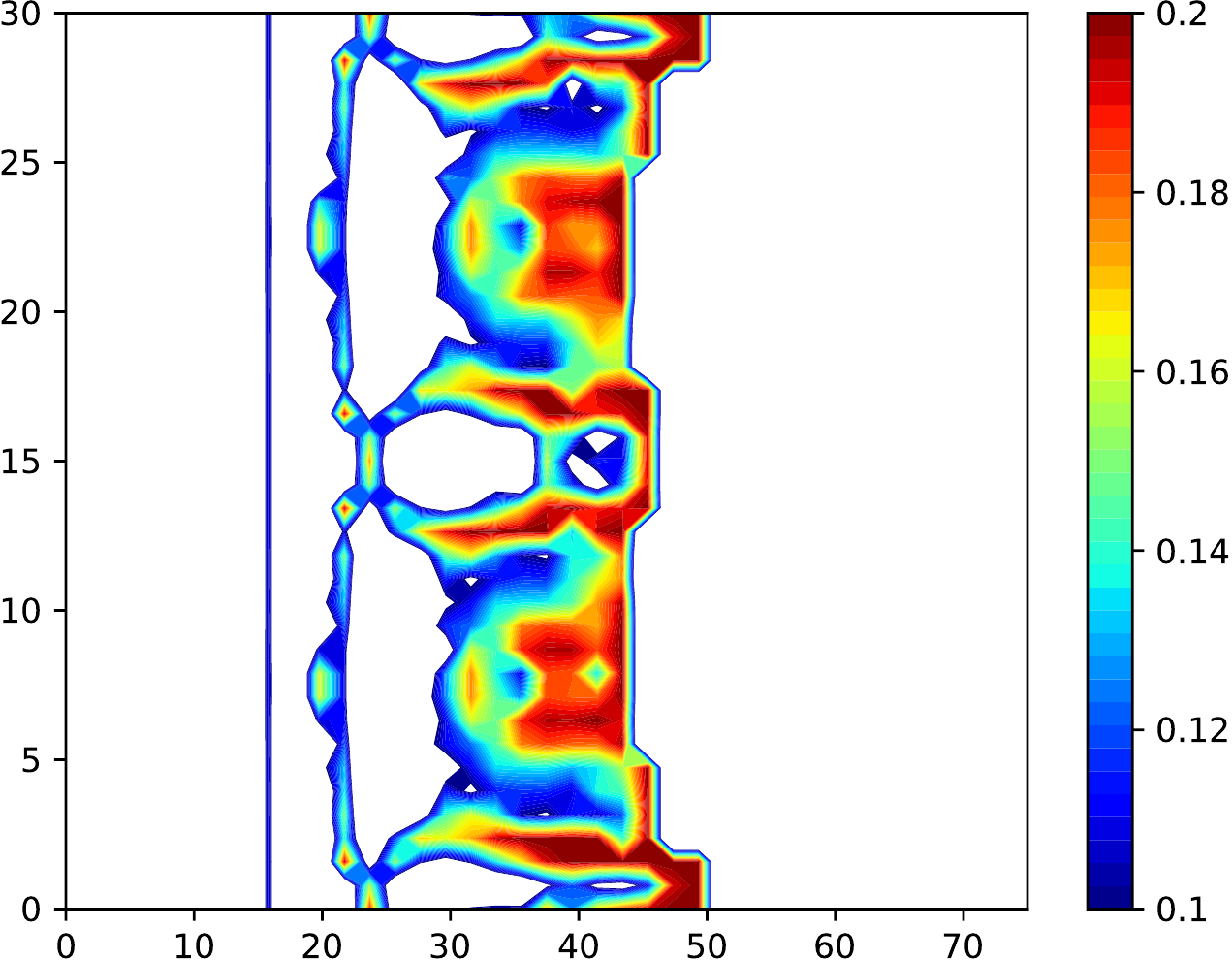}
    }
    \\
     \subfloat[$H$, $T=10.0$]
    {
    \includegraphics[width=0.5\textwidth]{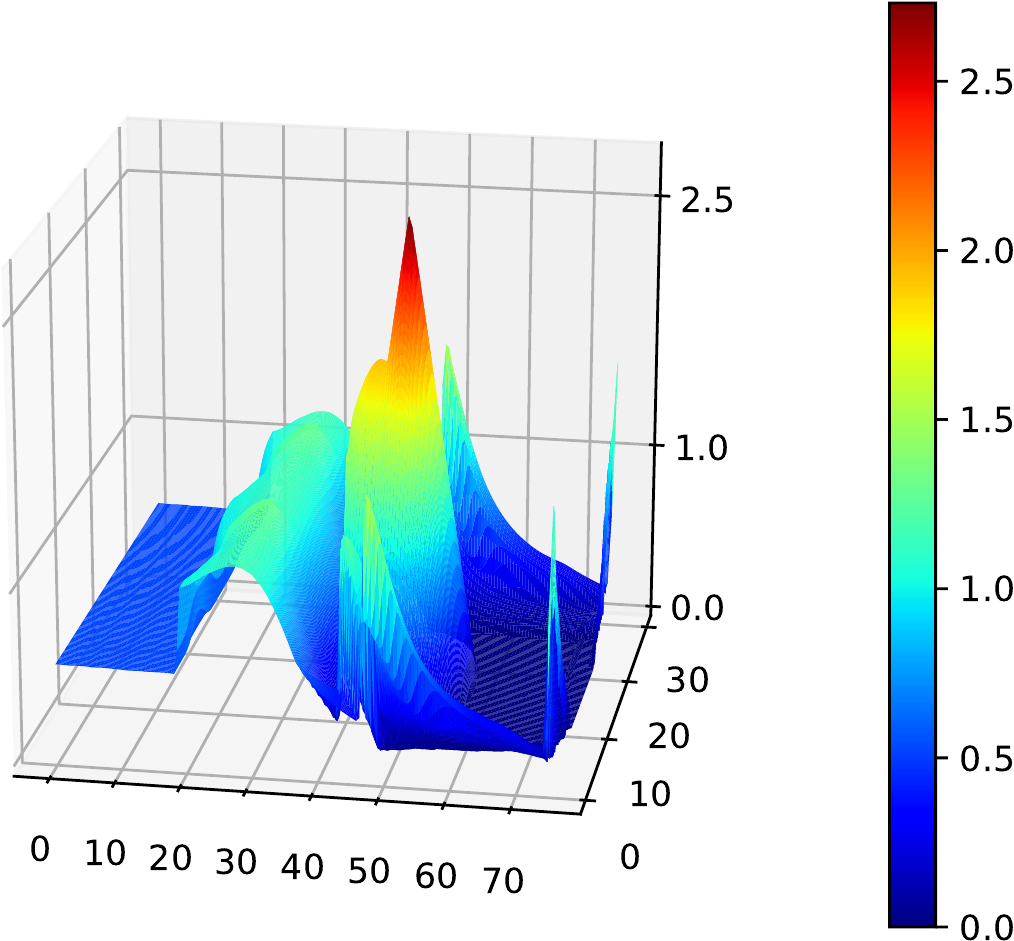}
    }
        \subfloat[Viscous coefficient $\epsilon$, $T=10.0$]
    {
        \includegraphics[width=0.5\textwidth]{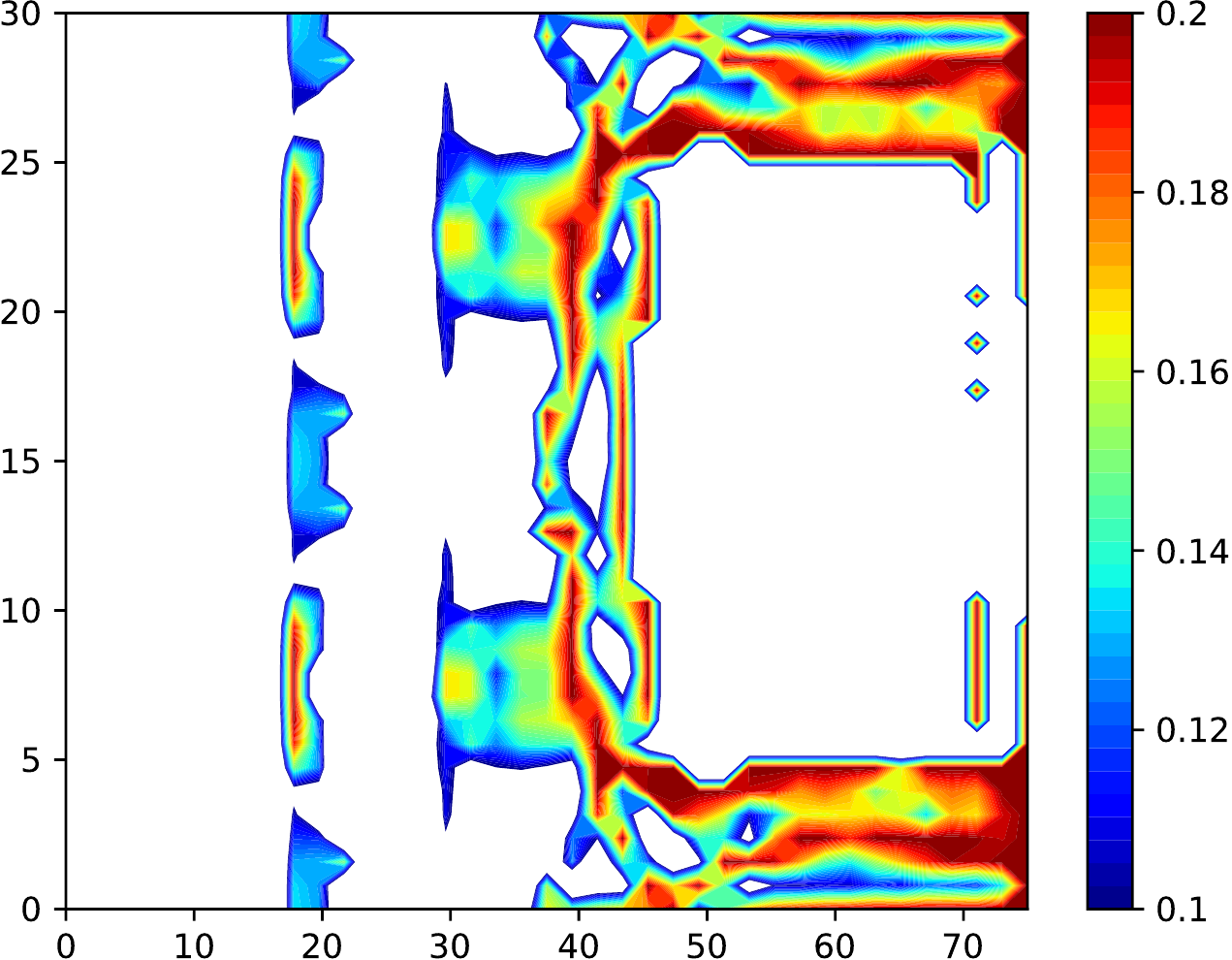}
    }
     \caption{ESDGSEM approximation with artificial viscosity and positivity limiter for the dam break over three mounds $N=3$.}
    \label{fig:ThreeMoundDamBreak2}
\end{figure}   
    \begin{figure}[!ht]
   \centering
\subfloat[$H$, $T=20.0$]
    {
        \includegraphics[width=0.5\textwidth]{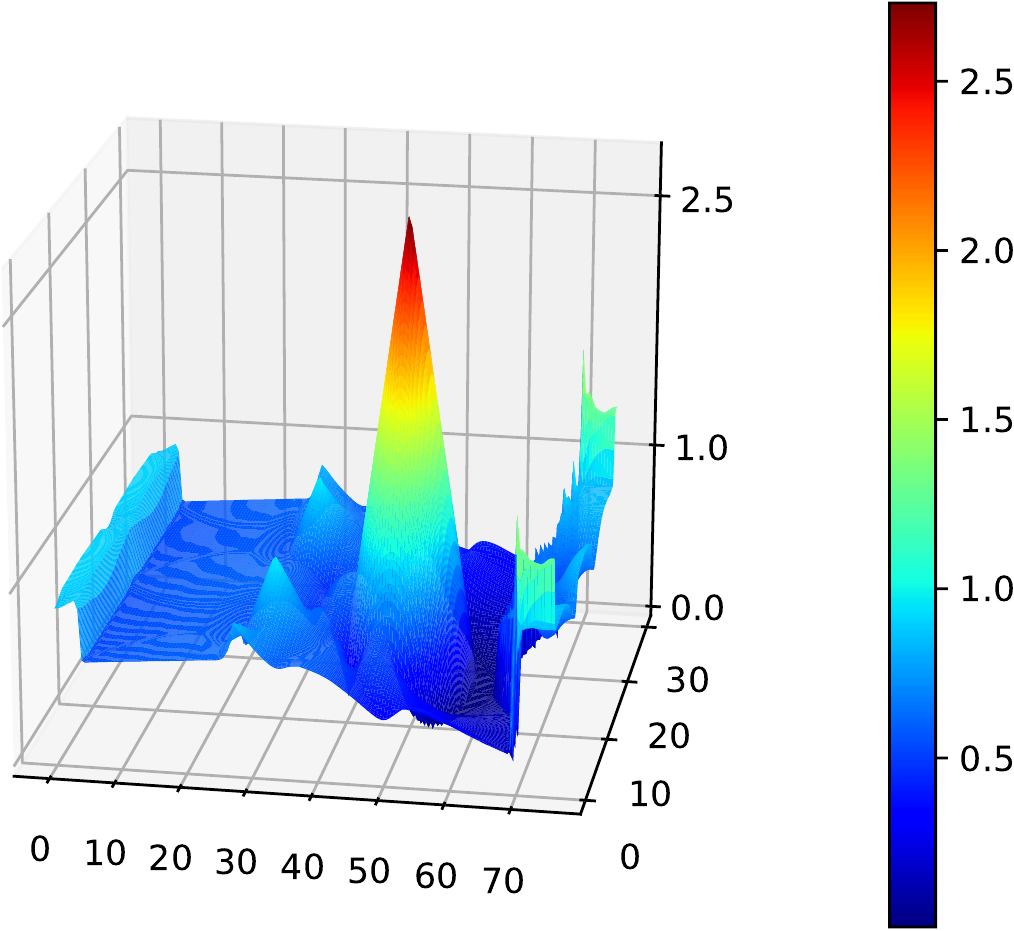}
    }
        \subfloat[Viscous coefficient $\epsilon$, $T=20.0$]
    {
        \includegraphics[width=0.5\textwidth]{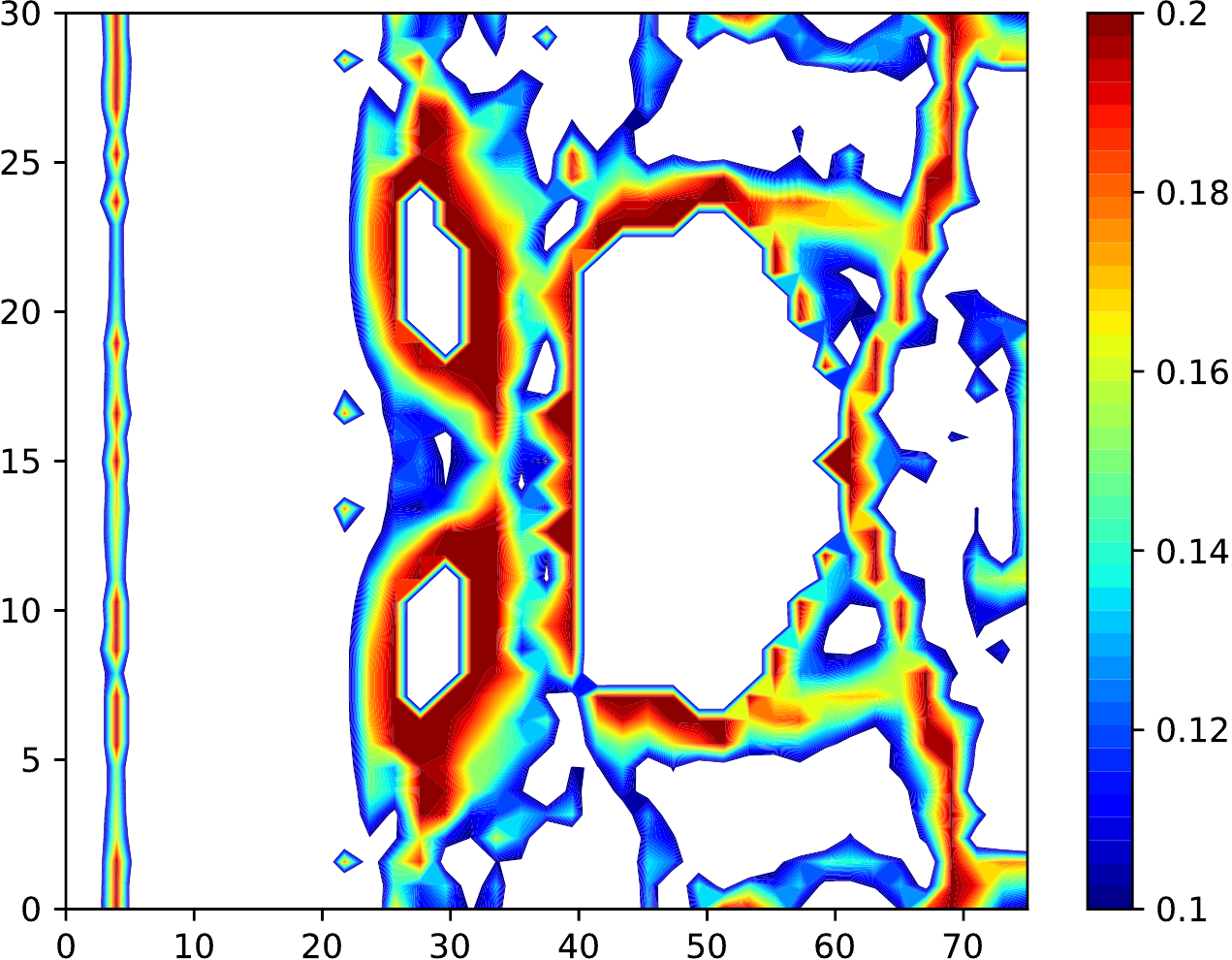}
    }
    \\
    \subfloat[$H$, $T=30.0$]
    {
        \includegraphics[width=0.5\textwidth]{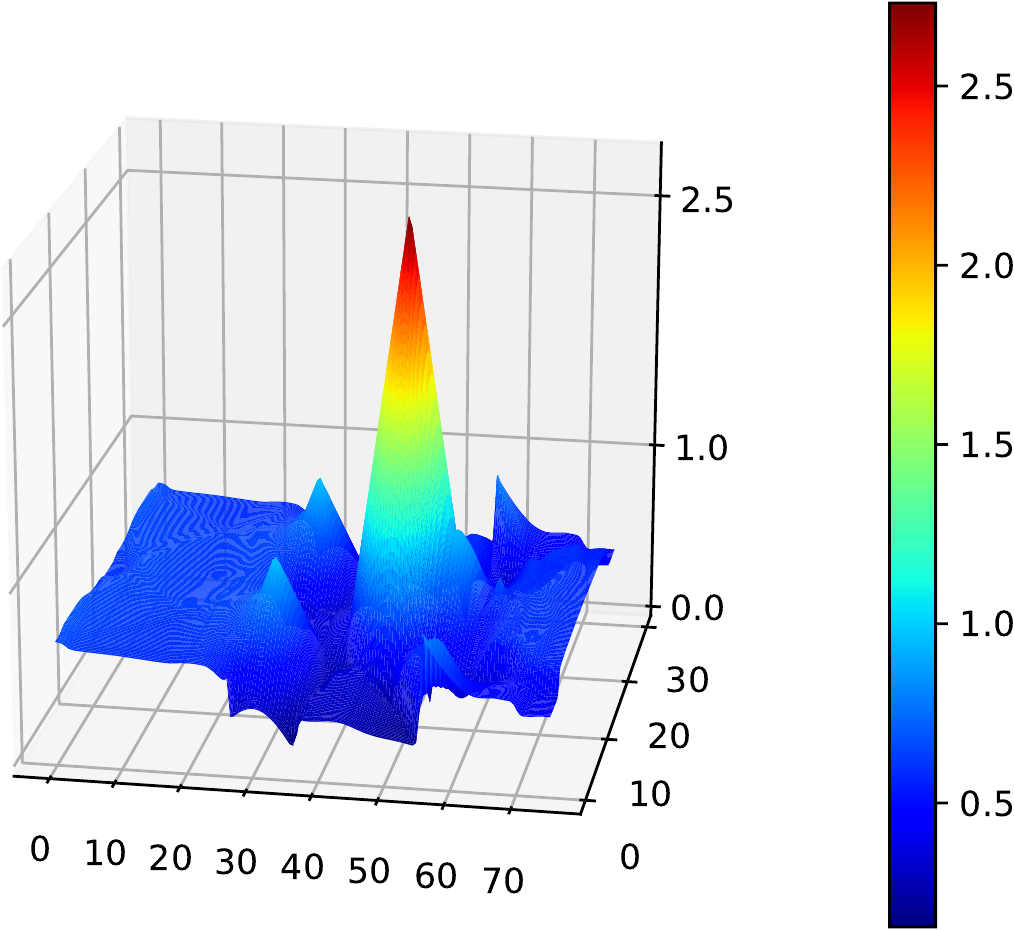}
    }
            \subfloat[Viscous coefficient $\epsilon$, $T=30.0$]
    {
        \includegraphics[width=0.5\textwidth]{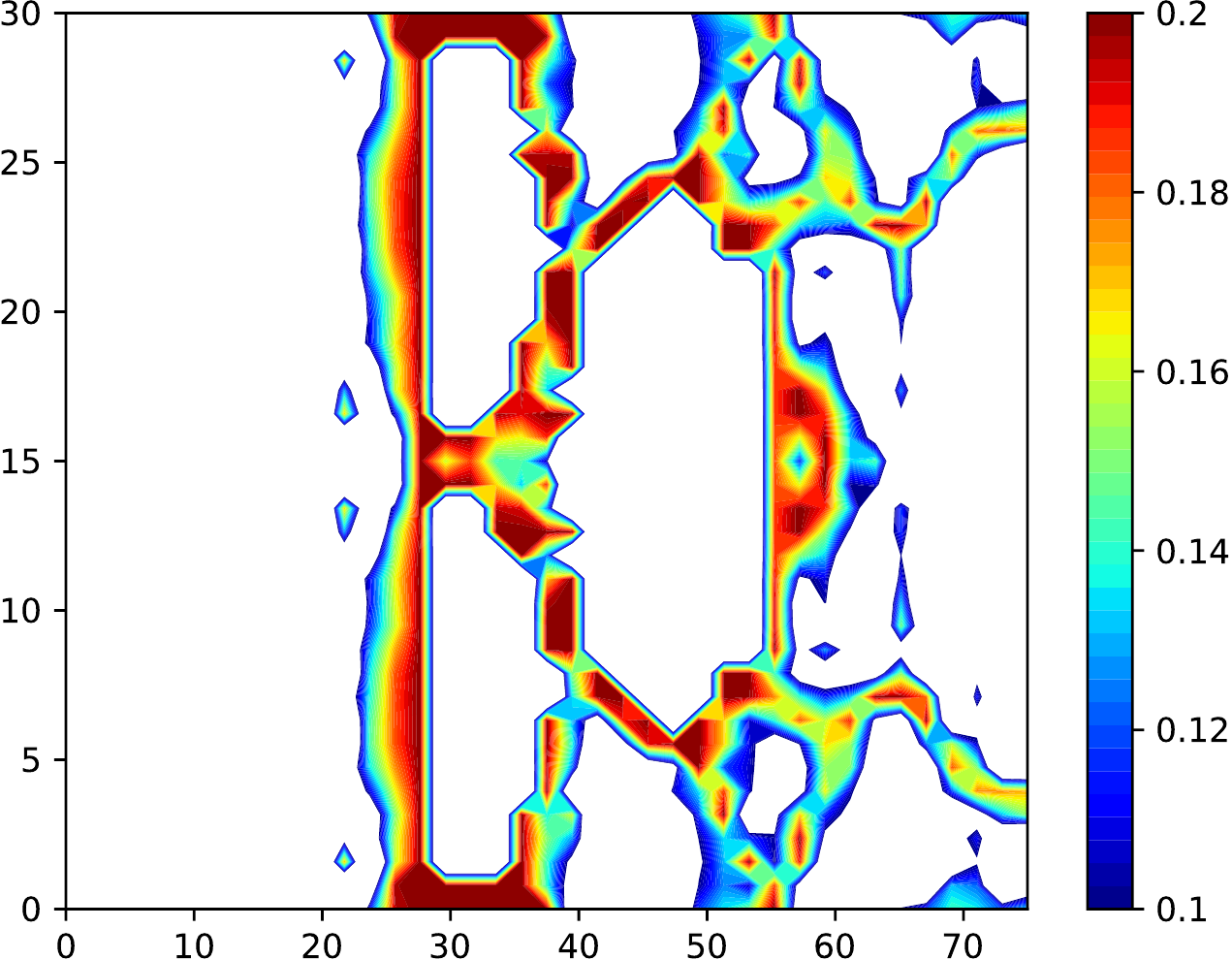}
    }
    \\
    \caption{ESDGSEM approximation with artificial viscosity and positivity limiter for the dam break over three mounds $N=3$.}
    \label{fig:ThreeMoundDamBreak3}
\end{figure}

    \begin{figure}[!ht]
   \centering
\subfloat[$H$, $T=40.0$]
    {
        \includegraphics[width=0.5\textwidth]{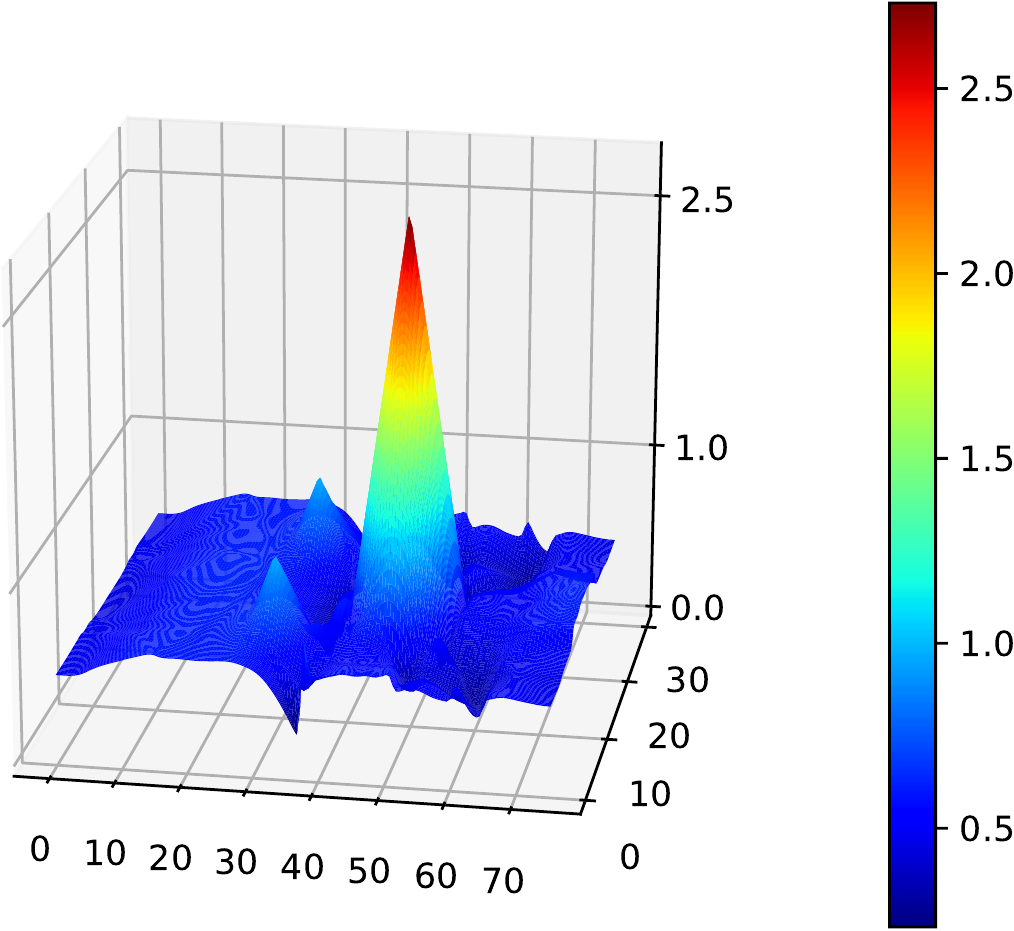}
    }
                \subfloat[Viscous coefficient $\epsilon$, $T=40.0$]
    {
        \includegraphics[width=0.5\textwidth]{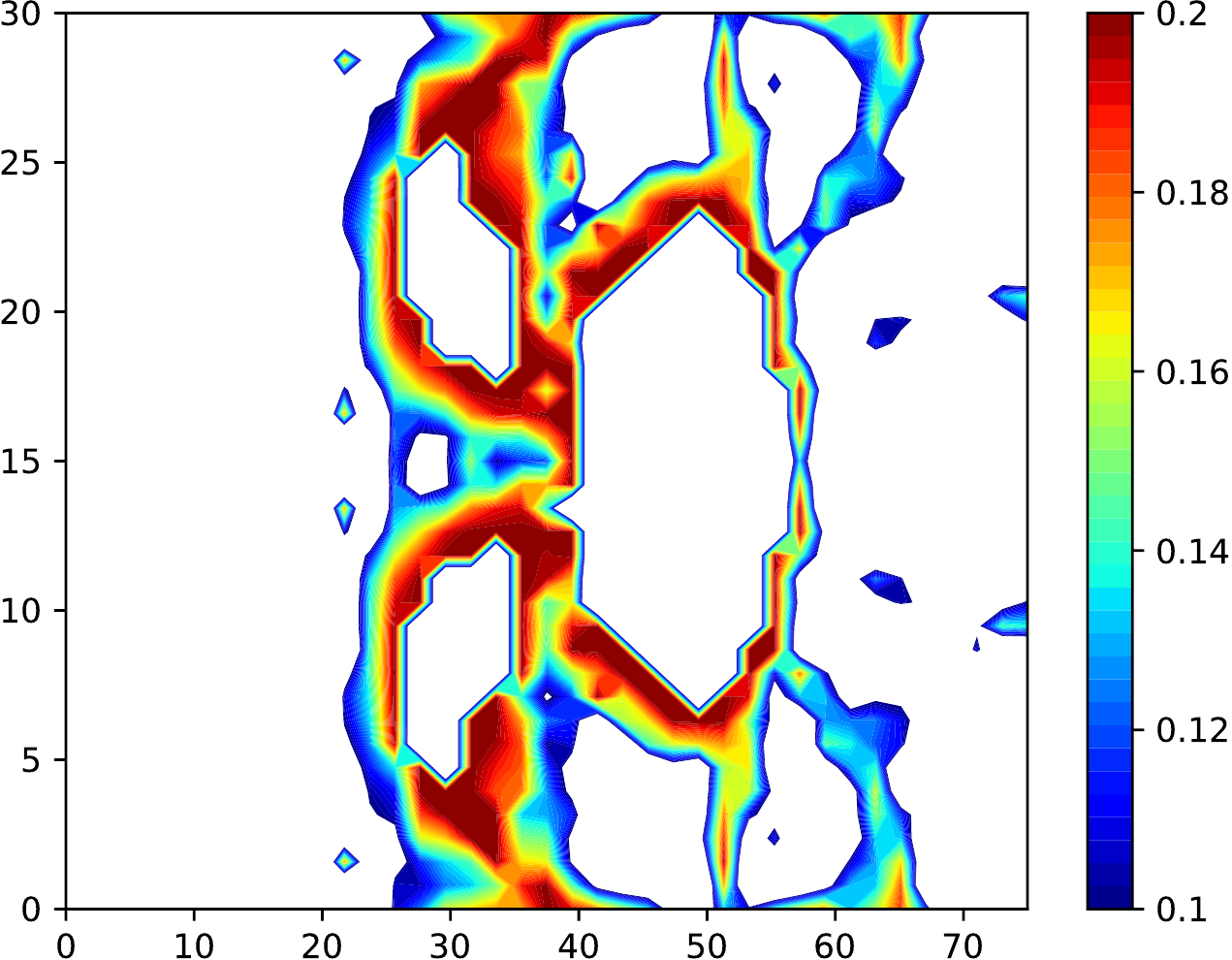}
    }\\
    \subfloat[$H$, $T=50.0$]
    {
        \includegraphics[width=0.5\textwidth]{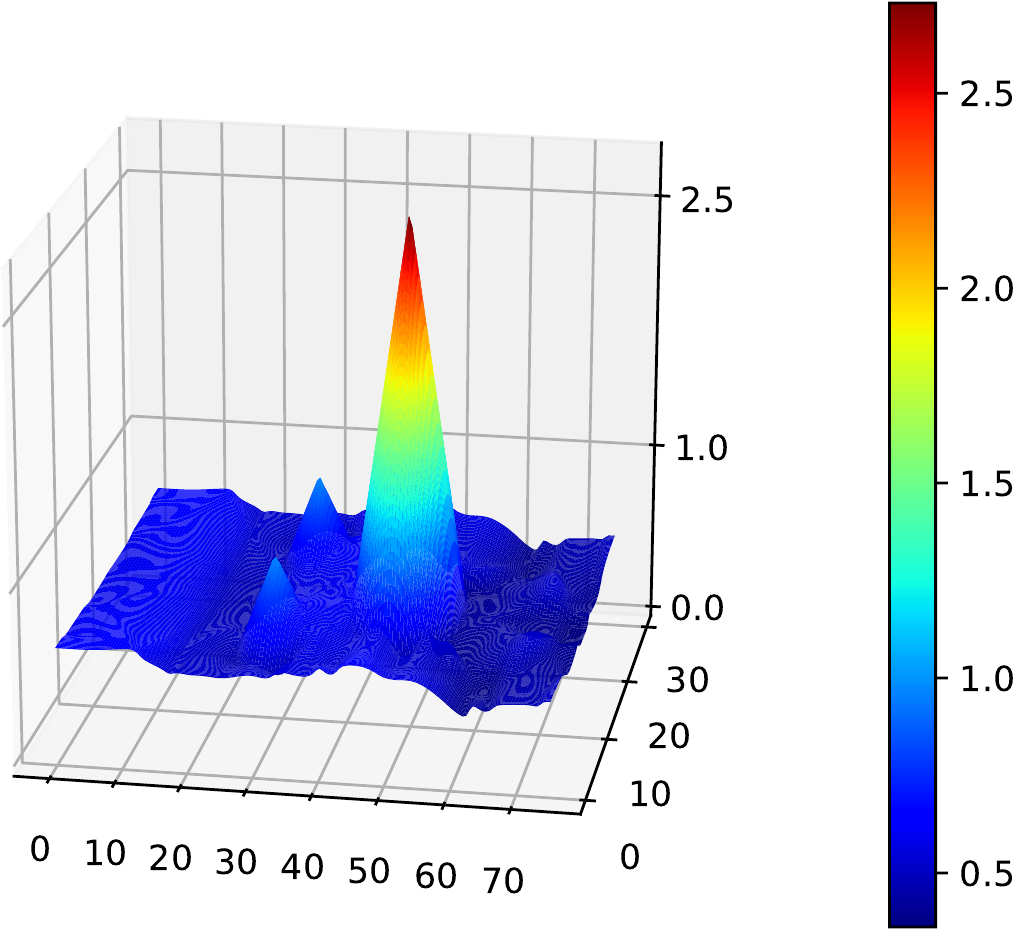}
    }
                \subfloat[Viscous coefficient $\epsilon$, $T=50.0$]
    {
        \includegraphics[width=0.5\textwidth]{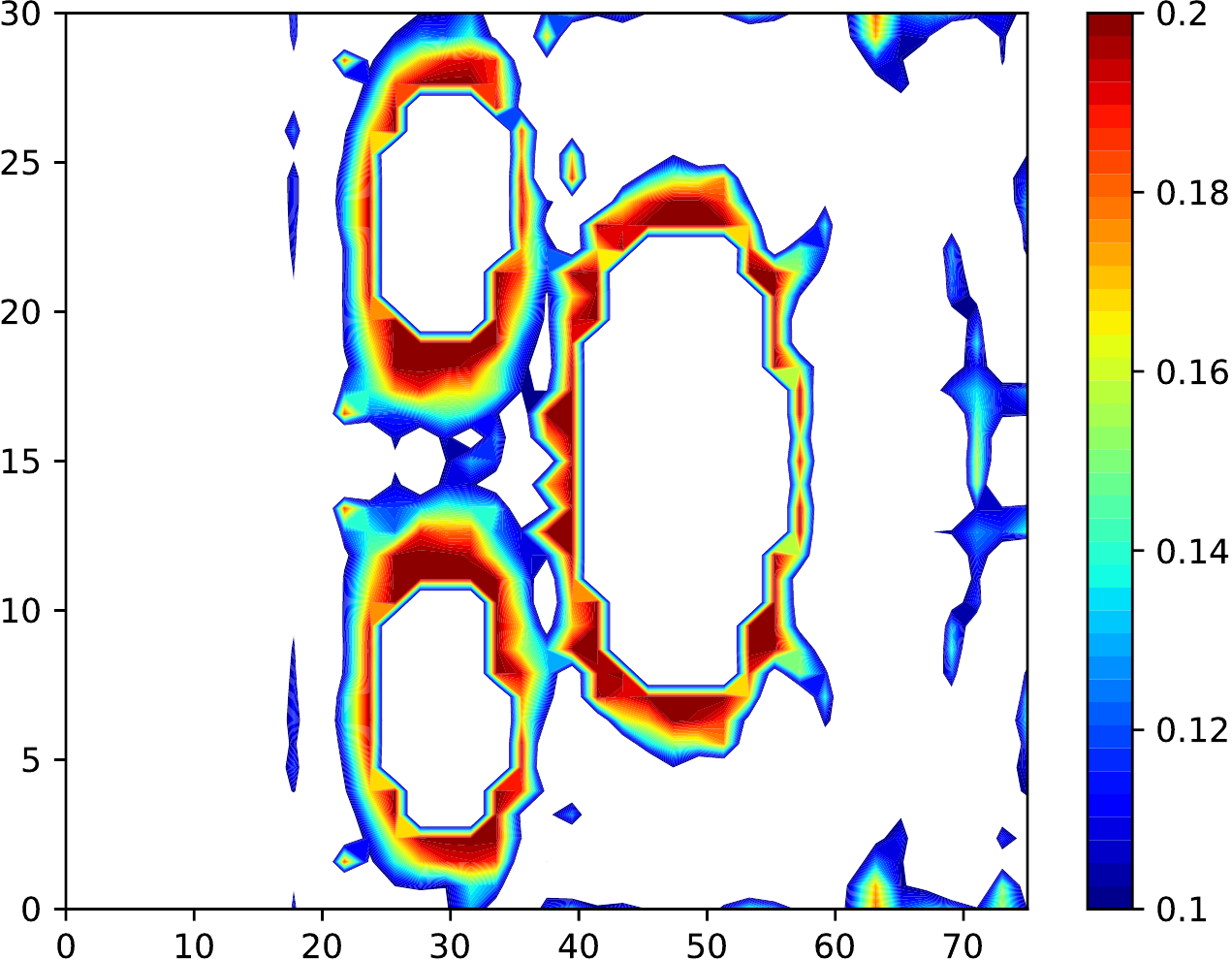}
    }
    \caption{ESDGSEM approximation with artificial viscosity and positivity limiter for the dam break over three mounds $N=3$.}
    \label{fig:ThreeMoundDamBreak4}
\end{figure}

\FloatBarrier

\subsection{Solitary wave runup on a conical island}
We examine the runup of a wave on the domain $\Omega = [0,25] \times [0,30]$ with a partly dry conical island in the center of the domain. The wave flows around the island, is reflected at the far end and flows back around it. This test case was previously studied numerically in \cite{Marras_GalerkinViscSW,synolakis1987runup} and experimentally in \cite{briggs1995laboratory}.
The initial wave $\eta$ defined by
\begin{equation}
    \eta(x,y,0) = \frac{A}{h_0}  \text{sech}^2 (\gamma (x-x_c)),
\end{equation} 
is set on top a flat water level of $h_0=0.32$, leading to initial conditions of
\begin{equation}
\begin{aligned}
&h(x,y,0)=\max \left( 0, h_0 + \eta(x,y,0) -b(x,y) \right),\\
&u(x,y,0)= \eta(x,y,0) \sqrt{\frac{g}{h_0}} \\
&v=0,
\end{aligned}
\end{equation}
where the parameters are set to $A=0.064m$, $x_c = 2.5m$, $\gamma = \sqrt{\frac{3A}{4h_0}}$. The bottom topography is a cone and defined by
\begin{equation}
b(x,y) = 0.93 \left(1-\frac{r}{r_c}\right)	\qquad \text{ if } r\leq r_c
\end{equation}
with $r=\sqrt{(x-x_c)^2 + (y-y_c)}$,  $r_c = 3.6m$ and $(x_c,y_c) = (12.5,15)$.  The domain $\Omega$ is bounded by solid walls everywhere. The test is run up to a final time of $T=50$. We use a uniform Cartesian mesh with varying spatial resolutions and show the results in Figure \ref{fig:WaveRunup}. The base viscosity parameter is set to be $\epsilon_0=0.1$.
\begin{figure}[!ht]
   \centering
    \subfloat[Total water height $H$, $50\times 50$ Elements]
    {
        \includegraphics[width=0.5\textwidth]{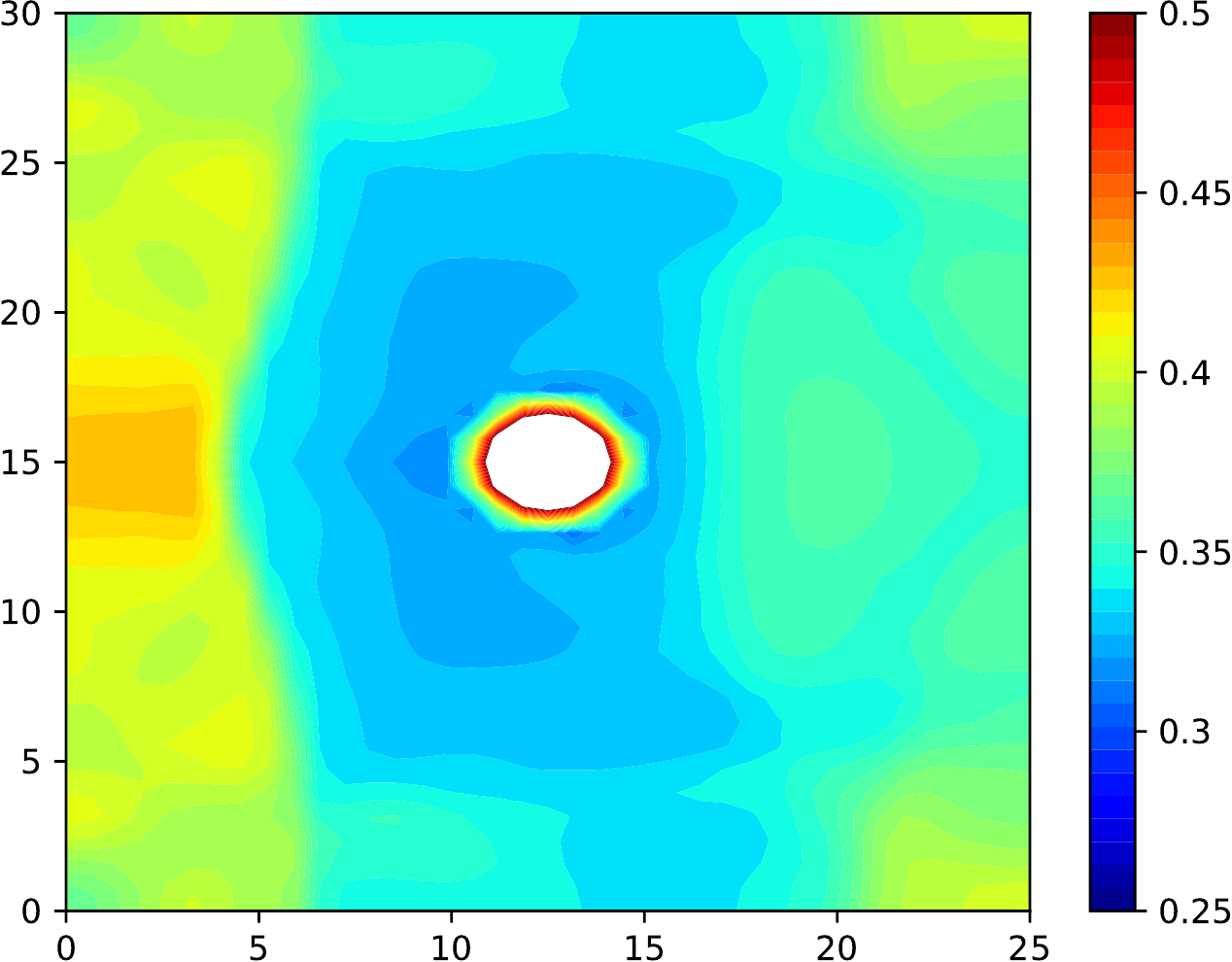}
    }
    \subfloat[Total water height $H$, $200\times 200$ Elements]
    {
        \includegraphics[width=0.5\textwidth]{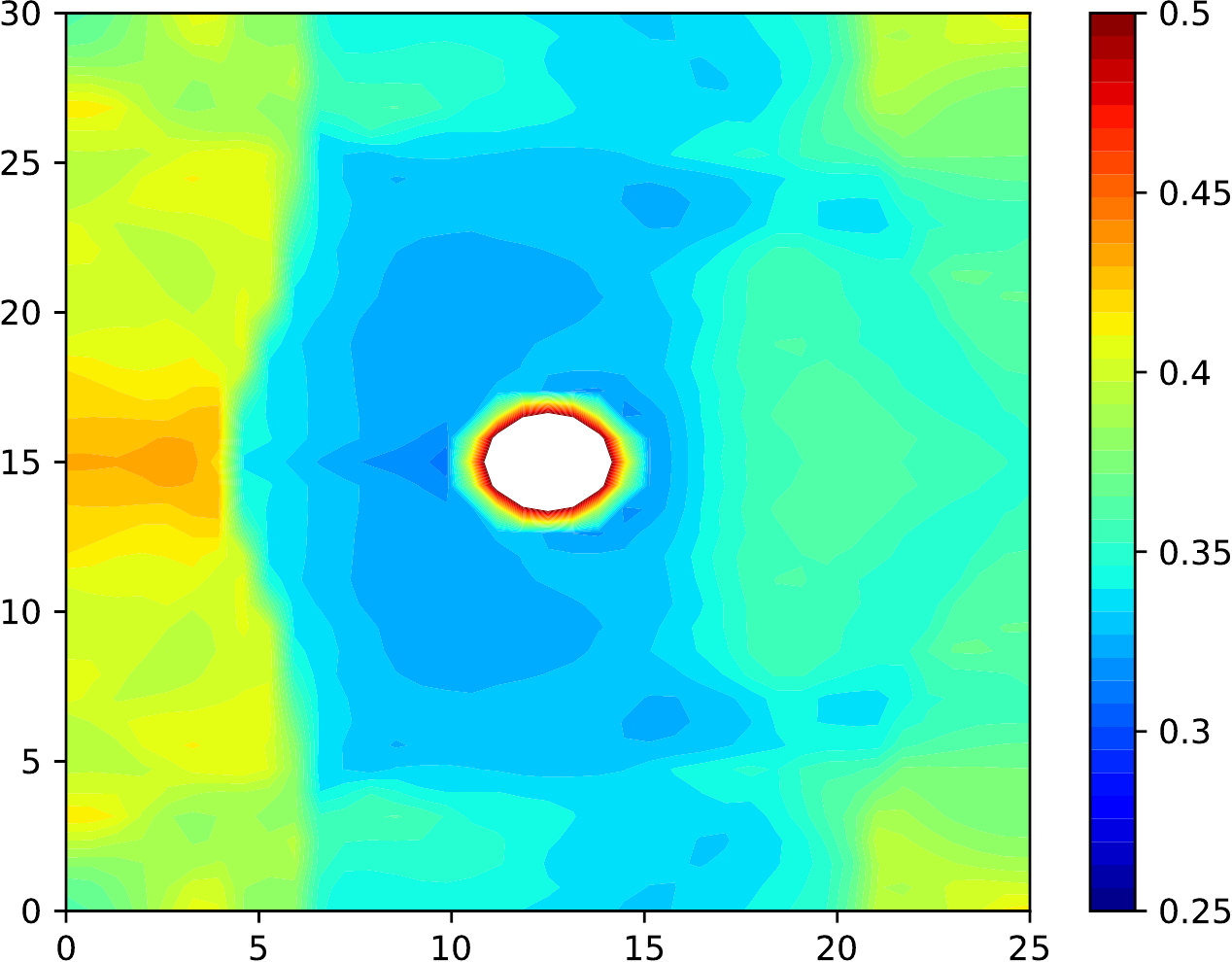}
    }
        \\
    \subfloat[Viscous coefficient $\epsilon$, $50\times 50$ Elements]
    {
        \includegraphics[width=0.5\textwidth]{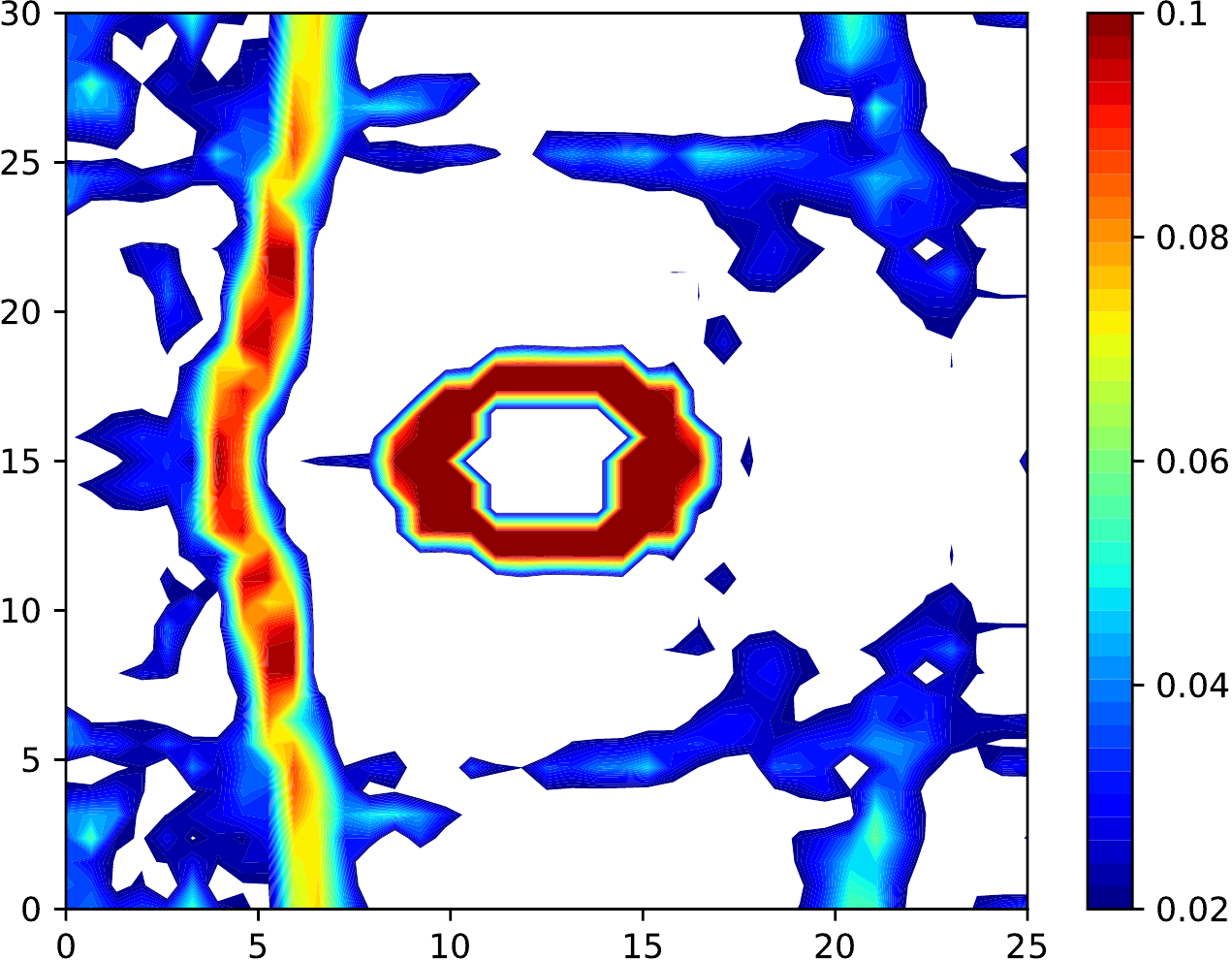}
    }
    \subfloat[Viscous coefficient $\epsilon$, $200\times 200$ Elements]
    {
        \includegraphics[width=0.5\textwidth]{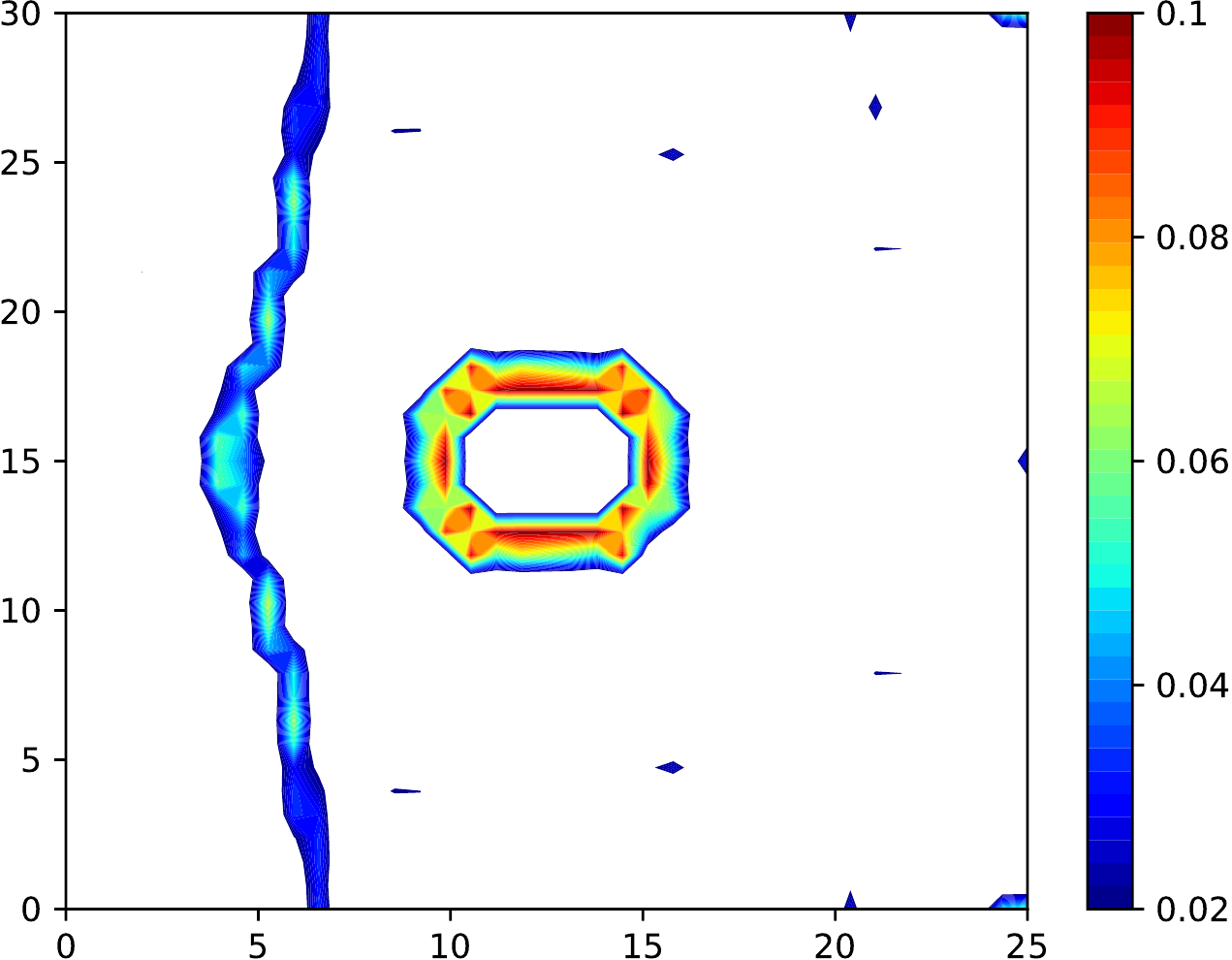}
    }
    \\
    \subfloat[$H$-Slice at $y=15$, $50\times 50$ Elements]
    {
        \includegraphics[width=0.5\textwidth]{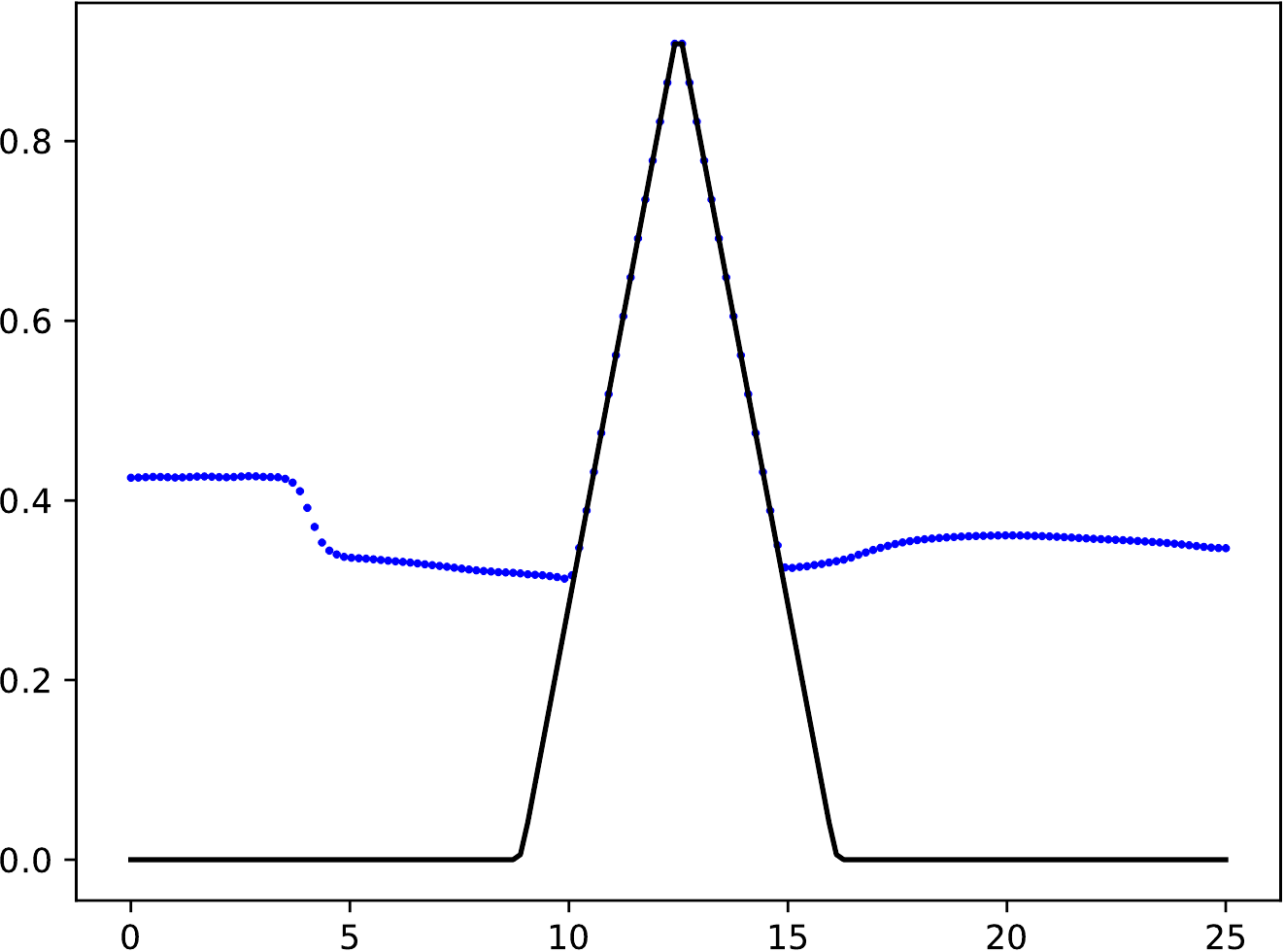}
    }
    \subfloat[$H$-Slice at $y=15$, $200\times 200$ Elements]
    {
        \includegraphics[width=0.5\textwidth]{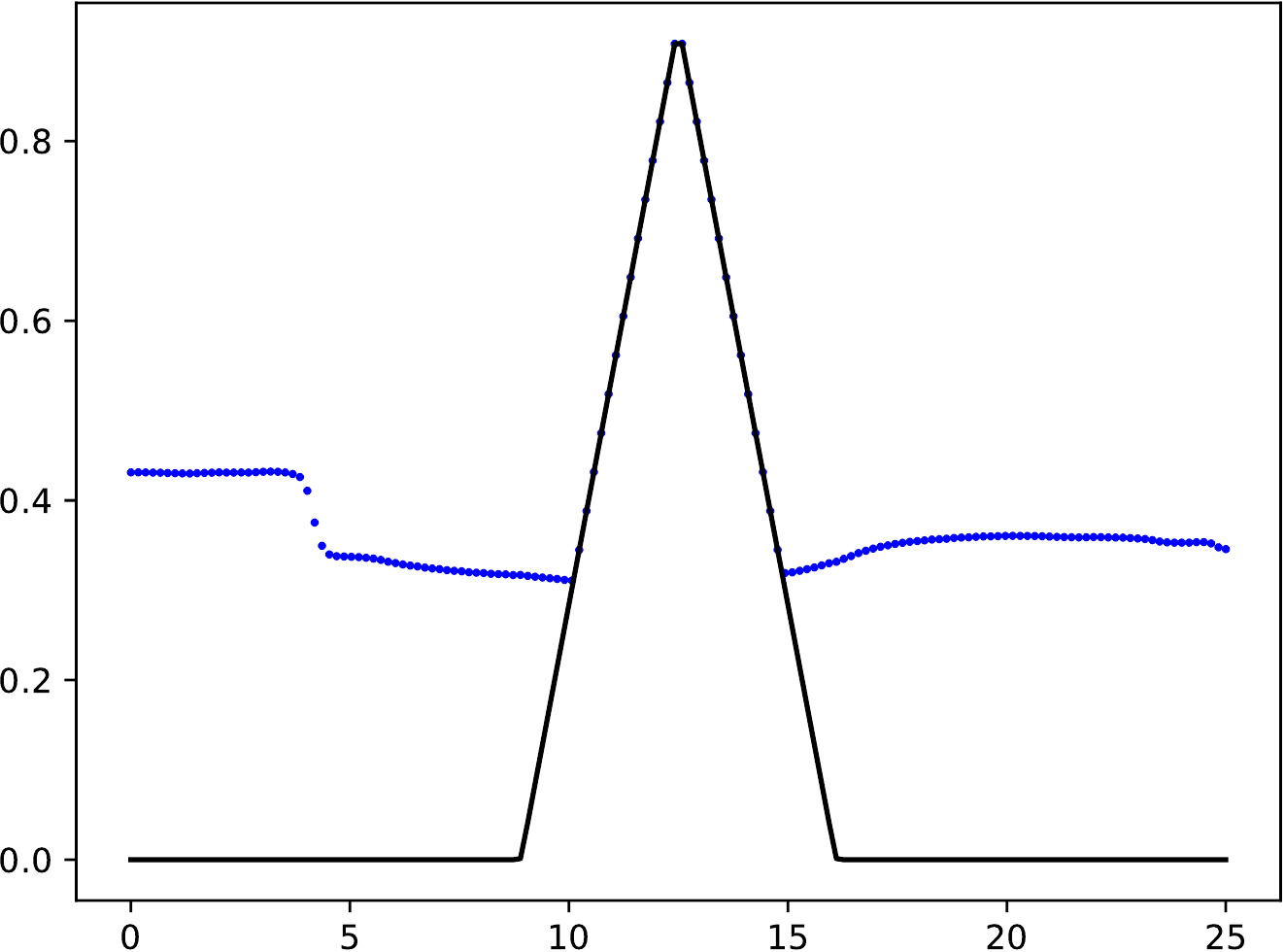}
    }
    \caption{ESDGSEM approximation with artificial viscosity and positivity limiter for the solitary wave runup for different grid resolutions with $N=3$ at $T=50$.}
    \label{fig:WaveRunup}
\end{figure}

\FloatBarrier

\subsection{Parabolic Partial Dam Break}
\label{subsec:WetParaDam}
In \cite{ESDGSEM2D_paper} the authors examined the ESDGSEM on curved meshes with a parabolic partial dam break test case. The mesh is shaped such that it aligns with the parabolic dam. We show the initial condition and the mesh in Figure \ref{fig:CurvedDamMesh}.
\begin{figure}[!ht]
   \centering
        \includegraphics[width=0.5\textwidth]{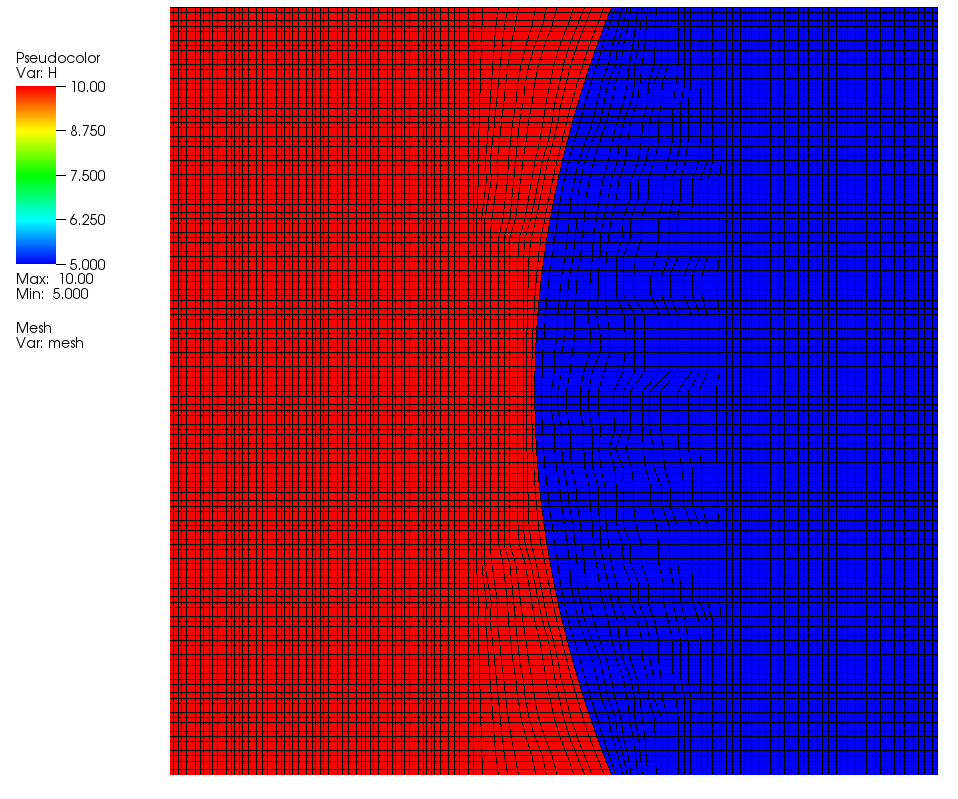}
    \caption{Initial condition and mesh for the parabolic dam break test case.}
    \label{fig:CurvedDamMesh}
\end{figure}While the results showed increased stability for the entropy stable scheme compared to a standard DGSEM, it also showed that the method suffered under oscillations in the shock region. We repeat this test with the additional dynamic artifical viscosity. The initial setup is given by
\begin{equation}
\begin{aligned}
&h= \left\{
\begin{aligned}
&10, \quad&\textrm{if } x< \frac{1}{25}y^2 - 0.25   \\
&5, \quad&\textrm{otherwise}
\end{aligned}
\right.,\\
&u=v=0.
\end{aligned}
\end{equation}
For simplicity, the gravitational constant is set to $g=1.0$ here. We also remove the discontinuous bottom topography from \cite{ESDGSEM2D_paper}. While the scheme still works for discontinuous bottoms, the multitude of different effects makes it hard to observe the impact of the artificial viscosity alone. We compare the results for $N=3$ and $N=7$ with and without added stabilization and show the the calculated dynamic viscosity coefficients in Figure \ref{fig:CurvedDamWet}. The base viscosity coefficient is set to $\epsilon_0=0.025$.
The stabilizing impact of the artificial viscosity is clearly visible. Oscillations have dramatically reduced at the shock front and also at the waves on the top side of the dam. The overshoot spikes close to the center of the dam break are significantly smaller. From the dynamic viscosity coefficient plots we can see that the shock front as well as the back waves at the top are smoothed by viscosity, whereas other smooth regions are not impacted. It is also visible that the viscosity works more accurately the finer the discretization is. For $N=7$ not only less elements are affected, but the absolute amount of viscosity is also reduced.
\begin{figure}[!ht]
   \centering
    \subfloat[Total water height $H$ without AV, $N=3$]
    {
        \includegraphics[width=0.5\textwidth]{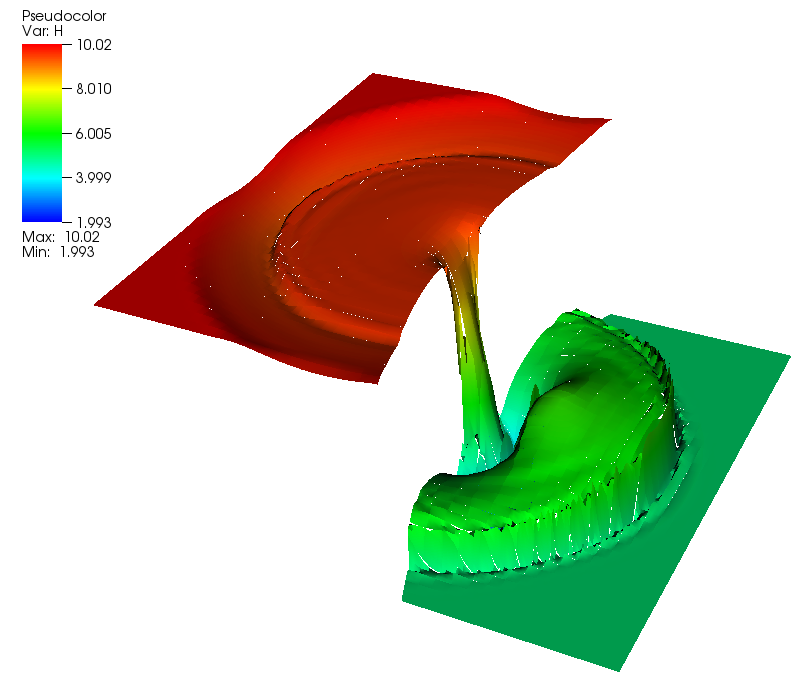}
    }
    \subfloat[Total water height $H$ without AV, $N=7$]
    {
        \includegraphics[width=0.5\textwidth]{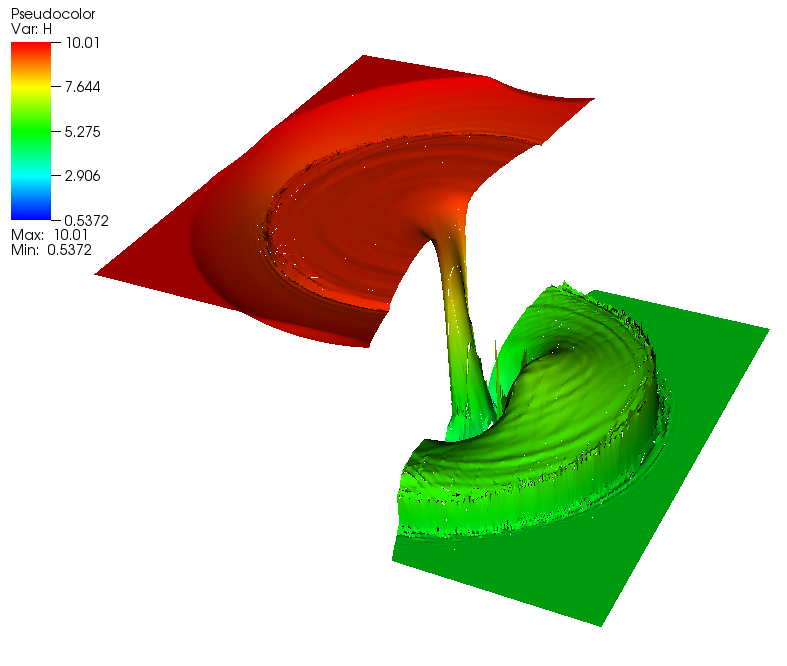}
    }
    \\
        \subfloat[Total water height $H$ with AV, $N=3$]
    {
        \includegraphics[width=0.5\textwidth]{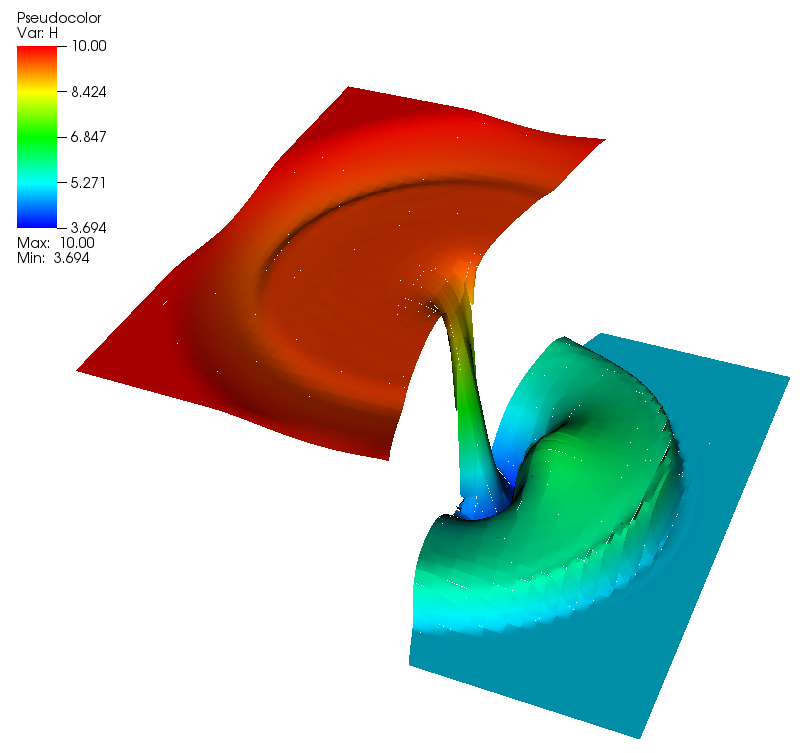}
    }
    \subfloat[Total water height $H$ with AV, $N=7$]
    {
        \includegraphics[width=0.5\textwidth]{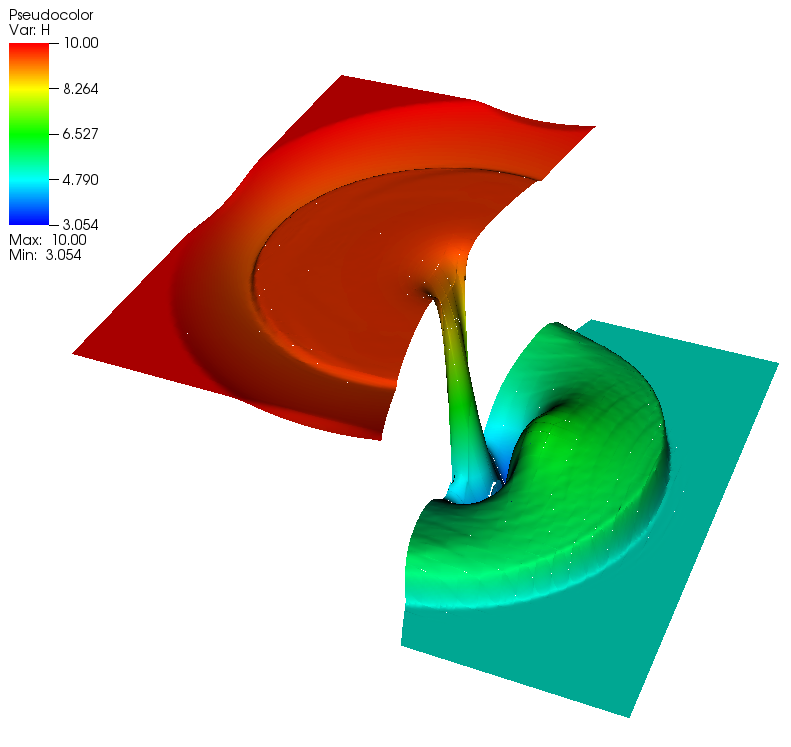}
    }
    \\
    \subfloat[Viscous coefficient $\epsilon$, $N=3$]
    {
        \includegraphics[width=0.5\textwidth]{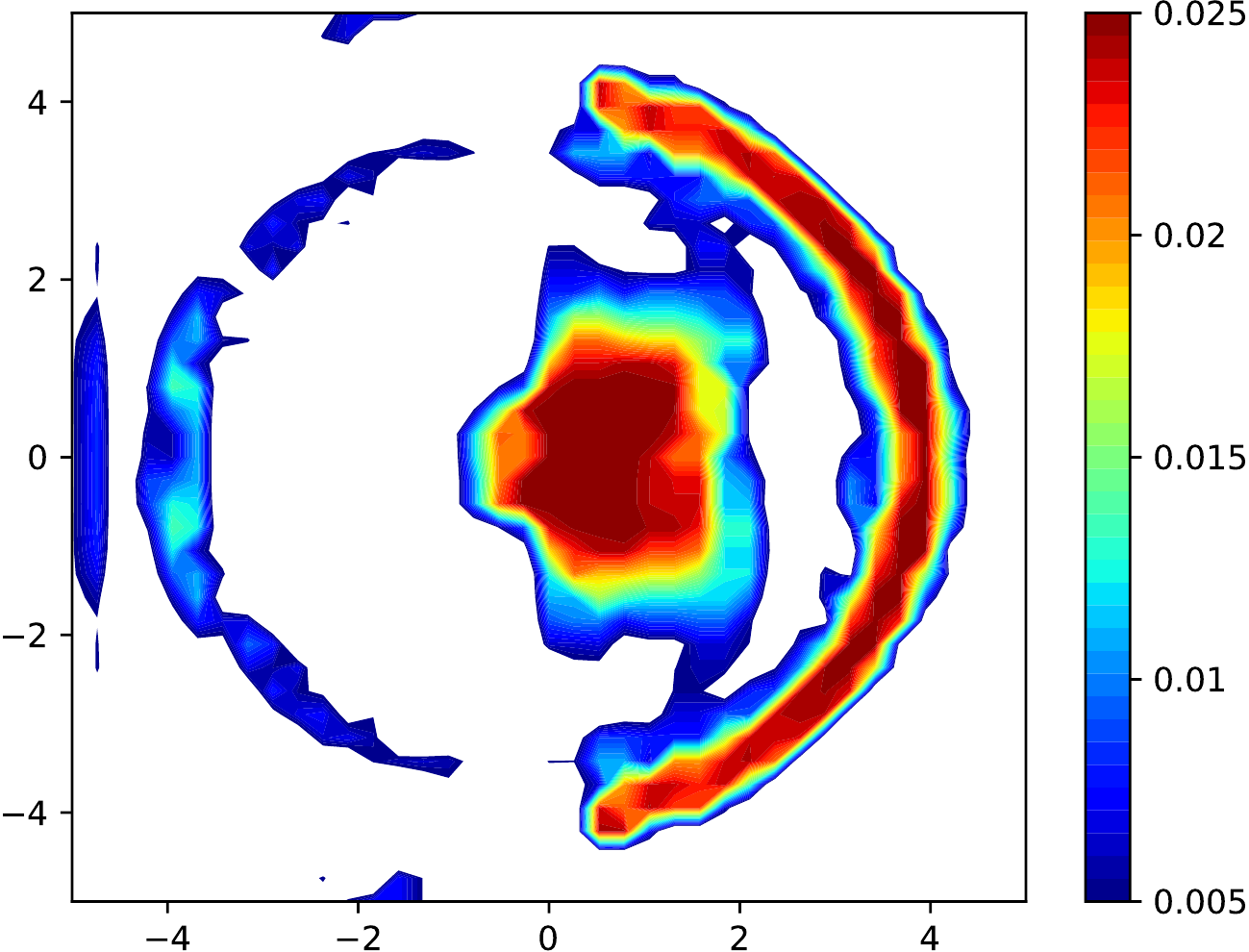}
    }
        \subfloat[Viscous coefficient $\epsilon$, $N=7$]
    {
        \includegraphics[width=0.5\textwidth]{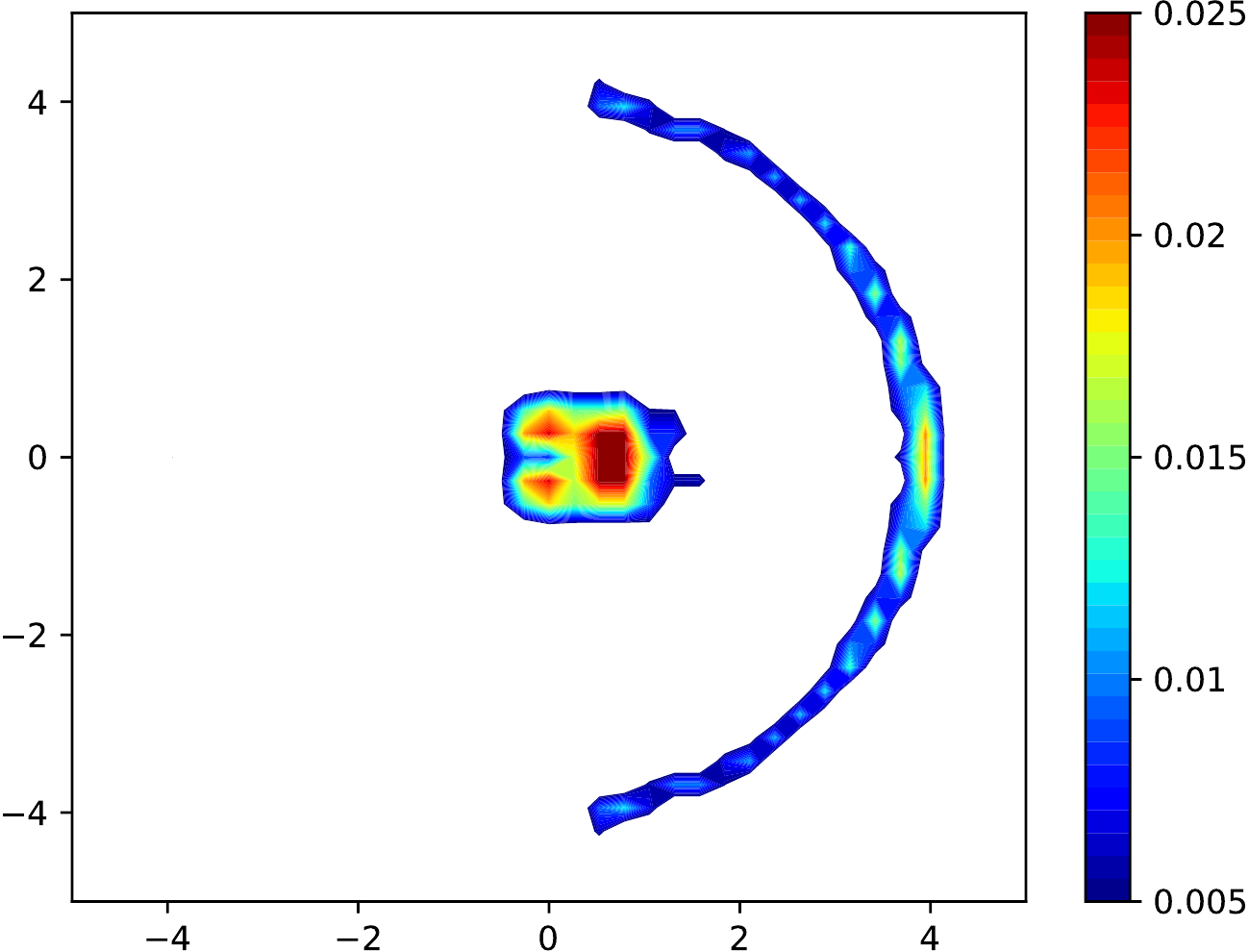}
    }
    \\
    \caption{ESDGSEM approximation for the curved dam break with and without artificial viscosity (AV) at $N=3$ and $N=7$ on a $40\times 40$ curved mesh at $T=1.5$.}
    \label{fig:CurvedDamWet}
\end{figure}

\FloatBarrier

\subsection{Wet/Dry Parabolic Partial Dam Break}
We modify the parabolic dam break problem to feature a dry area on the downstream side. We use the same mesh from Figure \ref{fig:CurvedDamMesh}. It is very challenging as it is a massive shock and thus requires the artificial viscosity as well as the positivity preserving limiter. The initial setup is given by
\begin{equation}
\begin{aligned}
&h= \left\{
\begin{aligned}
&10, \quad&\textrm{if } x< \frac{1}{25}y^2 - 0.25  \\
&0, \quad&\textrm{otherwise}
\end{aligned}
\right.,\\
&u=v=0.
\end{aligned}
\end{equation}
The gravitational constant is set $g=1.0$ again and the viscosity parameter is set to $\epsilon_0 = 0.05$ for $N=3$ and to $\epsilon_0 = 0.025$ for $N=7$. We plot the solution as well as the dynamic viscosity parameter for $N=3$ in Figure \ref{fig:CurvedDamDryN3} and for $N=7$ in Figure \ref{fig:CurvedDamDryN7}. As the dam break is steeper than in the completely wet case from Section \ref{subsec:WetParaDam}, the final time is set to $T=1.0$ so the water does not hit the back wall. The viscous parameter plots again show that only the critical regions are treated with the added artificial viscosity.

\begin{figure}[!ht]
   \centering
    \subfloat[Water height $H$, $T=0.5$]
    {
        \includegraphics[width=0.5\textwidth]{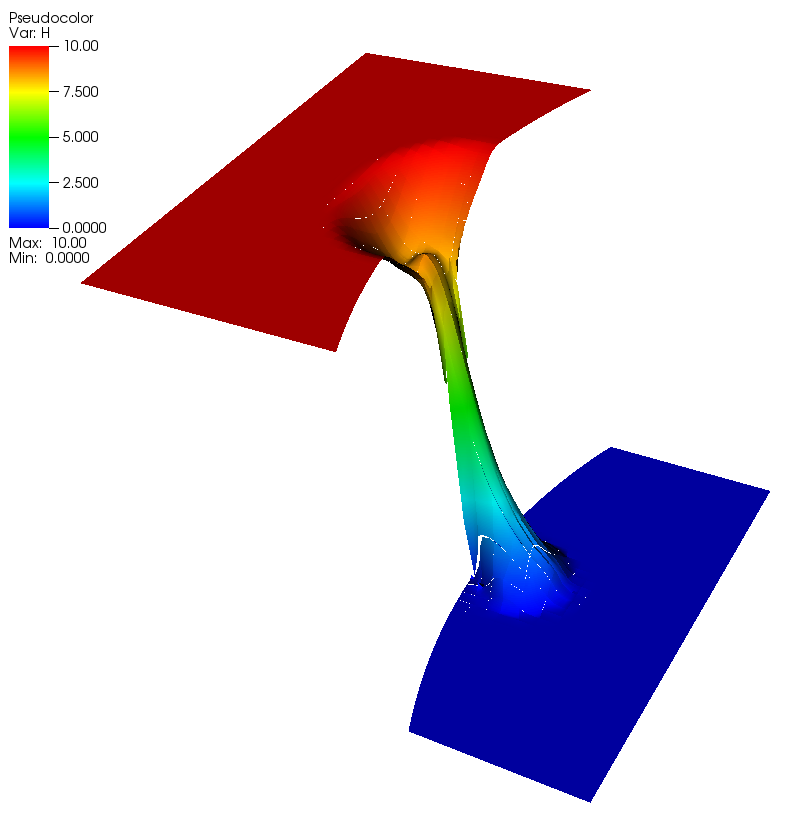}
    }
    \subfloat[Water height $H$, $T=1.0$]
    {
        \includegraphics[width=0.5\textwidth]{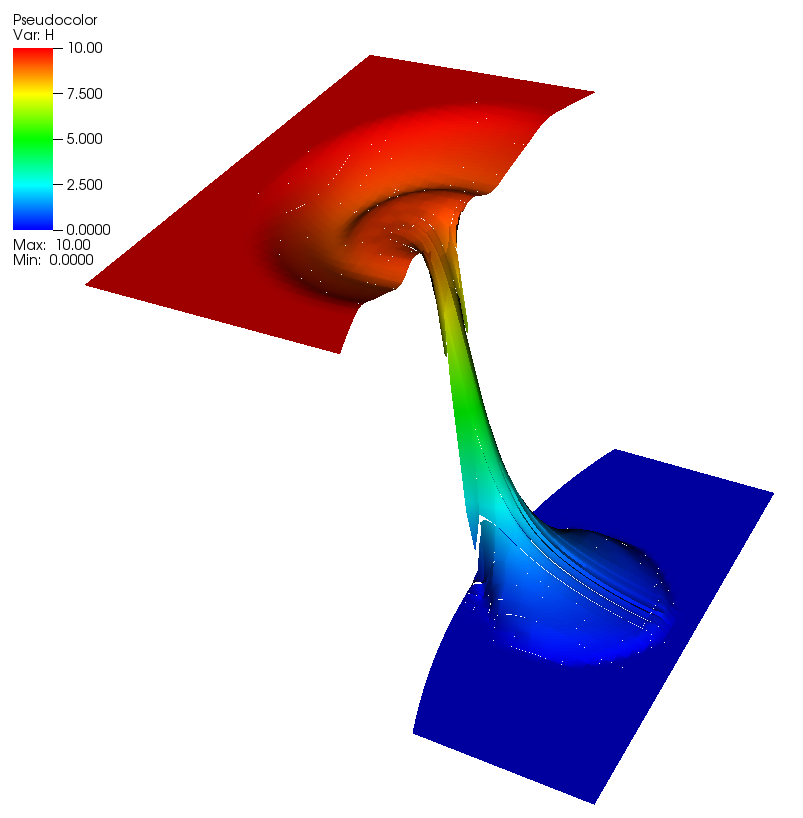}
    }
    \\
    \subfloat[Viscous coefficient $\epsilon$, $T=0.5$]
    {
        \includegraphics[width=0.5\textwidth]{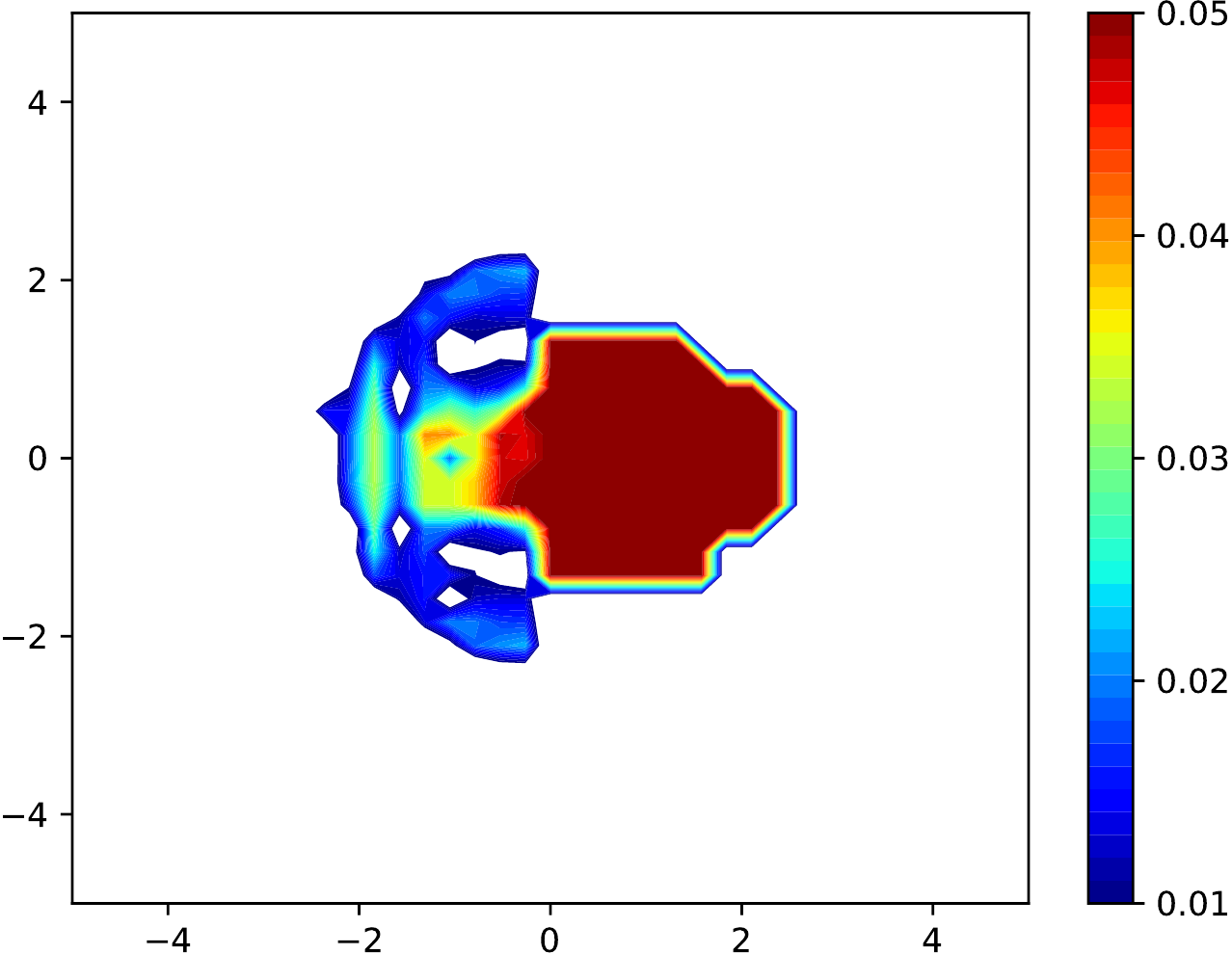}
    }
    \subfloat[Viscous coefficient $\epsilon$, $T=1.0$]
    {
        \includegraphics[width=0.5\textwidth]{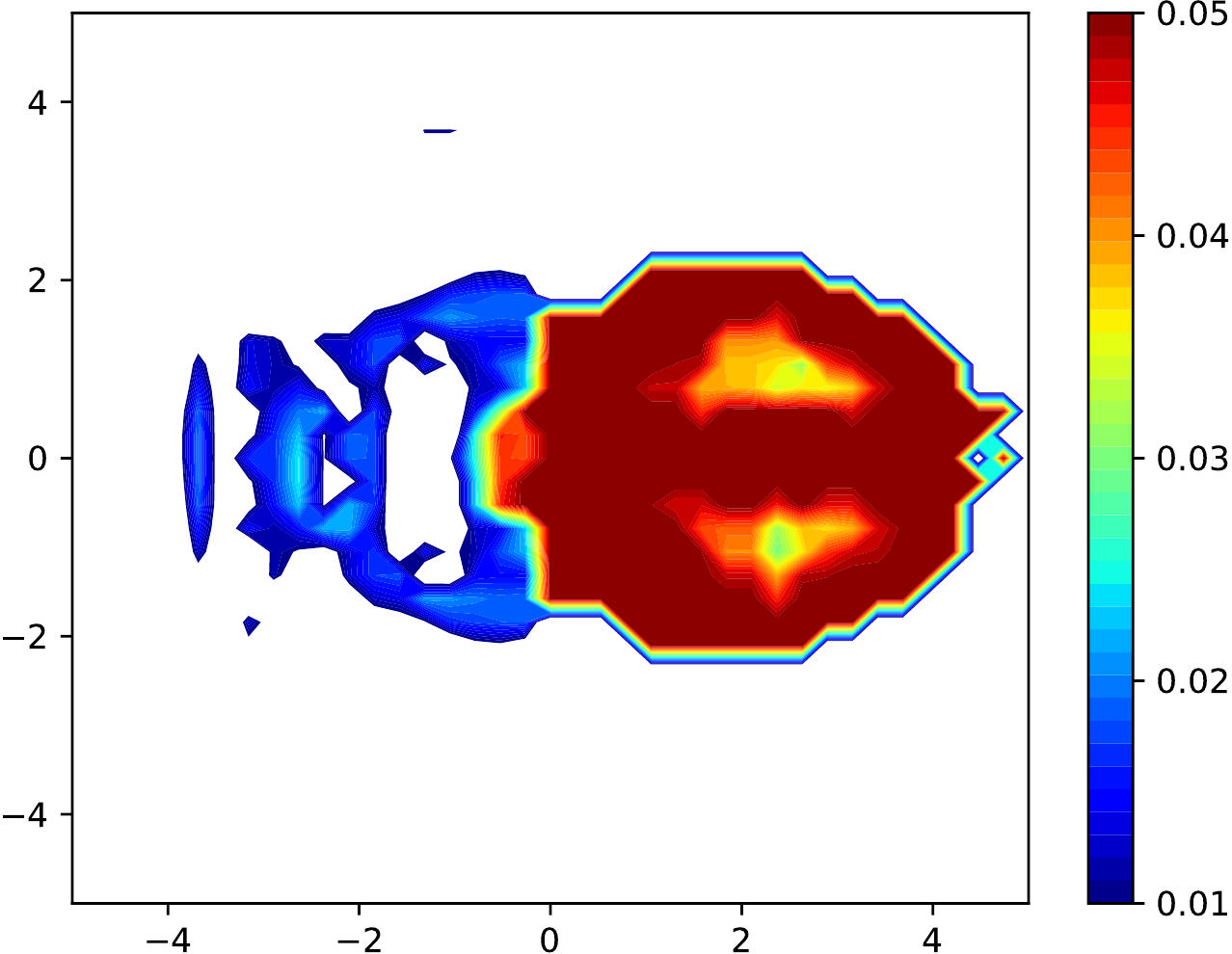}
    }
    \caption{ESDGSEM approximation with artificial viscosity for the curved dam break with zero water height on the downstream side at $N=3$ on a $40\times 40$ curved mesh.}
    \label{fig:CurvedDamDryN3}
\end{figure}
\begin{figure}[!ht]
   \centering
    \subfloat[Water height $H$, $T=0.5$]
    {
        \includegraphics[width=0.5\textwidth]{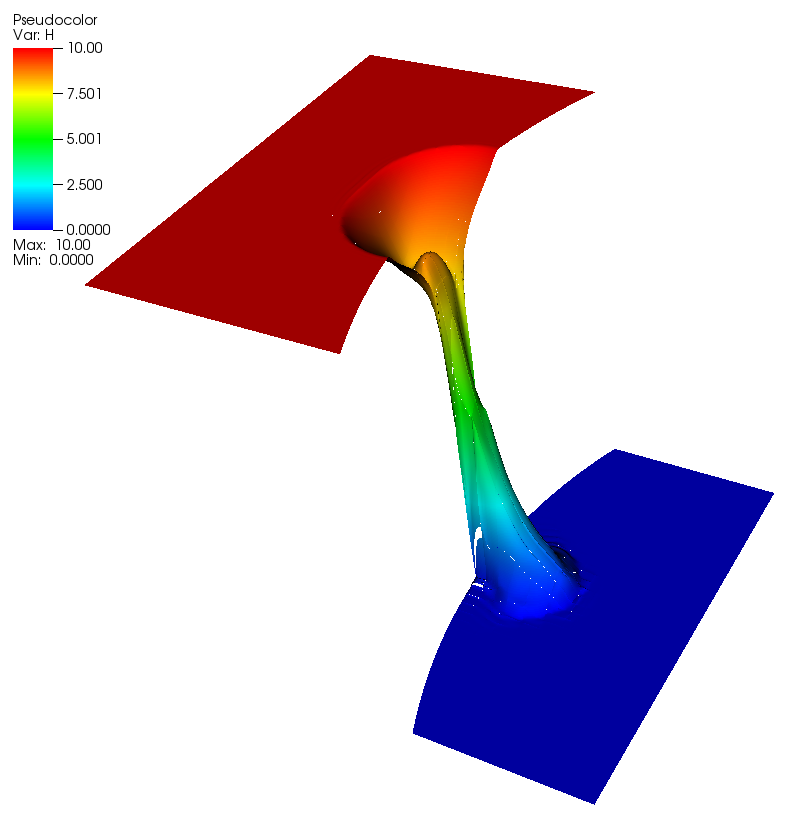}
    }
    \subfloat[Water height $H$, $T=1.0$]
    {
        \includegraphics[width=0.5\textwidth]{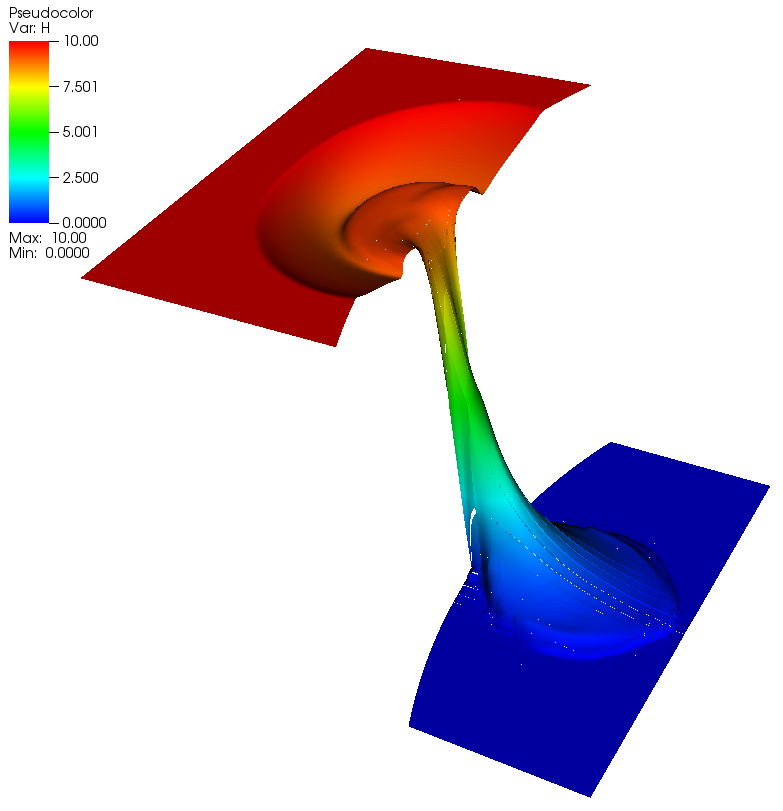}
    }
        \\
    \subfloat[Viscous coefficient $\epsilon$, $T=0.5$]
    {
        \includegraphics[width=0.5\textwidth]{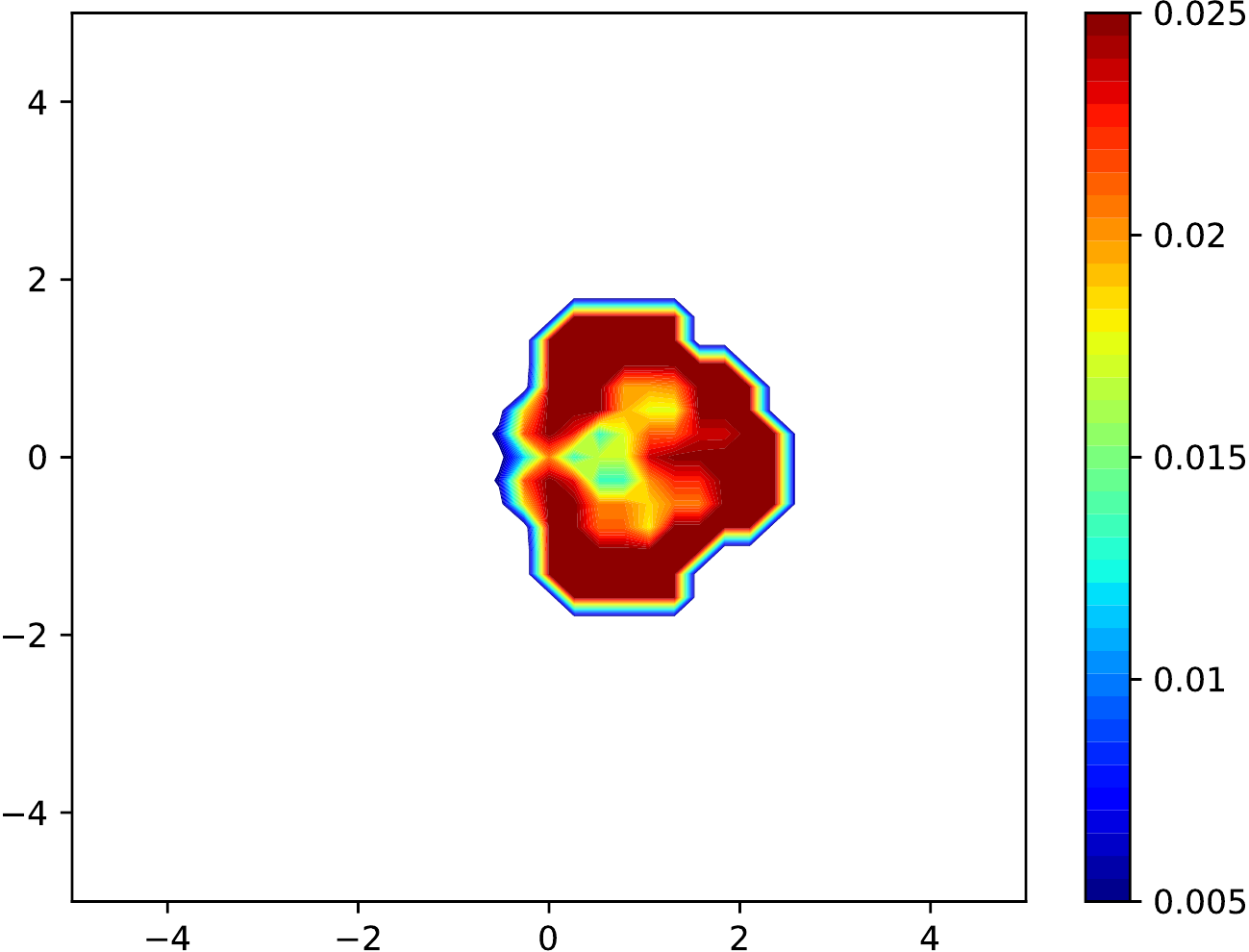}
    }
    \subfloat[Viscous coefficient $\epsilon$, $T=1.0$]
    {
        \includegraphics[width=0.5\textwidth]{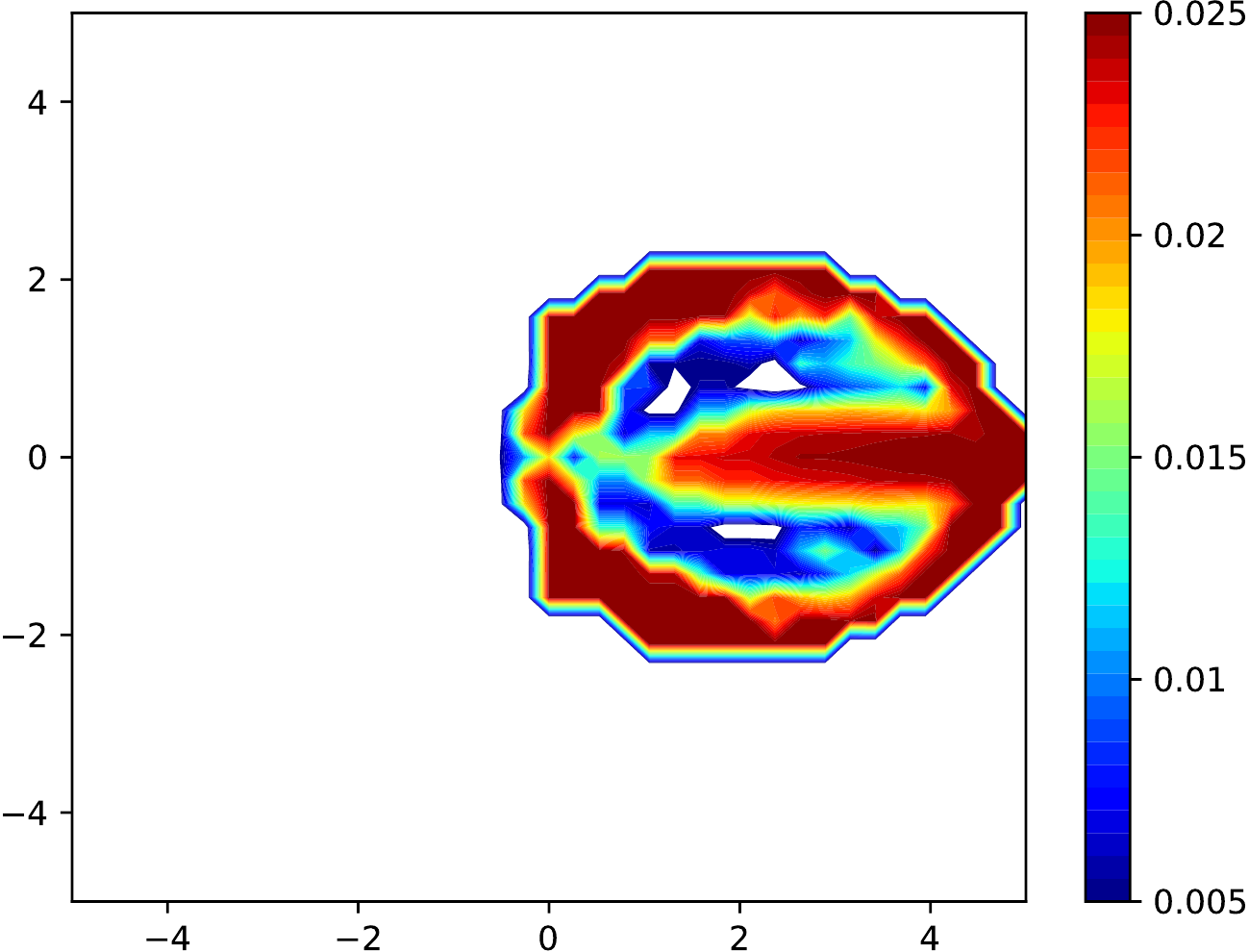}
    }
    \caption{ESDGSEM approximation with artificial viscosity for the curved dam break with zero water height on the downstream side at $N=7$ on a $40\times 40$ curved mesh.}
    \label{fig:CurvedDamDryN7}
\end{figure}
\FloatBarrier
\section{Conclusions}

In this work we extended the entropy stable discontinuous Galerkin (DG) spectral element approximation of Wintermeyer et al. \cite{ESDGSEM2D_paper} to include shock capturing capabilities as well as positivity preservation of the water height such that the numerical scheme can handle wet/dry regions. We demonstrated that these new features, necessary for applications in, e.g., oceanography, did not alter the entropy stable nature of the approximation. Further, we demonstrated that the entropy stable DG discretization for the shallow water equations is well suited for simulations on GPUs. In fact, we found that for polynomial orders of $N\leq 7$ the two methods remained memory bound on GPUs and had nearly identical runtimes. 

We then verified the properties of the scheme numerically. Specifically, we found that the entropy stable DG approximation remained conservative and entropy stable even with the additional shock capturing and positivity preserving methods. We also demonstrated that a numerical method which takes the entropy into account is useful to avoid unphysical solutions with an ``entropy glitch'' test case. Next, we provided five numerical examples to show the utility of the entropy stable, shock capturing, positive water height preserving DG method for problems that feature, among other things, smooth solutions with wet/dry regions, complex multi-shock interactions with bottom topographies or curvilinear element meshes.

\section*{Acknowledgements}
Gregor Gassner thanks the European Research Council for funding through the ERC Starting Grant ``An Exascale aware and Un-crashable Space-Time-Adaptive Discontinuous Spectral Element Solver for Non- Linear Conservation Laws" (Extreme), ERC grant agreement no. 714487. 

\bibliographystyle{plain}
\bibliography{References}

\appendix
\section{The viscous parameter $\epsilon$}
\label{app:ViscPara}
The viscous parameter $\epsilon$ in \eqref{sw-visc} is chosen dynamically for each element dependent on the smoothness of the solution. To get an estimate for the smoothness we transform our nodal DG solution $Q$ to modal space $\hat{Q}$ by
\begin{equation}
    \label{eq:TransformationToModal}
    \hat{Q}_{ij} = \sum_{i=0}^N\sum_{j=0}^N V^{-1}_{ij} Q_{ij} V^{-1}_{ji},
\end{equation}
with Vandermonde matrix $\mat{V}$ defined by
\begin{equation}
\begin{aligned}
    \label{eq:Vandermonde}
    V_{ij} &=  L_j (\xi^{GL}_i) \sqrt{j+0.5}, \\
    V^{-1}_{ij} &= \left(\ell_j, \tilde{L}_i\right)_{L^2} \approx \sum_{l=0}^N  L_i (\xi^G_l)\,   \ell_j^{GL}(x^G_i) \, \omega^G_l \,\sqrt{i+0.5},
\end{aligned}
\end{equation}
where $L_i$ is the $i$-the Legendre polynomial, $\ell^{GL}_i$ the $i$-th Lagrange polynomial based on Legendre-Gauss-Lobatto nodes and $\xi_i^G$ the Legendre-Gauss nodes. The $\omega^G$ are the Legendre-Gauss quadrature nodes.
We compute shock indicators similar to \cite{persson2006} by
\begin{equation}
\label{eq:sigmadof}
    \sigma_{dof} = \log_{10} \left( \max \left( \frac{(Q-\tilde{Q},Q-\tilde{Q})_{L^2}}{(Q,Q)_{L^2}} , \frac{(\tilde{Q}-\tilde{\tilde{Q}},\tilde{Q}-\tilde{\tilde{Q}})_{L^2}}{(\tilde{Q},\tilde{Q})_{L^2}} \right)\right),
\end{equation}
with
\begin{equation}
\begin{aligned}
    \tilde{Q} &:= \sum_{i,j=0}^{N-1} \hat{Q}_{ij} \tilde{L}_i \tilde{L}_j, \\
    \tilde{\tilde{Q}} &:= \sum_{i,j=0}^{N-2} \hat{Q}_{ij} \tilde{L}_i \tilde{L}_j.
\end{aligned}
\end{equation}
With these definitions \eqref{eq:sigmadof} can be simplified to
\begin{equation}
    \sigma_{dof} = \log_{10} \left( 
    \max \left( 
    \frac{
        \sum_{i=0}^{N-1} \left(\hat{Q}^2_{iN}
        + \hat{Q}^2_{Ni}\right)
        + \hat{Q}^2_{NN}
    }{
        \sum_{i,j=0}^{N} \hat{Q}^2_{ij}
    } 
    , 
    \frac{
        \sum_{i=0}^{N-2} \left(\hat{Q}^2_{i(N-1)}
        + \hat{Q}^2_{(N-1)i}\right)
        + \hat{Q}^2_{(N-1)(N-1)}
    }{
        \sum_{i,j=0}^{N-1} \hat{Q}^2_{ij}
    } 
    \right)
    \right).
\end{equation}
This $\sigma_{dof}$ is used to determine the amount of viscosity applied in every element individually by
\begin{equation}
    \epsilon= \left\{
\begin{aligned}
&0, \quad&\textrm{if } \sigma_{dof} \leq \sigma_{min},   \\
&\half \epsilon_0 \Delta, \quad&\textrm{if } \sigma_{dof} \leq \sigma_{min}   \\
&\epsilon_0, \quad&\textrm{else}. 
\end{aligned}
\right.,\\
\end{equation}
and
\begin{equation}
    \Delta := 1.0 + \sin\left(\frac{\pi (\sigma_{dof} - \half (\sigma_{max} + \sigma_{min})}{\sigma_{max} - \sigma_{min}} \right).
\end{equation}

\section{Simplification of entropy stable normal numerical $h$ flux}\label{app:SimpleNumFlux}
We aim to find a compact expression for the first entry of $\tilde{F}^{*,es}$ defined by
\begin{equation}
\label{fStabilized2d}
\tilde{F}^{*,es} =\vec{F}^{*,es}\cdot\vec{n}.
\end{equation}
with
\begin{equation}
\begin{aligned}
\label{fesandges}
F^{*,es} &=F^{*,ec} - \half \mat{R}_f  \, \big|\mat{\Lambda} \big| \, \mat{R}_f^T \jump{\,\statevec{q}\,},\\
G^{*,es} &=G^{*,ec} - \half \mat{R}_g  \, \big|\mat{\Lambda} \big| \, \mat{R}_g^T \jump{\,\statevec{q}\,}.\\
\end{aligned}
\end{equation}
The shallow water equations are rotationally invariant and we can compute the numerical flux in normal direction by rotating the velocities into the new coordinate system and then evaluating the numerical flux in $x$-direction, $F^{*,es}$, with the rotated velocities
\begin{equation}
    \label{eq:rotatingVelocities}
    \begin{aligned}
        \tilde{u} &:= n_x u + n_y v, \\
        \tilde{v} &:= t_x u + t_y v = -n_y u + n_x v.
    \end{aligned}
\end{equation}
After computing the $x$-direction numerical flux we then rotate back into the original coordinate system. Taking everything into account we obtain the following formulas for the numerical fluxes in normal direction
\begin{equation}\label{eq:surfaceFlux2DRotated}
\begin{aligned}
\tilde{F}_1^{*,es}(W^{+},W^{-},\statevec{n}) &=F_1^{*,es}(\tilde{W}^{+},\tilde{W}^{-}) \\
\tilde{F}_2^{*,es}(W^{+},W^{-},\statevec{n}) 
&=n_x F_2^{*,es}(\tilde{W}^{+},\tilde{W}^{-}) 
+t_x F_3^{*,es}(\tilde{W}^{+},\tilde{W}^{-})
\\
\tilde{F}_3^{*,es}(W^{+},W^{-},\statevec{n}) 
&=n_y F_2^{*,es}(\tilde{W}^{+},\tilde{W}^{-}) 
+t_y F_3^{*,es}(\tilde{W}^{+},\tilde{W}^{-}).
\end{aligned}
\end{equation}
The numerical flux in $x$-direction is computed using rotated velocities by
\begin{equation}\label{eq:surfaceFlux2DCurved}
\begin{aligned}
F^{*,es}(\tilde W^{+},\tilde W^{-})= \begin{pmatrix}
\average{h}\average{\tilde{u}}\\[0.1cm]
\average{h}\average{\tilde{u}}^2 + \frac{1}{2}\,g\,\average{h^2}\\[0.1cm]
\average{h}\average{\tilde{u}}\average{\tilde{v}}\\[0.1cm]
\end{pmatrix}- \half \mat{\tilde{R}}  \, \big|\mat{\tilde{\Lambda}} \big| \, \mat{\tilde{R}}^T \jump{\,\tilde{q}\,},
\end{aligned}
\end{equation}
with matrix of right eigenvectors
\begin{equation}
\label{RightEigenvectors2DCurved}
\mat{\tilde{R}} =\begin{pmatrix}
1 & 0 & 1\\
\average{\tilde{u}}+\average{c} & 0 & \average{\tilde{u}}-\average{c} \\
\average{\tilde{v}} & 1 & \average{\tilde{v}}\\
\end{pmatrix},
\end{equation}
and scaled diagonal eigenvalue matrix
\begin{equation}
\label{ScalingMatrix2DCurved}
\big|\mat{\tilde{\Lambda}} \big|  =\frac{1}{2g}\begin{pmatrix}
\big|\average{\tilde{u}}+\average{c} \big|& 0 & 0\\
 0 & 2g\big|\average{h}\average{\tilde{u}}\big| & 0 \\
0 & 0 & \big|\average{\tilde{u}}-\average{c}\big|
\end{pmatrix}.
\end{equation}
We compute the first row of the matrix product of the dissipation term by multiplying the first row $\mat{\tilde{R}}_1$
\begin{equation}
\begin{aligned}
& 2g \mat{\tilde{R}}_1  \, \big|\mat{\tilde{\Lambda}} \big| \, \mat{\tilde{R}}^T\\ 
&=
\begin{pmatrix}
1, & 0, & 1
\end{pmatrix}
\begin{pmatrix}
\big|\average{\tilde u}+\average{c} \big|& 0 & 0\\
 0 & 2g\big|\average{h}\average{\tilde u}\big| & 0 \\
0 & 0 & \big|\average{\tilde u}-\average{c}\big|
\end{pmatrix}
\begin{pmatrix}
1 & \average{\tilde u}+\average{c} & \average{\tilde v}\\
0 & 0 & 1 \\
1 & \average{\tilde u}-\average{c} & \average{\tilde v} \\
\end{pmatrix} \\[0.15cm]
&=
\begin{pmatrix}
\big| \average{\tilde u}+\average{c}\big|,  
& 0 ,
& \big| \average{\tilde u}-\average{c}\big| \\
\end{pmatrix}
\begin{pmatrix}
1 & \average{\tilde u}+\average{c} & \average{\tilde v}\\
0 & 0 & 1 \\
1 & \average{\tilde u}-\average{c} & \average{\tilde v} \\
\end{pmatrix} 
\\[0.15cm]
&=
\begin{pmatrix}
A,
& \average{\tilde u} A+\average{c}B,
& \average{\tilde v} A
\end{pmatrix},
\end{aligned}
\end{equation}
with 
\begin{equation}
\begin{aligned}
A &:= \big|\average{\tilde u}+\average{c}\big| + \big|\average{\tilde u}-\average{c}\big|,\\
B &:= \big|\average{\tilde u}+\average{c}\big| - \big|\average{\tilde u}-\average{c}\big|.
\end{aligned}
\end{equation}
Multiplying by $\jump{\,\tilde{q}}$ we find the first entry of $\half \mat{\tilde{R}}  \, \big|\mat{\tilde{\Lambda}} \big| \, \mat{\tilde{R}}^T \jump{\,\tilde{q}}$ 
\begin{equation}
\label{fStabilized_FirstEntry_2D}
\begin{aligned}
 2g\left(\mat{\tilde{R}}  \, \big|\mat{\tilde{\Lambda}} \big| \, \mat{\tilde{R}}^T \jump{\,\tilde{q}}\right)_1
&=
\begin{pmatrix}
A,
& \average{\tilde u} A+\average{c}B,
& \average{\tilde v} A
\end{pmatrix}
\begin{pmatrix}
g\jump{h+b} - \half \jump{\tilde u^2}- \half \jump{\tilde v^2} \\
 \jump{\tilde u}\\
 \jump{\tilde v}\\
\end{pmatrix}\\
&=
A\left(g\jump{h+b} - \average{\tilde u} \jump{\tilde u}- \average{\tilde v}\jump{\tilde v}\right)
+\left(\average{\tilde u} A+\average{c}B\right)\jump{\tilde u}
+ \average{\tilde v} A \jump{\tilde v}\\
&=
gA\jump{h+b}
+\average{c}B\jump{\tilde u}.
\end{aligned}
\end{equation}
We can find bounds on $A$ and $B$ by
\begin{equation}
\begin{aligned}
\label{BoundsAB}
2 \lambda_{\text{max}} = 2 \max \left( \big| \tilde u \big| + \big| c\big|\right)
\geq 2 (\big| \average{\tilde u}\big| + \big| \average{c}\big|) &\geq A  \geq 2 \big| \average{\tilde u}\big| \\
 \lambda_{\text{max}} = \max \left( \big| \tilde u \big| + \big| c\big|\right)\geq  (\big| \average{\tilde u}\big| + \big| \average{c}\big|) &\geq B \geq - \big| \average{\tilde u} -  \average{c}\big| \geq  -  \lambda_{\text{max}} .\\
\end{aligned}
\end{equation}
Overall, we can express the first entry of the entropy stable numerical flux in terms of the rotated velocities
\begin{equation}
\begin{aligned}
F_1^{*,es} 
&=\average{h}\average{\tilde{u}}- \frac{1}{4g} \left( A \jump{gh +gb } +   \average{c}B\jump{\tilde{u}}\right).
\end{aligned}
\end{equation}

\section{Versioning of OCCA Optimized Volume Integral Kernels}
\label{app:Versioning}
\begin{itemize}
\item{Version 1: Minimizing Global Loads / Utilizing Shared Memory}

The first step in improving the kernel performance is to reduce the number of costly loads from global GPU memory. To do this, we load all the necessary data only once. If the value is needed by several nodes of the element, we store it in shared memory, which is fast but limited. If it is only needed by an individual node, and thus only in one thread, we store it in a thread-local register. Also, data from shared memory that is used multiple times is loaded to register memory only once.

\item{Version 2: Declaring variables constant and pointers restricted}

As an additional step to improve data storage and transfer, we declare all variables that do not change their value during the computation as \textbf{constant}. We also set all pointers passed to the kernel as restricted. We store the constant values $\half g$ and $\fourth g$ as kernel infos at the start of the program. We also introduce an additional shared memory array that stores the inverse of the water height $1/h$ which is used in the flux calculations.

\item{Version 3: Multiple Elements per Block}

One GPU thread block is typically able to handle multiple DG elements in parallel. We aim to make full use of the GPU compute power and reduce the number of idle threads. Thus we introduce a parameter $NE_{Block}$ that sets the number of elements handled in one thread block. This number obviously depends on the polynomial order and typically decreases with $N$. Unfortunately changes in this parameter can drastically impact the performance of the kernel, so this is a parameter open to optimization.
\item{Version 4: Optimizing the Loops}

We need to  make sure that memory access is aligned as much as possible. Adjacent threads should access adjacent global device memory. This is ensured by accessing index $i,j$ as $i \times (N+1)+j$ if $j$ is the innermost loop index. Also shared memory is accessed fastest if the innermost loop is over the outermost index. So if the shared variable is accessed as $s\_Q[\text{ieLoc}][i][j]$ the inner loops should be over $\text{ieLoc}$, $i$, $j$ in exactly this order. Also, loop unrolling is added to the serial inner loops.

\item{Version 5: Avoid Bank Conflicts, add Padding}

To avoid bank conflicts we increase the size of the shared memory arrays by one, if $N+1$ is a multiple of $4$. This is done by a variable $nglPad$ which is $1$ or $0$ depending on the polynomial degree and added to the last entry of the shared memory arrays.

\item{Version 6: Split inner loop \& Hide shared memory loads}

We split the inner loop in two and compute the $F$ and $G$ fluxes and their contribution to the volume integral separately. This potentially opens up room for the compiler to optimize register loads and ease register pressure. We also change the order of operations such that variables needed for the update of the time integral such as $J$, $b_x$ or $b_y$ are loaded before the flux derivatives are computed. We hope that this potentially hides loads time behind the flux computations. We also introduce separate variables for the flux derivatives for $F$ and $G$ and then add them together in the end in the update step.

\end{itemize}

\end{document}